\documentclass[12pt,reqno]{amsart}

\usepackage{amsmath,amssymb,ifthen}
\usepackage{fullpage}
\usepackage{graphicx,psfrag,subfigure}
\usepackage{color}
\usepackage{moreverb}
\usepackage{amsthm}
\usepackage{mathrsfs}
\usepackage{esint}  
\usepackage{graphicx}
\usepackage{psfrag}
\usepackage{hhline}
\usepackage{enumerate}
\usepackage{bbm}
\usepackage[export]{adjustbox}

\def\coarse{\bullet}
\def\fine{\circ}
\def\ro#1{q_{\rm #1}}
\def\const#1{C_{\rm #1}}
\def\q#1{q_{\rm #1}}
\def\H{\widetilde{H}}
\def\K{\mathbb{K}}

\def\MM{\mathcal M}

\def\R{{\mathbb R}}
\def\Z{\mathbb {Z}}
\def\N{{\mathbb N}}

\def\KK{{\mathcal K}}
\def\NN{{\mathcal N}}
\def\NNr{{\mathcal N}^{\rm ref}}
\def\NNnr{{\mathcal N}^{\rm id}}

\def\PP{{\mathcal P}}

\def\XX{{\mathcal X}}
\def\VV{{\mathfrak V}}
\def\WW{\mathfrak{W}}
\def\WWe{\mathcal{W}}

\def\diam{{\rm diam}}

\def\edual#1#2{\langle#1\,,\,#2\rangle_{\WW}}
\def\edualV#1#2{\langle#1\,,\,#2\rangle_{\VV}}
\def\norm#1#2{\|#1\|_{#2}}

\def\enorm#1{|\hspace*{-.5mm}|#1|\hspace*{-.5mm}|_{\WW}}
\def\enormV#1{|\hspace*{-.5mm}|#1|\hspace*{-.5mm}|_{\VV}}
\def\set#1#2{\big\{#1\,:\,#2\big\}}

\def\dual#1#2{\langle#1\,,\,#2\rangle}

\def\QQ{{\mathcal{Q}}}

\def\A{\mathbb{A}}

\def\refine{{\tt ref}}
\def\RR{\mathcal{R}}

\def\supp{{\rm supp}}

\def\Cdrel{C_{\rm drel}}
\def\Cred{C_{\rm red}}
\def\Cstab{C_{\rm stab}}

\def\Cref{C_{\rm ref}}

\def\Cqo{C_{\rm qo}}

\def\rhored{\rho_{\rm red}}

\def\Copt{C_{\rm opt}}

\def\Crel{C_{\rm rel}}

\def\Cclos{C_{\rm clos}}

\def\Cscott{C_{\rm sz}}

\def\Cinv{C_{\rm inv}}
\def\Cmin{C_{\rm min}}
\def\Cmark{C_{\rm mark}}

\def\Cconv{C_{\rm conv}}

\newcounter{constantsnumber}
\def\namec#1#2{%
  \ifthenelse{\equal{#1}{lipschitz}}{C_{\rm lip}}{%
  \ifthenelse{\equal{#1}{c:unifEquivLevel}}{C_{\rm level}}{%
  \ifthenelse{\equal{#1}{mark}}{C_{\rm mark}}{%
  \ifthenelse{\equal{#1}{basis}}{C_{\rm basis}}{%
  \ifthenelse{\equal{#1}{monotone}}{C_{\rm mon}}{%
  \ifthenelse{\equal{#1}{cea}}{C_{\mbox{\rm\scriptsize C\'ea}}}{%
  \ifthenelse{\equal{#1}{norm}}{C_{\rm norm}}{%
  \ifthenelse{\equal{#1}{mon}}{{C}_{\rm mon}}{
  \ifthenelse{\equal{#1}{lip}}{{C}_{\rm lip}}{
  \ifthenelse{\equal{#1}{monA}}{c_{\rm mon}}{
  \ifthenelse{\equal{#1}{lipA}}{c_{\rm lip}}{
  \ifthenelse{\equal{#1}{normequiv1}}{c_{\rm norm}}{ 
  \ifthenelse{\equal{#1}{inv}}{C_{\rm inv}}{ 
  \ifthenelse{\equal{#1}{inv2}}{\widetilde{C}_{\rm inv}}{ 
  \ifthenelse{\equal{#2}{newcounter}}{\refstepcounter{constantsnumber}\label{const#1}}{}C_{\ref{const#1}}}%
  }}}}}}}}}}}}}}

\numberwithin{equation}{section}
\numberwithin{figure}{section}
\newtheorem{theorem}{Theorem}[section]
\newtheorem{proposition}[theorem]{Proposition}
\newtheorem{lemma}[theorem]{Lemma}

\newtheorem{algorithm}[theorem]{Algorithm}

\newtheorem{remark}[theorem]{Remark}

\def\refine{{\tt refine}}
\def\MM{\mathcal M}

\begin{document}

\title[Adaptive IGA boundary element methods]{Adaptive isogeometric boundary element methods\\ with local smoothness control}
\date{\today}

\author{Gregor Gantner}
\author{Dirk Praetorius}
\author{Stefan Schimanko}

\begin{abstract}
In the frame of isogeometric analysis, we consider a Galerkin  boundary element discretization of the hyper-singular integral equation associated with the 2D Laplacian. We propose and analyze an adaptive algorithm which locally refines the boundary partition and, moreover, steers the smoothness of the NURBS ansatz functions across elements. In particular and unlike prior work, the algorithm can increase and decrease the local smoothness properties and hence exploits the full potential of isogeometric analysis. We prove that the new adaptive strategy leads to linear convergence with optimal algebraic rates. Numerical experiments confirm the theoretical results.
A short appendix comments on analogous results for the weakly-singular integral equation.
\end{abstract}


\maketitle

\section{Introduction}\label{section:introduction}

\noindent
In this work, we prove optimal convergence rates for an adaptive isogeometric boundary element method for the (first-kind) hyper-singular integral equation 
\begin{align}\label{eq:hyper strong}
 \mathfrak{W}u=g:=(1/2-\mathfrak{K'})\phi
 \quad\text{on }\Gamma:=\partial\Omega
\end{align}
associated with the 2D Laplacian. Here, $\Omega \subset \R^2$ is a bounded Lipschitz domain, whose boundary can be parametrized via non-uniform rational B-splines (NURBS); see Section~\ref{section:preliminaries} for the precise statement of the integral operators $\mathfrak{W}$ and $\mathfrak{K}'$ as well as for definition and properties of NURBS. Given boundary data $\phi$, we seek for the unknown integral density $u$. We note that~\eqref{eq:hyper strong} is equivalent to the Laplace--Neumann problem
\begin{align}
 -\Delta P = 0 \text{ in } \Omega
 \quad \text{subject to Neumann boundary conditions} \quad
 \partial P/\partial \nu = \phi \text{ on } \Gamma,
\end{align}
where $u = P|_\Gamma$ is the trace of the sought potential $P$.

The central idea of isogeometric analysis (IGA) is to use the same ansatz functions for the discretization of \eqref{eq:hyper strong}, as are used for the representation of the problem geometry in CAD.
 This concept, originally invented in~\cite{hughes2005} for finite element methods (IGAFEM) has proved very fruitful in applications; see also the monograph \cite{bible}. 
Since CAD directly provides a parametrization of the boundary $\partial \Omega$, this makes the boundary element method (BEM) the most attractive numerical scheme, if applicable (i.e., provided that the fundamental solution of the differential operator is explicitly known); see \cite{igabem2d,igabem3d} for the first works on isogeometric BEM (IGABEM) for 2D resp.\ 3D. 

We refer {to~\cite{SBTR,helmholtziga,simpson,adss16,tran}} for numerical experiments, to \cite{cad2wave,TM,zechner,dhp16,wolf18,wolf_new} for fast IGABEM based on wavelets, fast multipole, $\mathcal{H}$-matrices resp.\ $\mathcal{H}^2$-matrices, and 
to~\cite{stokesiga,keuchel,acdsss18,acdsss18,fgkss18} for some quadrature analysis.



On the one hand, IGA naturally leads to high-order ansatz functions. On the other hand, however, optimal convergence behavior with higher-order discretizations is only observed in simulations, if the (given) data $\phi$ as well as the (unknown) solution $u$ are smooth. Therefore, {\sl a~posteriori} error estimation and related adaptive strategies are mandatory to realize the full potential of IGA. Rate-optimal adaptive strategies for IGAFEM have been proposed and analyzed independently in~\cite{bg2017,ghp2017} for IGAFEM, while the earlier work~\cite{bg2016} proves only linear convergence. As far as IGABEM is concerned, available results focus on the weakly-singular integral equation with energy space $H^{-1/2}(\Gamma)$; see~\cite{igafaermann,resigabem} for {\sl a~posteriori} error estimation as well as~\cite{optigabem} for the analysis of a rate-optimal adaptive IGABEM in 2D, and~\cite{gantner17} for corresponding results for IGABEM in 3D with hierarchical splines. 
Recently, \cite{priga} investigated optimal preconditioning for IGABEM in 2D with locally refined meshes.

In this work, we consider the hyper-singular integral equation~\eqref{eq:hyper strong} with energy space $H^{1/2}(\Gamma)$. We stress that the latter is more challenging than the weakly-singular case, with respect to numerical analysis as well as stability of numerical simulations. Moreover, the present work addresses also the adaptive steering of the smoothness of the NURBS ansatz spaces across elements. The adaptive strategy thus goes beyond the classical 
\begin{align*}
 \boxed{\rm\, SOLVE\,}
 \,\,\longrightarrow\,\,
 \boxed{\rm\, ESTIMATE\,}
 \,\,\longrightarrow\,\,
 \boxed{\rm\, MARK\,}
 \,\,\longrightarrow\,\,
 \boxed{\rm\, REFINE\,}
\end{align*}
considered, e.g., in~\cite{fkmp,gantumur,part1,part2} for standard BEM with piecewise polynomials. Moreover, while the adaptive algorithm from~\cite{optigabem} only allows for a smoothness reduction (which makes the ansatz space larger), the new algorithm also stears the local increase of smoothness (which makes the ansatz space smaller). 
Additionally, we also account for the approximate computation of the right-hand side.
We prove that the new algorithm is rate optimal in the sense of~\cite{axioms}. Moreover, as a side result, we observe that the related approximation classes are independent of the smoothness of the ansatz functions. 

\def\Ceff{C_{\rm eff}}
To steer the algorithm, we adopt the weighted-residual error estimator from standard BEM~\cite{cs95,carstensen97,cmps04,part2} and prove that it is reliable and weakly efficient,~i.e.,
\begin{subequations}
\begin{align}
 \eta_\coarse := \big( \norm{h_\coarse^{1/2} ( (1/2-\mathfrak{K}')\phi_\coarse - \mathfrak{W}U_\coarse )}{L^2(\Gamma)}^2
 + \norm{h_\coarse^{1/2} ( \phi - \phi_\coarse )}{L^2(\Gamma)}^2 \big)^{1/2}
\end{align}
satisfies (with the arclength derivative $\partial_\Gamma$) that 
\begin{align}
 \Crel^{-1} \, \norm{u - U_\coarse}{\H^{1/2}(\Gamma)}
 \le \eta_\coarse
 \le \Ceff \big( \norm{h_\coarse^{1/2} \partial_\Gamma( u - U_\coarse )}{L^2(\Gamma)}^2
 + \norm{h_\coarse^{1/2} ( \phi - \phi_\coarse )}{L^2(\Gamma)}^2 \big)^{1/2}.
\end{align}
\end{subequations}
Here, $h_\coarse$ is the \emph{local mesh-size}, and $U_\coarse$ is the Galerkin solution with respect to some approximate discrete data $\phi_\coarse \approx \phi$. We compute $\phi_\coarse$ by the $L^2$-orthogonal projection of $\phi$ onto discontinuous piecewise polynomials. We stress that data approximation is an important subject in numerical computations, and reliable numerical algorithms have to properly account for it. In particular, the benefit of our approach is that the implementation has to deal with discrete integral operators only. Since $\phi$ is usually non-smooth with algebraic singularities, the stable numerical evaluation of $\mathfrak{K}'\phi$ would also require non-standard (and problem dependent) quadrature rules, which simultaneously resolve the logarithmic singularity of $\mathfrak{K}'$ as well as the algebraic singularity of $\phi$. 
This is avoided by our approach. 
Finally, in the appendix, we generalize the presented results also to slit problems and the weakly-singular integral equation. 

\subsection*{Outline}
The remainder of the work is organized as follows: 
Section~\ref{section:preliminaries} provides the functional analytic setting of the boundary integral operators, the definition of the mesh, B-splines and NURBS together with their basic properties.
In Section~\ref{sec:main result}, we introduce the new adaptive Algorithm~\ref{the algorithm} and provide our main results on {\sl a~posteriori} error analysis and optimal convergence in Theorem~\ref{thm:main}.
The proof of the latter is postponed to Section~\ref{sec:proof}, where we essentially verify the abstract \textit{axioms of adaptivity} of \cite{axioms} and sketch how they imply optimal convergence.
Auxiliary results of general interest include a new Scott--Zhang-type operator onto rational splines (Section~\ref{sec:scott zhang}) and inverse inequalities (Section~\ref{sec:inverse inequalities}), which are well-known for standard BEM.
In Section~\ref{section:numerics}, we underline our theoretical findings via numerical experiments. 
There, we consider both the hyper-singular integral equation  as well as  weakly-singular integral equation.
Indeed, the our results for the hyper-singular case are briefly generalized in the appendix, where we also comment on slit problems.




\section{Preliminaries}
\label{section:preliminaries}
\subsection{General notation}
Throughout and without any ambiguity, $|\cdot|$ denotes the absolute value of scalars, the Euclidean norm of vectors in $\R^2$,  the cardinality of a discrete set, the measure of a set in $\R$ (e.g., the length of an interval), or the arclength of a curve in $\R^{2}$.
We write $A\lesssim B$ to abbreviate $A\le cB$ with some generic constant $c>0$, which is clear from the context.
Moreover, $A\simeq B$ abbreviates $A\lesssim B\lesssim A$. 
Throughout, mesh-related quantities have the same index, e.g., $\NN_\coarse$ is the set of nodes of the partition $\QQ_\coarse$, and $h_\coarse$ is the corresponding local mesh-width function etc. 
The analogous notation is used for partitions $\QQ_\fine$ resp.\ $\QQ_\ell$ etc. 
We use  \,$\widehat{\cdot}$\, to transform notation on  the boundary to the  parameter domain, e.g., $\widehat\QQ_\ell$ is the partition of the parameter domain corresponding to the partition $\QQ_\ell$ of $\Gamma$. 
Throughout, we make the the following convention: If $\NN_\coarse$ is a set of nodes and $\alpha_\coarse(z)\ge 0$ is defined for all $z\in\NN_\coarse$, then 
\begin{align}\label{eq:estimator convention}
\alpha_\coarse:=\alpha_\coarse(\NN_\coarse), \quad\text{where}\quad \alpha_\coarse(\mathcal{S}_\coarse)^2:=\sum_{z\in\mathcal{S}_\coarse}\alpha_\coarse(z)^2\quad\text{for all }\mathcal{S}_\coarse\subseteq\NN_\coarse.
\end{align}

\def\Cgamma{C_\Gamma}
\subsection{Sobolev spaces}
\label{section:sobolev}
The usual Lebesgue and Sobolev spaces on $\Gamma$ are denoted by $L^2(\Gamma)=H^0(\Gamma)$ and  
$H^1(\Gamma)$. 
For measurable $\Gamma_0\subseteq\Gamma$, we define the corresponding seminorm 
 \begin{align}
 |u|_{H^{1}({\Gamma_0})} := \norm{\partial_\Gamma u}{L^2(\Gamma_0)}\quad \text{for all }u\in H^1(\Gamma)
 \end{align}
with the arclength derivative $\partial_\Gamma$.
It holds that 
\begin{align}
\norm{u}{H^1(\Gamma)}^2= \norm{u}{L^2(\Gamma)}^2+|u|_{H^1(\Gamma)}^2\quad\text{for all }u\in H^1(\Gamma).
\end{align}
Moreover, $\H^1(\Gamma)$ is the space of 
$H^1(\Gamma)$ functions, which have a vanishing trace on the relative boundary 
$\partial\Gamma$ equipped with the same norm.
Sobolev spaces of fractional order $0<\sigma<1$ are defined by
the 
$K$-method of interpolation~\cite[Appendix B]{mclean}: 
For $0<\sigma<1$, let
$H^\sigma(\Gamma) := [L^2(\Gamma),H^1(\Gamma)]_\sigma$.

For $0<\sigma\le1$, Sobolev spaces of negative order are defined by duality
$H^{\mp\sigma}(\Gamma) := H^{\pm\sigma}(\Gamma)^*$, 
 where duality is understood with
respect to the extended $L^2(\Gamma)$-scalar
product $\dual\cdot\cdot_\Gamma$. 
Finally, we define $H^{\pm \sigma}_0(\Gamma) = \set{v\in H^{\pm \sigma}(\Gamma)}{\dual{v}{1}_{\Gamma} = 0}$
for all $0\le \sigma\le1$.

All details and equivalent definitions of the Sobolev spaces 
are, for instance, found in the monographs~\cite{hw,mclean,ss}.

\subsection{Hyper-singular integral equation}
\label{subsec:hypsing}%
The hyper-singular integral equation~\eqref{eq:hyper strong} employs the  hyper-singular operator $\mathfrak{W}$ as well as the adjoint double-layer operator $\mathfrak{K}'$.
With the  fundamental solution $G(x,y):=-\frac{1}{2\pi}\log|x-y|$ of the 2D Laplacian and the outer normal vector $\nu$, 
these have  the following boundary integral representations
\begin{align}
 \mathfrak{W}v(x)
 = \frac{\partial_x}{\partial\nu(x)}\int_\Gamma v(y)\frac{\partial_y}{\partial\nu(y)}G(x,y)\,dy
\quad \text{and} \quad
 \mathfrak{K}'\psi(x)
 = \int_\Gamma \psi(y)\frac{\partial_x}{\partial\nu(x)}G(x,y)\,dy
\end{align}
for smooth densities $v,\psi:\Gamma\to\R$.

For $0\le \sigma\le1$, the hyper-singular integral operator 
$\WW:H^{\sigma}(\Gamma)\to H^{\sigma-1}(\Gamma)$  and the adjoint double-layer operator $\mathfrak{K}':H^{\sigma-1}(\Gamma)\to H^\sigma(\Gamma)$ are well-defined, linear, and continuous. 


For connected  $\Gamma=\partial\Omega$ and $\sigma=1/2$, the 
operator $\WW$ is symmetric and elliptic up to the constant functions, i.e., 
$\WW:H^{1/2}_0(\Gamma)\to H^{-1/2}_0(\Gamma)$ is 
elliptic. In particular 
\begin{align}
\edual{u}{v}:=\dual{\WW u}{v}_{\Gamma} + \dual{u}{1}_{\Gamma}\dual{v}{1}_{\Gamma}
\end{align}
defines an equivalent scalar product on $H^{1/2}(\Gamma)$ with corresponding norm $\enorm{\cdot}$.
Moreover, there holds  the additional mapping property $\mathfrak{K}':H_0^{-1/2}(\Gamma)\to H_0^{-1/2}(\Gamma)$.

With this notation and provided that $\phi\in H^{-1/2}_0(\Gamma)$, the strong form~\eqref{eq:hyper strong} is equivalently stated in variational form: Find $u\in H^{1/2}(\Gamma)$ such that
\begin{align}
\label{eq:hyper weak}
 \edual{u}{v} = \dual{(1/2-\mathfrak{K}')\phi}{v}_\Gamma
 \quad\text{for all }v\in H^{1/2}(\Gamma).
\end{align}
Therefore, the  Lax-Milgram lemma applies and proves that \eqref{eq:hyper weak} resp.\ \eqref{eq:hyper strong}  admits a unique solution $u\in H^{1/2}(\Gamma)$.
Details are found, e.g., in \cite{hw,mclean,ss,s}.

\subsection{Boundary parametrization}
\label{subsec:boundary parametrization}
We assume that $\Gamma$ is parametrized by a continuous and piecewise continuously differentiable path $\Gamma:[a,b]\to \Gamma$ such that $\gamma|_{(a,b)}$ is injective.
In particular, $\gamma|_{[a,b)}$ and  $\gamma|_{(a,b]}$ are bijective.
Throughout and by abuse of notation, we write $\gamma^{-1}$ for the inverse of $\gamma|_{[a,b)}$ resp. $\gamma|_{(a,b]}$.
The meaning will be clear from the context.

For the left- and right-hand derivative of $\gamma$, we assume that {$\gamma^{\prime_\ell}(t)\neq 0$ for $t\in(a,b]$ and $\gamma^{\prime_r}(t)\neq 0$  for $t\in [a,b)$.}
Moreover, we assume for all $c>0$ that $\gamma^{\prime_\ell}(t)
+c\gamma^{\prime_r}(t)\neq0$  for $t\in(a,b)$ and $\gamma^{\prime_\ell}(b)+c\gamma^{\prime_r}(a)\neq0$.

\subsection{Boundary discretization}
\label{section:boundary:discrete}
In the following, we describe the different quantities, which define the discretization.

{\bf Nodes $\boldsymbol{z_{\coarse,j}=\gamma(\widehat{z}_{\coarse,j})\in\mathcal{N}_\coarse}$.}\quad Let $\mathcal{N}_\coarse:=\set{z_{\coarse,j}}{j=1,\dots,n_\coarse}$ and $z_{\coarse,0}:=z_{n_\coarse}$ 
 be a set of nodes. We suppose that $z_{\coarse,j}=\gamma(\widehat{z}_{\coarse,j})$ for some $\widehat{z}_{\coarse,j}\in[a,b]$ with
$a=\widehat{z}_{\coarse,0}<\widehat{z}_{\coarse,1}<\widehat{z}_{\coarse,2}<\dots<\widehat{z}_{\coarse,n_\coarse}=b$ such that 
$\gamma|_{[\widehat{z}_{\coarse,j-1},\widehat{z}_{\coarse,j}]}\in C^1([\widehat{z}_{\coarse,j-1},\widehat{z}_{\coarse,j}])$.

{\bf Multiplicity $\boldsymbol{\#_\coarse z}$, $\#_\coarse \mathcal{S}_\coarse$, and knots $\boldsymbol{\KK_\coarse}$.}\quad
Let $p\in\N$ be some fixed polynomial order.
Each  interior node $z_{\coarse,j}$ has a multiplicity $\#_\coarse z_{\coarse,j}\in\{1,2\dots, p\}$ and $\#_\coarse {z}_{\coarse,0}=\#_\coarse z_{n_\coarse}=p+1$. 
For $\mathcal{S}_\coarse\subseteq\NN_\coarse$, we set 
\begin{align}
\#_\coarse\mathcal{S}_\coarse:=\sum_{z\in\mathcal{S}_\coarse} \#_\coarse z.
\end{align}
The multiplicities induce the  knot vector
\begin{align}
\KK_\coarse=(\underbrace{z_{\coarse,1},\dots,z_{\coarse,1}}_{\#_\coarse z_{\coarse,1}-\text{times}},\dots,\underbrace{z_{\coarse,n_\coarse},\dots,z_{\coarse,n_\coarse}}_{\#_\coarse z_{\coarse,n_\coarse}-\text{times}}).
\end{align} 

{\bf Elements $\boldsymbol{Q_{\coarse,j}}$ and  partition $\boldsymbol{\QQ_\coarse}$.} \quad
Let $\QQ_\coarse=\{Q_{\coarse,1},\dots,Q_{\coarse,n_\coarse}\}$ be the partition of $\Gamma$ into compact and connected segments $Q_{\coarse,j}=\gamma(\widehat{Q}_{\coarse,j})$ with $\widehat{Q}_{\coarse,j}=[\widehat{z}_{\coarse,j-1},\widehat{z}_{\coarse,j}]$.

{\bf Local mesh-sizes $\boldsymbol{h_{Q}}$, $\boldsymbol{\widehat{h}_{Q}}$ and $\boldsymbol{h_\coarse}$, $\boldsymbol{\widehat{h}_\coarse}$.}\quad
For each element $Q\in\QQ_\coarse$, let $h_Q:=|Q|$ be its arclength on the physical boundary and $\widehat{h}_Q=|\gamma^{-1}(Q|)$ its length in the parameter domain.
Note that the lengths $h_{Q}$ and $\widehat{h}_{Q}$ of an element $Q$ are equivalent, and the equivalence constants depend only on $\gamma$.
We define the local mesh-width function $h_\coarse\in L^\infty(\Gamma)$ by $h_\coarse|_Q:=h_{Q}$.
Additionally, we define $\widehat{h}_\coarse\in L^\infty(\Gamma)$ by $\widehat{h}_\coarse|_Q:=\widehat{h}_{Q}$. 

{\bf Local mesh-ratio $ \boldsymbol{\widehat{\kappa}_\coarse}$.}\quad
We define the {local mesh-ratio} by
\begin{align}\label{eq:meshratio}
\widehat{\kappa}_\coarse&:=\max\set{\widehat{h}_{Q}/\widehat{h}_{Q'}}{{Q},{Q}'\in\QQ_\coarse \text{ with }  Q\cap Q'\neq \emptyset}.
\end{align}

{\bf Patches $\boldsymbol{\pi_\coarse^m(\Gamma_0)}$ and $\boldsymbol{\Pi_\coarse^m(\Gamma_0)}$.}
For $\Gamma_0\subseteq\Gamma$, we inductively define patches  by
\begin{align}
 \pi_\coarse^0(\Gamma_0) := \Gamma_0,
 \quad 
 \pi_\coarse^m(\Gamma_0) := \bigcup\set{Q\in\QQ_\coarse}{ {Q}\cap \pi_\coarse^{m-1}(\Gamma_0)\neq \emptyset}.
\end{align}
The corresponding set of elements is defined as
\begin{align}
 \Pi_\coarse^m(\Gamma_0) := \set{Q\in \QQ_\coarse}{ {Q} \subseteq \pi_\coarse^m(\Gamma_0)},
 \quad\text{i.e.,}\quad
 \pi_\coarse^m(\Gamma_0) = \bigcup\Pi_\coarse^m(\Gamma_0).
\end{align}
To abbreviate notation, we set $\pi_\coarse(\Gamma_0) := \pi_\coarse^1(\Gamma_0)$ and $\Pi_\coarse(\Gamma_0) := \Pi_\coarse^1(\Gamma_0)$.
If $\Gamma_0=\{z\}$ for some $z\in\Gamma$, we write  $\pi^m_\coarse(z):=\pi^m_\coarse(\{z\})$ and $\Pi_\coarse^m(z) := \Pi_\coarse^m(\{z\})$, where we skip the index  for $m=1$ as before.

\subsection{Mesh-refinement}
\label{section:mesh-refinement}
We suppose that we are given fixed initial knots $\KK_0$.
For refinement, we use the following strategy.

 \begin{algorithm}\label{alg:refinement}
\textbf{Input:} Knot vector  $\KK_\coarse$, marked nodes $\MM_\coarse\subseteq\NN_\coarse$, local mesh-ratio $\widehat\kappa_0\ge1$.
\begin{enumerate}[\rm(i)]
\item Define the set of marked elements $\MM_{\coarse}':=\emptyset$.
\item If both nodes of an element $Q\in \QQ_\coarse$ belong to $\MM_\coarse$, mark  $Q$ by adding it to $\MM_{\coarse}'$.
\item For all other nodes in $\MM_\coarse$, increase the multiplicity if it is less or equal to $p-1$. 
Otherwise mark the elements which contain one of these nodes,  by adding them to $\MM_{\coarse}'$.
\item
Recursively enrich $\MM_\coarse'$ by 
$\mathcal{U}_\coarse':=\{Q\in\QQ_\coarse\setminus\MM_\coarse':$ $\exists Q'\in\MM_\coarse'\quad Q'\cap Q\neq\emptyset$ and $\max\{\widehat h_Q/\widehat h_{Q'},\widehat h_{Q'}/\widehat h_Q\}>\widehat\kappa_0\}$
until $\mathcal{U}_\coarse'=\emptyset$. 

\item
Bisect all $ Q\in \MM_\coarse'$ in the parameter domain by inserting the midpoint of  $\gamma^{-1}(Q)$ with multiplicity one to the current knot vector.  
\end{enumerate}
\textbf{Output:} Refined knot vector $\KK_\fine=:\refine(\KK_\coarse,\MM_\coarse)$.
\end{algorithm}

The optimal 1D bisection algorithm in~step {\rm(iii)--(iv)} is analyzed in \cite{meshoneD}.
Clearly, $\KK_\fine=\refine(\KK_\coarse,\MM_\coarse)$ is finer than $\KK_\coarse$ in the sense that $\KK_\coarse$ is a subsequence of $\KK_\fine$.
For any knot vector $\KK_\coarse$ on $\Gamma$, we define $\refine(\KK_\coarse)$ as the set of all knot vectors $\KK_\fine$ on $\Gamma$ such that there exist knot vectors $\KK_{(0)},\dots,\KK_{(J)}$ and corresponding marked nodes $\MM_{(0)},\dots,\MM_{(J-1)}$ with $\KK_\fine$ $=\KK_{(J)}$ $=\refine(\KK_{(J-1)},\MM_{(J-1)}),\dots,\KK_{(1)}$ $=\refine(\KK_{(0)},\MM_{(0)})$, and $\KK_{(0)}=\KK_\coarse$. 
Note that $\refine(\KK_\coarse,\emptyset)=\KK_\coarse$, wherefore $\KK_\coarse\in\refine(\KK_\coarse)$. 
We define the set of all \textit{admissible}  knot vectors on $\Gamma$ as 
\begin{align}
\mathbb{K}:=\refine(\KK_0).
\end{align}
According to \cite[Theorem~2.3]{meshoneD}, there holds for arbitrary $\KK_\coarse\in\K$ that 
\begin{align}\label{eq:two kappa}
\widehat h_Q/\widehat h_{Q'}\le 2\widehat\kappa_0 \quad\text{for all }  Q,Q'\in\QQ_\coarse \text{ with }Q\cap Q'\neq \emptyset.
\end{align}
Indeed, one can easily show that 
 $\mathbb{K}$ coincides with the set of all knot vectors $\KK_\coarse$ which are obtained via iterative bisections in the parameter domain and arbitrary knot multiplicity increases, which satisfy \eqref{eq:two kappa}.

\subsection{B-splines and NURBS}
\label{subsec:splines}
Throughout this subsection, we consider an arbitrary but fixed sequence {$\widehat{\mathcal{K}}_\coarse:=(t_{\coarse,i})_{i\in\Z}$} on $\R$ with multiplicities $\#_\coarse t_{\coarse,i}$ which satisfy
$t_{\coarse,i-1}\leq t_{\coarse,i}$ for $i\in \Z$ and $\lim_{i\to \pm\infty}t_{\coarse,i}=\pm \infty$.
Let $\widehat{\mathcal{N}}_\coarse:=\set{t_{\coarse,i}}{i\in\Z}=\set{\widehat{{z}}_{\coarse,j}}{j\in \Z}$ denote the corresponding set of nodes with $\widehat{{z}}_{\coarse,j-1}<\widehat{{z}}_{\coarse,j}$ for $j\in\Z$.
Throughout, we use the convention that $(\cdot)/0:=0$.
For $i\in\Z$, the $i$-th \textit{B-spline} of degree $p$ is defined for $t\in\R$ inductively by
\begin{align}
\begin{split}
\widehat B_{\coarse,i,0}(t)&:=\chi_{[t_{\coarse,i-1},t_{\coarse,i})}(t),\\
\widehat B_{\coarse,i,p}(t)&:=\frac{t-t_{\coarse,i-1}}{t_{\coarse,i-1+p}-t_{\coarse,i-1}}  \widehat B_{\coarse,i,p-1}(t)+\frac{t_{\coarse,i+p}-t}{t_{\coarse,i+p}-t_{\coarse,i}} \widehat B_{\coarse,i+1,p-1}(t) \quad \text{for } p\in \N.
\end{split}
\end{align}
The following lemma collects some
basic properties of B-splines; see, e.g., \cite{Boor-SplineBasics}.

\begin{lemma}\label{lem:properties for B-splines}
For $-\infty<a<b<\infty$,  $I=[a,b)$, and $p\in \N_0$, the following assertions  {\rm(i)--(vi)} hold:
\begin{enumerate}[\rm(i)]
\item \label{item:spline basis}
The set $\set{\widehat B_{\coarse,i,p}|_I}{i\in \Z\wedge \widehat B_{\coarse,i,p}|_I\neq 0}$ is a basis of the space of all right-continuous $\widehat{\mathcal{N}}_\coarse$-piecewise polynomials of degree $\le p$ on $I$, which are, at each knot $t_{\coarse,i}$, $p-\#_\coarse t_{\coarse,i}$ times continuously differentiable if $p-\#_\coarse t_{\coarse,i}> 0$ { resp.\ continuous for $p=\#_\coarse t_{\coarse,i}$.}
\item \label{item:B-splines local} For $i\in\Z$, the B-spline $\widehat B_{\coarse,i,p}$ vanishes outside the interval $[t_{\coarse,i-1},t_{\coarse,i+p})$. 
It is positive on the open interval $(t_{\coarse,i-1},t_{\coarse,i+p})$.
\item \label{item:B-splines determined} For $i\in \Z$,  the B-splines $\widehat B_{\coarse,i,p}$ is completely determined by the $p+2$ knots $t_{\coarse,i-1},\dots,t_{\coarse,i+p}$, wherefore we also write
\begin{align}
\widehat B(\cdot|t_{\coarse,i-1},\dots,t_{\coarse,i+p}):=\widehat B_{\coarse,i,p}.
\end{align}

\item\label{item:B-splines partition} The  B-splines of degree $p$ form a (locally finite) partition of unity, i.e.,
\begin{equation}
\sum_{i \in\Z} \widehat B_{\coarse,i,p}=1\quad \text{on }\R.
\end{equation}
\item\label{item:interpolatoric} For $i\in \Z$ with $t_{\coarse,i-1}<t_{\coarse,i}=\dots=t_{\coarse,+p}<t_{\coarse,i+p+1}$, it holds that 
\begin{align}
\widehat B_{\coarse,i,p}(t_{\coarse,i}-)=1\quad\text{and} \quad \widehat B_{\coarse,i+1,p}(t_{\coarse,i})=1,
\end{align}
where $\widehat B_{\coarse,i,p}(t_{\coarse,i}-)$ denotes the left-hand limit at $t_{\coarse,i}$.

\item \label{item:derivative of splines}
For $p\ge 1 $ and $i\in\Z$, it holds for the right derivative
\begin{equation}\label{eq:derivative of splines}
\widehat B_{\coarse,i,p}'^{_r}=\frac{p}{t_{\coarse,i+p-1}-t_{\coarse,i-1}} \widehat B_{\coarse,i,p-1}-\frac{p}{t_{\coarse,i+p}-t_{\coarse,i}}\widehat B_{\coarse,i+1,p-1},
\end{equation}
\hfill\qed
\end{enumerate}
\end{lemma}

In addition to the knots $\widehat{\mathcal{K}}_\coarse=(t_{\coarse,i})_{i\in\Z}$, we consider fixed positive weights {$\mathcal{W}_\coarse:=(w_{\coarse,i})_{i\in\Z}$} with $w_{\coarse,i}>0$.
For $i\in \Z$ and $p\in \N_0$, we define the $i$-th 
NURBS by
\begin{equation}
\widehat R_{\coarse,i,p}:=\frac{w_{\coarse,i} \widehat B_{\coarse,i,p}}{\sum_{k\in\Z}  w_{\coarse,k}\widehat B_{\coarse,k,p}}.
\end{equation}
Note that the denominator is locally finite and positive.
For any $p\in\N_0$, we define the spline space as well as the rational spline space 
\begin{equation}\label{eq:NURBS space defined} 
\widehat{\mathcal{S}}^p(\widehat{\mathcal{K}}_\coarse):={\rm span} \,
\set{\widehat B_{\coarse,i,p}}{i\in\Z}\quad\text{and}\quad
\widehat{\mathcal{S}}^p(\widehat{\mathcal{K}}_\coarse,\mathcal{W}_\coarse):={\rm span} \,
\set{\widehat R_{\coarse,i,p}}{i\in\Z}.
\end{equation}

\subsection{Ansatz spaces}
\label{section:igabem} 
 We abbreviate $N_0:=\#_0 \NN_0$ and suppose that additionally to the initial knots $\KK_0$,   $\mathcal{W}_0=(w_{0,i})_{i=1-p}^{N_0-p}$ are given initial weights with  $w_{0,1-p}=w_{0,N_0-p}$.
To apply the results of Section~\ref{subsec:splines},  extend the knot sequence in the parameter domain, i.e., $\widehat{\KK}_0=(t_{0,i})_{i=1}^{N_0}$ arbitrarily to $(t_{0,i})_{i\in \Z}$ with $t_{0,-p}=\dots=t_{0,0}=a$, $t_{0,i}\le t_{0,i+1}$, $\lim_{i\to \pm\infty}t_{0,i}=\pm \infty$.
For the extended sequence, we also write  $\widehat{\mathcal{K}}_0$.
We define the weight function 
\begin{align}
\widehat W_0:=\sum_{k=1-p}^{N_0-p} w_{0,k}\widehat B_{0,k,p}|_{[a,b)}.
\end{align}

Let $\KK_\coarse\in\K$ be a knot vector and abbreviate $N_\coarse:=\#_\coarse\NN_\coarse$.
Outside of $(a,b]$, we extend  the corresponding knot sequence $\widehat\KK_\coarse$ as before to guarantee that $\widehat\KK_0$ forms a subsequence of $\widehat\KK_\coarse$.
Via {knot insertion} from $\widehat\KK_0$ to $\widehat\KK_\coarse$,  Lemma~\ref{lem:properties for B-splines} \eqref{item:spline basis} proves the existence and uniqueness of   weights $\WWe_\coarse=(w_{\coarse,i})_{i=1-p}^{N_\coarse-p}$ with
\begin{align}\label{eq:w}
\widehat W_0=\sum_{k=1-p}^{N_0-p} w_{0,k}\widehat B_{0,k,p}|_{[a,b)}=\sum_{k=1-p}^{N_\coarse-p} w_{\coarse,k} \widehat B_{\coarse,k,p}|_{[a,b)}.
\end{align}
By choosing these weights, we ensure that the denominator of the considered rational splines  does not change.
These weights are just convex combinations of the initial weights $\WWe_0$; see, e.g., \cite[Section 11]{Boor-SplineBasics}. 
For $\widehat w\in\WWe_\coarse$, this shows that
\begin{align}\label{eq:weights}
w_{\min}:=\min(\WWe_0)\le\min(\WWe_\coarse)\le \widehat w\le\max(\WWe_\coarse)\le \max(\WWe_0)=:w_{\max}.
\end{align}
Moreover, $\widehat B_{\coarse,1-p}(a)=\widehat B_{\coarse,N_\coarse-p}(b-)=1$ from Lemma \ref{lem:properties for B-splines} \eqref{item:interpolatoric}  implies  that $w_{\coarse,1-p}=w_{\coarse,N_\coarse-p}$.
Finally, we extend $\mathcal{W}_\coarse$ arbitrarily to $(w_{\coarse,i})_{i\in\Z}$ with $w_{\coarse,i}>0$, identify the extension with $\mathcal{W}_\coarse$, and set
\begin{equation}
{\mathcal{S}}^p({\mathcal{K}}_{\coarse},\mathcal{W}_{\coarse}):=\set{\widehat S_\coarse\circ\gamma^{-1}}{\widehat S_\coarse \in \widehat{\mathcal{S}}^p(\widehat{\mathcal{K}}_\coarse,\mathcal{W}_\coarse}.
\end{equation}
Lemma \ref{lem:properties for B-splines} \eqref{item:spline basis} shows that this definition does not depend on how the sequences are extended.
We define the transformed basis functions 
\begin{align}\label{eq:basis def}
{R}_{\coarse,i,p}:=
\widehat R_{\coarse,i,p}\circ\gamma^{-1}.
\end{align}

We introduce the  ansatz space 
\begin{align}\label{eq:hypsing X0}
\XX_\coarse:=\set{V_\coarse\in{\mathcal{S}}^p({\mathcal{K}}_\coarse,\mathcal{W}_\coarse)}{ V_\coarse(\gamma(a))=V_\coarse(\gamma(b-))}\subset H^{1}(\Gamma)
\end{align}w
Lemma \ref{lem:properties for B-splines}  \eqref{item:spline basis} and  \eqref{item:interpolatoric}  show that bases of these spaces are given by 
\begin{align}\label{eq:hypsing basis}
\set{  R_{\coarse,i,p}}{i=2-p,\dots,N_\coarse-p-1}\cup \{  R_{\coarse,1-p,p}+  R_{\coarse,N_\coarse-p,p}\}.
\end{align}
By Lemma \ref{lem:properties for B-splines} \eqref{item:spline basis}, the ansatz spaces are nested, i.e., 
\begin{align}\label{eq:nested}
\XX_\coarse\subseteq \XX_\fine\quad\text{for all }\KK_\coarse\in\K\text{ and } \KK_\fine\in\refine(\KK_\coarse).
\end{align}

We  define $\phi_\coarse:=P_\coarse\phi$, where $P_\coarse$ is throughout either the  identity or the $L^2$-orthogonal projection onto the space of transformed piecewise polynomials 
\begin{align}
\PP^p(\QQ_\coarse):=\set{\widehat \Psi_\coarse\circ\gamma^{-1}}{\widehat\Psi_\coarse \text{ is }\widehat\QQ_\coarse\text{-piecewise polynomial of degree } p}.
\end{align}
The corresponding Galerkin approximation $U_\coarse\in\XX_\coarse$ reads 
\begin{align}\label{eq:hypsing Galerkin}
 \edual{U_\coarse}{V_\coarse} = \dual{g_\coarse}{V_\coarse}_\Gamma 
\,\text{ for all }V_\coarse\in \XX_\coarse,\quad\text{where }g_\coarse:=(1/2-\mathfrak{K}')\phi_\coarse. 
\end{align}
We note that the choice $\phi_\coarse:=\phi$ is only of theoretical interest as it led to instabilities in our numerical experiments, in contrast to the weakly-singular case \cite{igafaermann,resigabem,gantner17}.


\section{Main result}\label{sec:main result}
In this section, we introduce a novel adaptive algorithm and state its convergence behavior.

\subsection{Error estimators} \label{sec:estimators}
Let $\KK_\coarse\in\K$. 
The definition of the error estimator \eqref{eq:hyper estimator} requires the additional regularity  
$\phi\in L^2(\Gamma)$, which leads to 
$g_\coarse=(1/2-\KK')\phi_\coarse\in L^2(\Gamma)$ due to the mapping properties of 
$\mathfrak{K}'$.
Moreover, note that 
$\WW U_\coarse\in L^2(\Gamma)$ due to the mapping properties of 
$\WW$ and the fact that 
$U_\coarse\in\XX_\coarse\subset H^1(\Gamma)$. 
Therefore, the following error indicators  are well-defined.
We consider the sum of weighted-residual error indicators by \cite{cs95,carstensen97} and  oscillation terms
\begin{subequations}\label{eq:hyper estimator}\begin{align}
\eta_\coarse(z)^2:={\rm res}_\coarse(z)^2+{\rm osc}_\coarse(z)^2 \quad\text{for all }z\in\NN_\coarse,
\end{align}
where
\begin{align}
{\rm res}_\coarse(z):=\norm{h_\coarse^{1/2}(g_\coarse-\WW U_\coarse)}{L^2(\pi_\coarse(z))} \quad \text{and}\quad  {\rm osc}_\coarse(z):=\norm{h_\coarse^{1/2}(\phi-\phi_\coarse)}{L^2(\pi_\coarse(z))}. 
\end{align}
\end{subequations}
Recall the convention~\eqref{eq:estimator convention} for $\alpha_\coarse\in \{\eta_\coarse,{\rm res}_\coarse,{\rm osc}_\coarse\}$.

To incorporate the possibility of knot multiplicity decrease, we define the  knots $\KK_{\coarse\ominus 1}$ by decreasing the multiplicities of all nodes $z\in\NN_\coarse$ whose multiplicity is larger than $1$ and  the original  multiplicity $\#_0 z$ if $z\in\NN_0$, i.e., $\NN_{\coarse\ominus 1}:=\NN_\coarse$ and $\#_{\coarse\ominus 1}z:=\max\{\#_\coarse z-1,1,\#_0z\}$ (where we set $\#_0z:=0$ if $z\notin\NN_0$).
Let 
\begin{align}
J_{\coarse\ominus1}:L^2(\Gamma)\to\XX_{\coarse\ominus 1}
\end{align}
denote the corresponding Scott--Zhang-type projection 
from Section~\ref{section:scott} below. 
To measure the approximation error by multiplicity decrease, we consider the following indicators 
\begin{align}\label{eq:inverse estimator}
\mu_\coarse(z):=\norm{  h_\coarse^{-1/2}(1-J_{\coarse\ominus 1})U_\coarse}{L^2(\pi_\coarse^{2p+1}(z))}\text{ for all }z\in\NN_\coarse.
\end{align}
We define $\mu_\coarse$ and $\mu_\coarse({\mathcal{S}}_\coarse)$ as in \eqref{eq:hyper estimator}. 

\subsection{Adaptive algorithm}
We propose the following  adaptive algorithm.
\begin{algorithm}\label{the algorithm}
\textbf{Input:} Adaptivity parameters $0<\theta<1$, $\vartheta\ge0$, $\Cmin\ge 1$, $\Cmark>0$.\\
\textbf{Adaptive loop:} For each $\ell=0,1,2,\dots$ iterate the following steps {\rm(i)--(iv)}:
\begin{itemize}
\item[\rm(i)] Compute Galerkin approximation 
$U_\ell\in\XX_\ell$.
\item[\rm(ii)] Compute refinement and coarsening indicators $\eta_{\ell}(z)$ and $\mu_\ell(z)$ for all $z\in\NN_{\ell}$.
\item[\rm(iii)] Determine an up to the multiplicative constant 
$\Cmin$  
minimal set of nodes $\MM_\ell^1\subseteq\NN_\ell$, which satisfies the D\"orfler marking 
\begin{align}\label{eq:Doerfler}
 \theta\,\eta_{\ell}^2 \le \eta_\ell(\MM_\ell^1)^2.
\end{align}
\item[\rm(iv)]Determine a  set of  nodes $\MM^-_{\ell}\subseteq \set{z\in\NN_\ell}{\#_\ell z>\#_{\ell\ominus 1}z}$ with $|\MM^-_\ell| \le \Cmark|\MM_\ell|$,  which satisfies the following  marking criterion 
\begin{align}\label{eq:invdo}
\mu_{\ell}(\MM_\ell^-)^2\le \vartheta \eta_\ell^2, 
\end{align}
and define $\MM_\ell^2:=\set{z\in\NN_\ell}{\pi_\ell(z)\cap\MM_\ell^{-}\neq\emptyset}$ as well as $\MM_\ell:=\MM_\ell^1\cup\MM_\ell^2$.
\item[\rm(v)] Generate refined intermediate knot vector $\KK_{\ell+1/2}:=\refine(\KK_\ell,\MM_\ell)$ and then coarsen knot vector $\KK_{\ell+1/2}$ to $\KK_{\ell+1}$ from by decreasing the multiplicities of all  $z\in\MM^-_\ell$ by one.
\end{itemize}
\textbf{Output:} Approximations 
$U_\ell$ and error estimators $\eta_\ell$ for all $\ell \in \N_0$.
\end{algorithm}

\begin{remark}\label{rem:feature}
{\rm (a)} By additionally marking the nodes $\MM_\ell^2$, we enforce that the neighboring elements of any node $z\in\MM_\ell^-$, marked for multiplicity decrease, are bisected.
We emphasize that the enriched set $\MM_\ell$ still satisfies the D\"orfler marking with parameter $\theta$ and is minimal up to the multiplicative constant $\Cmin+3\Cmark$.

{\rm (b)} Algorithm~\ref{the algorithm} allows the choice $\vartheta=0$ and $\MM_\ell^-=\emptyset$, and then formally coincides with the adaptive algorithm from \cite{resigabem} for the weakly-singular integral equation.

{\rm (c)} 
Let even $\const{min}\ge 3$.
If we choose in each step $\widetilde\MM_\ell^1$ up to the multiplicative constant $\const{min}/3$ minimal such that $\theta\,\eta_{\ell}^2 \le \eta_\ell(\widetilde\MM_\ell^1)^2$, and define $\MM_\ell^1:=\set{z\in\NN_\ell}{\pi_\ell(z)\cap\widetilde\MM_\ell^1\neq\emptyset}$, then $\MM_\ell^1$ is as in  Algorithm~\ref{the algorithm} {\rm (iii)}, then this leads to standard $h$-refinement with no multiplicity increase and thus no decrease (independently on how $\const{mark}$ and $\vartheta$ are chosen).

\end{remark}


\subsection{Linear and optimal convergence}
\label{sec:main}
Our main result is that Algorithm~\ref{the algorithm} guarantees linear convergence with optimal algebraic rates.
For standard BEM with piecewise polynomials, such a result is proved in \cite{gantumur,fkmp,part1} for weakly-singular integral equations and in \cite{gantumur,part2} for hyper-singular integral equations, where \cite{part1,part2} also account for data oscillation terms.
For IGABEM for the weakly-singular integral equation (but without knot multiplicity decrease), an analogous result is already  proved in our recent work \cite{optigabem}. 
To precisely state the main theorem, let
\begin{align}
\K(N):=\set{\KK_\coarse\in\K}{\#_\coarse\NN_\coarse-\#_0\NN_0\le N}
\end{align}
be the finite set of all  refinements having at most $N$ knots more than $\KK_0$.

Analogously to \cite{axioms}, we introduce the estimator-based approximability constant 
\begin{align}\label{eq:Capprox}
\norm{u}{\mathbb{A}_s}:=\sup_{N\in\N_0} \big( (N+1)^s \inf_{\KK_\coarse\in\K(N)} \eta_\coarse\big)\in\mathbb{R}_{\ge 0} \cup\{\infty\}\quad \text{for all }s>0. 
\end{align}
By this constant, one can characterize the best possible convergence rate.
In explicit terms, this constant is finite if and only if an algebraic convergence rate of $\mathcal{O}(N^{-s})$ for the estimator is possible for suitably chosen  knot vectors.
Similarly, we define 
\begin{align}\label{eq:Ks}
\begin{split}
&\K^1:=\set{\KK_\coarse\in\K}{\#_\coarse z_{\coarse,j}=\max\{1,\#_0 z_{\coarse,j}\}\text{ for all }j=1,\dots,|\NN_\coarse|-1},\\
&\K^1(N):=\K(N)\cap\K^1,
\\
&\K^p:=\set{\KK_\coarse\in\K}{\#_\coarse z_{\coarse,j}=p\text{ for all }j=1,\dots,|\NN_\coarse|-1},\\
&\K^p(N):=\set{\KK_\coarse\in\K^p}{\#_\coarse\NN_\coarse-\#_{0,p}\NN_0\le N} \text{ with }\KK_{0,p}\in\K^p\text{ and }\NN_{0,p}=\NN_0,
\end{split}
\end{align}
and
\begin{align}
\norm{u}{\mathbb{A}^1_s}:=\sup_{N\in\N_0} \big( (N+1)^s \inf_{\KK_\coarse\in\K^1(N)} \eta_\coarse\big)\quad\text{ and }\quad
\norm{u}{\mathbb{A}^p_s}:=\sup_{N\in\N_0} \big( (N+1)^s \inf_{\KK_\coarse\in\K^p(N)} \eta_\coarse\big).
\end{align}
The constant $\norm{u}{\mathbb{A}^1_s}$ characterizes the best possible convergence rate starting from  $\KK_0$ when  only bisection is used and all new nodes have multiplicity $1$.
The constant $\norm{u}{\mathbb{A}^p_s}$ characterizes the best possible rate starting from the coarsest knot vector $\KK_{0,p}\in\K^p$ when only bisection is used  and all new nodes have maximal multiplicity $p$. 
Hence, $\norm{u}{\mathbb{A}^p_s}$ characterizes the rate for standard BEM with continuous piecewise polynomial ansatz functions.
Note that the constants coincide if $p=1$.

The following theorem is  the main result of our work.
The proof is given in Section~\ref{sec:proof}.

\begin{theorem}\label{thm:main}
Let 
$\phi\in L^2(\Gamma)$ 
so that the weighted-residual error estimator   is well-defined.
Then, the estimator $\eta$ from \eqref{eq:hyper estimator} is reliable as well as weakly efficient, i.e., there exist $\Crel,\Ceff>0$ such that, for all $\KK_\coarse\in\K$, 
\begin{align}\label{eq:reliable}
\Crel^{-1}\norm{u-U_\coarse}{\H^{1/2}(\Gamma)}
\le 
\eta_\coarse
\le
\Ceff \big( \norm{h_\coarse^{1/2} \partial_\Gamma( u - U_\coarse )}{L^2(\Gamma)}^2
 + \norm{h_\coarse^{1/2} ( \phi - \phi_\coarse )}{L^2(\Gamma)}^2 \big)^{1/2}.
\end{align}
For each $0<\theta\le1$, there is a constant $\vartheta_{\rm opt}>0$ such that for all $0\le \vartheta<\vartheta_{\rm opt}$ there  exist constants $0<\q{conv}<1$ and $\Cconv>0$ such that Algorithm~\ref{the algorithm} is linearly convergent in the sense that
\begin{align}\label{eq:R-linear}
\eta_{\ell+k}\leq \Cconv \,\q{conv}^k\,\eta_\ell \quad \text{for all }k,\ell\in\N_0.
\end{align}
Moreover, there is a constant $0<\theta_{\rm opt}<1$ such that for all $0<\theta<\theta_{\rm opt}$ and $0\le\vartheta<\vartheta_{\rm opt}$, there exist constants $c_{\rm opt},\Copt>0$ such that, for all  $s>0$, there holds that 
\begin{align}\label{eq:optimal}
c_{\rm opt}\norm{u}{\mathbb{A}_s}
\le
\sup_{\ell\in\N_0}{(\#_\ell\NN_\ell-\#_0\NN_0+1)^{s}}{\eta_\ell}\le\Copt \norm{u}{\mathbb{A}_s}.
\end{align}
Finally, there exist constants $c_{\rm apx},\const{apx}>0$ such that, for all  $s>0$, there holds that 
\begin{align}\label{eq:classes}
c_{\rm apx}\norm{u}{\A_s^1}\le\norm{u}{\A_s}\le 
\min\{\norm{u}{\A_s^1},\norm{u}{\A_s^p}\}\le\const{apx}\norm{u}{\A_s^1}. 
\end{align}

The constants $\Crel$ and $\Ceff$ depend only on $\gamma, p, \widehat\QQ_0, w_{\min},$ and $w_{\max}$.
The constant $\vartheta_{\rm opt}$ depends additionally on $\theta$.
The constants  $\q{conv}$ as well as $\Cconv$ depend further on $\theta$ and $\vartheta$.
The constant $\theta_{\rm opt}$ depends only on $\gamma, p, \widehat\QQ_0, w_{\min},$ and $w_{\max}$, whereas, $\Copt$ depends additionally on $\theta,\vartheta,\Cmin, \Cmark,$ and $s$.
The constant $c_{\rm opt}$ depends only on $\#_0\NN_0$.
Finally, the constants $c_{\rm apx},\const{apx}$ depend only  on $\gamma, p, \widehat\QQ_0, w_{\min}$, $w_{\max}$, and $s$.
\hfill$\square$
\end{theorem}

\begin{remark}
Theorem~\ref{thm:main} holds accordingly for indirect BEM, where 
 $g=(1/2-\mathfrak{K}')\phi$ in \eqref{eq:hyper strong} is replaced by $g=\phi$, and $g_\coarse=(1/2-\mathfrak{K}')\phi_\coarse$ in~\eqref{eq:hypsing Galerkin} is replaced either by $g_\coarse=\phi$ or by $g_\coarse=\phi_\coarse$. 
Indeed, due to the absence of the operator $\mathfrak{K}'$ for indirect BEM, the proof is even  simplified.
\end{remark}


\section{Proof of Theorem~\ref{thm:main} }
\label{sec:proof}

To prove Theorem~\ref{thm:main}, we follow the abstract convergence theory for adaptive algorithms of \cite{axioms}, which provide a set of so-called \textit{axioms of adaptivity}, 
which automatically guarantee linear convergence at optimal algebraic rate. 
Although we cannot directly apply their result, since it does not cover multiplicity increase or decrease, we will verify slightly modified axioms, which yield Theorem~\ref{thm:main} with the same ideas as in \cite{axioms}.
In Section~\ref{subsec:abstract}, we present these axioms.
Their verification, which is inspired by the corresponding verification for standard BEM \cite{part2}, is postponed to Section~\ref{sec:hypsing stability}--\ref{sec:hypsing discrete reliability} and \ref{subsec:hypsing orthogonality}--\ref{sec:hypsing closure}, after providing some auxiliary results in Section~\ref{sec:interpolation}--\ref{sec:inverse inequalities}.
In Section~\ref{sec:reliable hyper} and \ref{sec:linear convergence}--\ref{sec:optimal convergence}, we briefly show  how these axioms conclude 
reliability in \eqref{eq:reliable}, linear convergence \eqref{eq:R-linear}, and optimal convergence \eqref{eq:optimal}
(along the lines of \cite{axioms}).
Efficiency in \eqref{eq:reliable} is proved in Section~\ref{sec:efficient hyper} similarly as for standard BEM \cite[Section~3.2]{invest}.  
Finally,  Section~\ref{sec:classes} verifies the relation \eqref{eq:classes} between the approximability constants.

\subsection{Axioms of adaptivity}\label{subsec:abstract}
In this section, we formulate node-based versions of the \textit{axioms of adaptivity} of \cite{axioms}. 
These are not satisfied for the error estimator $\eta$ itself, but only for a locally equivalent estimator $\widetilde\eta$.
To introduce this  estimator, we first recall an equivalent mesh-size function that has been constructed in \cite[Proposition~5.8.2]{gantner17} or in \cite[Proposition~4.2]{optigabem} in a slightly different element-based version.

\begin{proposition}\label{prop:h tilde}
For $\KK_\coarse\in\K$ and $z\in\NN_\coarse$, 
let  $z_{\coarse,{\rm left}}\in\NN_\coarse\cap\pi_\coarse(z)$ be the (with respect to $\gamma$) left neighbor and  $z_{\coarse,{\rm right}}\in \NN_\coarse\cap\pi_\coarse(z)$ the right neighbor of $z$. 
Let $\#_\coarse z_{\coarse, \rm left},\#_\coarse z_{\coarse,\rm right}$ be the corresponding multiplicities. 
Then, there exist $0<\ro{eq}<1$ and $\const{eq}>0$ such that
\begin{align}\label{eq:tilde h equivalent}
\const{eq}^{-1}h_\coarse|_{\pi_\coarse(z)}\le
\widetilde h_{\coarse,z}
:=|\gamma^{-1}(\pi_\coarse(z))| \, \ro{eq}^{\#_\coarse z_{\coarse,{\rm left}}+\#_\coarse z+\#_\coarse z_{\coarse,{\rm right}}}
\le \const{eq} h_\coarse|_{\pi_\coarse(z)},
\end{align}
where $\ro{eq}$ depends only on $p$ and $\widehat\QQ_0$ and  $\const{eq}$ depends additionally on $\gamma$.
If additionally $\KK_\fine\in\refine(\KK_\coarse)$, then there exists a constant $0<\ro{ctr}<1$ such that for all $  z\in\NN_\coarse$, whose patch is changed by bisection or multiplicity increase (i.e., $\pi_\coarse(z)\neq\pi_\fine(z)$ or $\#_\coarse z_{\coarse,{\rm left}}\neq\#_\coarse z_{\fine,{\rm left}}$ or $\#_\coarse z\neq\#_\fine z$ or $\#_\coarse z_{\coarse,{\rm right}}\neq\#_\fine z_{\coarse,{\rm right}}$), 
and all $z'\in\NN_\fine$, 
it holds that
\begin{align}\label{eq:h tilde ctr}
\widetilde h_{\fine,z'}\le\ro{ctr}\,\widetilde  h_{\coarse,z}\quad\text{if }z'\in \pi_\coarse(z)\setminus\NN_\coarse\quad\text{ or }\quad\text{if }z'=z.
\end{align}
where $\ro{ctr}$ depends only on $p$ and $\widehat\QQ_0$. \hfill$\square$
\end{proposition}

For $\KK_\coarse\in\K$, we  define the  estimator 
\begin{subequations}\label{eq:hyper estimator tilde}
\begin{align}
\widetilde\eta_\coarse(z)^2:=\widetilde{\rm res}_\coarse(z)^2+\widetilde{\rm osc}_\coarse(z)^2 \quad\text{for all }z\in\NN_\coarse,
\end{align}
where
\begin{align}
\widetilde{\rm res}_\coarse(z):=\widetilde h_{\coarse,z}^{\,1/2}\norm{g_\coarse-\WW U_\coarse}{L^2(\pi_\coarse(z))}\quad\text{and}\quad \widetilde{\rm osc}_\coarse(z):=\widetilde h_{\coarse,z}^{\,1/2}\norm{\phi-\phi_\coarse}{L^2(\pi_\coarse(z))}. 
\end{align}
\end{subequations}
In particular, \eqref{eq:tilde h equivalent} implies the local equivalence 
\begin{align}\label{eq:abstract equivalence}
\const{eq}^{-1}\eta_\coarse(z)^2\le \widetilde\eta_\coarse(z)^2\le\const{eq}\eta_\coarse(z)^2\quad\text{for all }z\in\NN_\coarse.
\end{align}

To present the axioms of adaptivity in a compact way, we abbreviate for $\KK_\coarse,\KK_\fine\in\K$ the corresponding perturbation terms
\begin{align}
\varrho_{\coarse,\fine}:=\norm{U_\coarse-U_\fine}{H^{1/2}(\Gamma)}+\norm{\phi_\coarse-\phi_\fine}{H^{-1/2}(\Gamma)}.
\end{align}
Moreover,  we define the set of all nodes in $\NN_\coarse\cap \NN_\fine$ whose patch is identical in $\KK_\coarse$ and $\KK_\fine$
\begin{align}\label{eq:nnr}
\begin{split}
\NNnr_{\coarse,\fine}:=\big\{z\in\NN_\coarse\cap \NN_\fine:\,&\pi_\coarse(z)=\pi_\fine(z), 
\#_\coarse z'=\#_\fine z' \text{ for all } z'\in\{z,z_{\coarse,\rm left},z_{\coarse,\rm right}\}\big\}.
\end{split}
\end{align}
We abbreviate its complement in $\NN_\coarse$ and $\NN_\fine$
\begin{align}
\NNr_{\coarse,\fine}:=\NN_\coarse\setminus\NNnr_{\coarse,\fine}\quad\text{and}\quad\NNr_{\fine,\coarse}:=\NN_\fine\setminus\NNnr_{\coarse,\fine}. 
\end{align}
In Section~\ref{sec:hypsing stability}--\ref{sec:hypsing discrete reliability}, we will verify that if $\vartheta>0$ is chosen sufficiently small, then there exist constants $\Cstab,\Cred,\Cqo, \Cref,\Cdrel, \const{son},\Cclos\ge 1$  and $0\le\q{red},\varepsilon_{\rm qo}<1$ 
 such that the following properties for the estimator (E1)--(E4) and the refinement (R1)--(R3) are satisfied:
\begin{itemize}
\item[(E1)]\textbf{Stability on non-refined node patches:} For all  $\KK_\coarse\in\K$ and  all $\KK_\fine\in\refine(\KK_\coarse)$ as well as for all $\ell\in\N_0$, $\KK_\coarse:=\KK_\ell$, and  $\KK_\fine:=\KK_{\ell+1}$, it holds that
\begin{align*}
|\widetilde\eta_\fine(\NNnr_{\coarse,\fine})-\widetilde\eta_\coarse(\NNnr_{\coarse,\fine})|\le\Cstab\,\varrho_{\coarse,\fine}.
\end{align*}
\item[(E2)]\textbf{Reduction on refined node patches:} For all  $\KK_\coarse\in\K$ and all $\KK_\fine\in\refine(\KK_\coarse)$ as well as for all $\ell\in\N_0$, $\KK_\coarse:=\KK_\ell$ and  $\KK_\fine:=\KK_{\ell+1}$, it holds that
\begin{align*}
\widetilde\eta_\fine(\NNr_{\fine,\coarse})^2\le\q{red}\widetilde\eta_\coarse(\NNr_{\coarse,\fine})^2+\Cred^2\,\varrho_{\coarse,\fine}^2.
\end{align*}
\item[(E3)] \textbf{Discrete reliability:} For all $\KK_\coarse\in\K$ and all $\KK_\fine\in\refine(\KK_\coarse)$,
there exists $\NNr_{\coarse,\fine}\subseteq\RR_{\coarse,\fine}\subseteq\NN_\coarse$ with
 $\#_\coarse\RR_{\coarse,\fine}\le\Cref (\#_\fine\NN_\fine-\#_\coarse\NN_\coarse)$ such that
\begin{align*}
\varrho_{\coarse,\fine}^2\le
\Cdrel^2\,\widetilde\eta_\coarse(\RR_{\coarse,\fine})^2.
\end{align*}
\item[(E4)] \textbf{General quasi-orthogonality:} There holds that
\begin{align*}
0\le\varepsilon_{\rm qo}<\sup_{\delta>0}\frac{1-(1+\delta)(1-(1-\rhored)\theta)}{\Cred+(2+\delta^{-1})\Cstab^2},
\end{align*}
and the sequence of knots $(\KK_\ell)_{\ell\in\N_0}$ satisfies 
that 
\begin{align*}
\sum_{k=\ell}^{\ell+N}(\varrho_{k,k+1}^2-\varepsilon_{\rm qo}\widetilde\eta_k^{\,2})\le\Cqo\, \widetilde\eta_\ell^{\,2}\quad \text{for all }\ell,N\in\N_0. 
\end{align*}
\item[(R1)]\textbf{Son estimate}
For all $\ell\in\N_0$, it holds that 
\begin{align*}
\#_{\ell+1}\NN_{\ell+1}\le \const{son} \#_\ell\NN_\ell.
\end{align*}

\item[(R2)] \textbf{Closure estimate:} 
For all $\ell\in\N_0$, there holds that 
\begin{align*}
\#_\ell\NN_\ell-\#_0\NN_0\le\Cclos \sum_{k=0}^{\ell-1}\#_k\MM_k.
\end{align*}
\item[(R3)] \textbf{Overlay property:} 
For all $\KK_\coarse, \KK_\star\in\K$, there exists a common refinement  $\KK_\fine\in\refine(\KK_\coarse)\cap\refine(\KK_\star)$  such that 
\begin{align*}
\#_\fine\NN_\fine \le \#_\coarse\NN_\coarse+\#_{\star}\NN_{\star}-\#_0\NN_0.
\end{align*}
\end{itemize}

\subsection{Interpolation theory}
\label{sec:interpolation}
We start with a maybe well-known abstract interpolation result (stated, e.g.,  in {\cite[Lemma~2]{hypsing3d}}), which will be applied in the following. 
\begin{lemma}\label{lem:discrete interpolation}
For $j=0,1$, let $H_j$ be Hilbert spaces with subspaces $X_j\subseteq H_j$, which satisfy the continuous inclusions $H_0\supseteq H_1$ and $X_0\supseteq X_1$.
Assume that $A:H_j\to X_j$ is a well-defined linear and continuous projection with operator norm $c_j=\norm{A:H_j\to X_j}{}$, for both $j=0,1$. 
Then, there holds equivalence of the interpolation norms
\begin{align}
\norm{v}{[H_0,H_1]_\sigma}\le \norm{v}{[X_0,X_1]_\sigma}\le c_0^{1-\sigma} c_1^\sigma \norm{v}{[H_0,H_1]_\sigma}\quad \text{for all } v\in [X_0,X_1]_\sigma
\end{align}
and all $0<\sigma<1$.\hfill$\square$
\end{lemma}
In the following, we write $(X,\norm{\cdot}{X})\simeq (Y,\norm{\cdot}{Y})$ if $X=Y$ in the sense of sets with equivalent norms $\norm{\cdot}{X}\simeq\norm{\cdot}{Y}$. The Stein-Weiss interpolation theorem (see, e.g., \cite[Theorem 5.4.1]{bergh}) shows for fixed $\alpha\in\{0,1\}$ that for $\KK_\coarse\in\K$ there holds that
\begin{align}\label{eq:intp1}
(L^2(\Gamma),\norm{h_\coarse^{\alpha-\sigma}(\cdot)}{L^2(\Gamma)})\simeq [(L^2(\Gamma),\norm{h_\coarse^{\alpha}(\cdot)}{L^2(\Gamma)}),(L^2(\Gamma),\norm{h_\coarse^{\alpha-1}(\cdot)}{L^2(\Gamma)})]_\sigma,
\end{align}
where the hidden constants depend only on $\Gamma$ and $\sigma$.
Moreover, it holds by definition that
\begin{align}\label{eq:intp2}
(H^\sigma(\Gamma), \norm{\cdot}{H^\sigma(\Gamma)}) = [(L^2(\Gamma),\norm{\cdot}{L^2(\Gamma)}),(H^1(\Gamma), \norm{\cdot}{H^1(\Gamma)}]_\sigma.
\end{align}

\subsection{Scott--Zhang-type projection}
\label{sec:scott zhang}
\label{section:scott} 
\noindent
In this section, we  introduce a Scott--Zhang-type operator for $\KK_\coarse\in\K$.
In \cite[Section 2.1.5]{overview}, it is shown that, for $i\in \{1-p,\dots,N_\coarse-p\}$, there exist dual basis {functions $\widehat B_{\coarse,i,p}^*\in L^2(a,b)$ such that
\begin{align}
\supp \widehat B_{\coarse,i,p}^*&=\supp \widehat B_{\coarse,i,p}=[t_{\coarse,i-1},t_{\coarse,i+p}],\\
\int_a^b  \widehat B_{\coarse,i,p}^* \widehat B_{\coarse,j,p} \,dt&=\delta_{ij}=\begin{cases} 1,  \text{ if }i=j,\\ 0,\text{ else,}\end{cases}\\
\label{eq:dual inequality}
\norm{\widehat B_{\coarse,i,p}^*}{L^2([a,b])}&\le 9^p(2p+3) {|\supp \widehat B_{\coarse,i,p}|^{-1/2}}.
\end{align}
Each dual basis function depends only on the knots $t_{\coarse,i-1},\dots,t_{\coarse,i+p}$. 
With the denominator $\widehat w$ from \eqref{eq:w},  define  
\begin{align}\label{eq:ripstar}
\widehat R_{\coarse,i,p}^*:=\widehat B_{\coarse,i,p}^* \widehat w/w_{\coarse,i}.
\end{align}
This immediately proves that 
\begin{align}
\label{eq:duality}&\int_a^b  \widehat R_{\coarse,i,p}^* (R_{\coarse,j,p}\circ\gamma) dt=\delta_{ij}=\begin{cases} 1,\quad \text{if }i=j,\\ 0,\quad \text{else,}\end{cases}
\end{align}
and
\begin{align}\label{eq:dual inequality2}
\norm{\widehat R_{\coarse,i,p}^*}{L^2(a,b)}\lesssim 9^p(2p+3) {|\supp R_{\coarse,i,p}|^{-1/2}},
\end{align}
where the hidden constant depends only on $\gamma, w_{\min}$, and $w_{\max}$. 
With the abbreviation $N_\coarse:=\#_\coarse \NN_\coarse$, we define the Scott--Zhang-type operator $J_\coarse:L^2(\Gamma)\to\XX_\coarse$ by
\begin{subequations}\label{eq:scotty}
\begin{align}
\begin{split}
J_\coarse v&:=
\Big(\int_a^b  \frac{\widehat R_{\coarse,1-p,p}^*+\widehat R_{\coarse,N_\coarse-p,p}^*}{2}( v\circ\gamma) \,dt \Big)(R_{\coarse,1-p,p}+R_{\coarse,N_\coarse-p,p})\\
&\quad+\sum_{i=1-p}^{N_\coarse-p-1} \Big(\int_a^b \widehat R_{\coarse,i,p}^* (v \circ\gamma) \,dt \Big)R_{\coarse,i,p}.
\end{split}
\end{align}
\end{subequations}
A similar operator, namely $I_\coarse := \sum_{i=1-p}^{N_\coarse-p} \Big(\int_a^b  \widehat R_{\coarse,i,p}^* (v\circ \gamma) \,dt \Big)R_{\coarse,i,p}$, has been analyzed in \cite[Section 3.1.2]{overview}.
However, $I_\coarse$ is not applicable here because it does not guarantee that  $I_\coarse v$ is continuous at $\gamma(a)=\gamma(b)$.

\begin{proposition}\label{lem:Scott properties}
Given $\KK_\coarse\in \K$, the operator $J_\coarse$ from \eqref{eq:scotty} satisfies the following properties~{\rm (i)--(iv)} with a constant $\Cscott>0$ depending only on $\gamma,p, \widehat \QQ_0, w_{\min}, w_{\max}$, and  $\sigma$:
\begin{enumerate}[{\rm(i)}]
\item Local projection property: For  all $v\in L^2(\Gamma)$ and all $Q\in\QQ_\coarse$, it holds that 
\begin{align}\label{eq:local projection}
(J_\coarse v)|_Q= v|_Q \quad\text{if }v|_{\pi_\coarse^p(Q)}\in \XX_\coarse|_{\pi_\coarse^p(Q)}=\set{V_\coarse|_{\pi_\coarse^p(Q)}}{V_\coarse\in\XX_\coarse}.
\end{align}
\item Local $L^2$-stability: For all $v\in L^2(\Gamma)$ and all $Q\in \QQ_\coarse$, it holds that 
\begin{align}\label{eq:local L2}
\norm{J_\coarse v}{L^2(Q)}\le \Cscott \norm{v}{L^2(\pi_\coarse^{p}(Q))}.
\end{align}
\item Local $H^1$-stability: For all $v\in H^1(\Gamma)$ and all $Q\in \QQ_\coarse$, it holds that 
\begin{align}\label{eq:local H1}
|J_\coarse v|_{H^1(Q)}\le \Cscott |v|_{H^1(\pi_\coarse^{p}(Q))}.
\end{align}
\item Approximation properties: For all $0\le\sigma\le1$ and all $v\in H^\sigma(\Gamma)$, it holds that
\begin{align}\label{eq:local approx}
\norm{h_\coarse^{-\sigma}(1-J_\coarse) v}{L^2(\Gamma)}\le \Cscott \norm{v}{H^\sigma(\Gamma)}
\end{align}
as well as
\begin{align}\label{eq:local approx2}
\norm{(1-J_\coarse)v}{H^{\sigma}(\Gamma)} \le \Cscott \norm{h_\coarse^{1-\sigma}\partial_\Gamma v}{L^2(\Gamma).}
\end{align}
 
\end{enumerate}
\end{proposition}
\begin{proof}

\noindent
\textbf{Proof of{\rm \textbf{(i)}}:}
The proof follows immediately from \eqref{eq:duality} and the fact that 
\eqref{eq:hypsing basis} forms a basis of $\XX_\coarse$.

\noindent
\textbf{Proof of {\rm \textbf{(ii)}}:}
Abbreviate $\overline R_{\coarse,1-p,p}:=R_{\coarse,1-p,p}+R_{\coarse,N_\coarse-p,p}$.
Because of \eqref{eq:dual inequality} and $\supp R_{\coarse,1-p,p}\cap\supp R_{\coarse,N_\coarse-p,p}=\emptyset$, it holds  that 
\begin{align*}
\norm{J_\coarse v}{L^2(Q)}&\stackrel{\eqref{eq:scotty}}{=}\Big\|\Big(\int_a^b  \frac{\widehat R_{\coarse,1-p,p}^*+\widehat R_{\coarse,N_\coarse-p,p}^*}{2} v\circ\gamma \,dt \Big)\overline R_{\coarse,1-p,p}\\
&\hspace{5cm}+\sum_{i=1-p}^{N_\coarse-p-1} \Big(\int_a^b \widehat R_{\coarse,i,p}^* v\circ\gamma \,dt \Big)R_{\coarse,i,p}\Big\|_{L^2(Q)}\\
&\,\,\le\,\, \int_a^b \Big| \frac{\widehat R_{\coarse,1-p,p}^*+\widehat R_{\coarse,N_\coarse-p,p}^*}{2} v\circ\gamma\Big| \,dt \,|Q\cap \supp\overline R_{\coarse,1-p,p}|^{1/2}\\
&\quad+\sum_{i=1-p\atop Q\subseteq \supp R_{\coarse,i,p}}^{N_\coarse-p-1} \int_a^b |\widehat R_{\coarse,i,p}^* v\circ\gamma |\,dt \,|Q|^{1/2}\\
&\stackrel{\eqref{eq:dual inequality}}{\lesssim}|\supp\overline R_{\coarse,1-p,p}|^{-1/2}\norm{v\circ\gamma}{L^2(\supp \widehat R_{\coarse,1-p,p}\cup\supp  \widehat R_{\coarse,N_\coarse-p,p})}\,|Q\cap\supp \overline R_{\coarse,1-p,p}|^{1/2}\\
&\quad+ \sum_{i=1-p\atop Q\subseteq \supp R_{\coarse,i,p}}^{N_\coarse-p-1} |\supp {R}_{\coarse,i,p}|^{-1/2}\norm{v\circ\gamma}{L^2(\supp  \widehat R_{\coarse,i,p})}|Q|^{1/2}\lesssim \norm{ v}{L^2(\pi_\coarse^{p}(Q))}.
\end{align*}

\noindent

\noindent
\textbf{Proof of {\rm \textbf{(iii)}}:}
We show that $|(1-J_\coarse)v|_{H^1(Q)}\lesssim |v|_{H^1(\pi_\coarse^p(Q))}$. 
With Lemma~\ref{lem:properties for B-splines} \eqref{item:derivative of splines},  for $i=1-p,\dots,N_\coarse-p$, it holds that 
\begin{align*}
|R_{\coarse,i,p}|_{H^1(\Gamma)}&\simeq |\widehat R_{\coarse,i,p}|_{H^1(a,b)}=\norm{(w_{\coarse,i}\widehat B_{\coarse,i,p}/\widehat w)'}{L^2(a,b)}\simeq\Big\|\frac{\widehat B_{\coarse,i,p}'\widehat w-\widehat B_{\coarse,i,p}\widehat w'}{\widehat w^2}\Big\|_{L^2(a,b)}\\
&\lesssim\norm{\widehat B_{\coarse,i,p}'}{L^2(a,b)}+\norm{\widehat w'}{L^2(\supp \widehat B_{\coarse,i,p})}\stackrel{\eqref{item:derivative of splines}}\lesssim|\supp \widehat B_{\coarse,i,p}|^{-1/2}\simeq |\supp {R}_{\coarse,i,p}|^{-1/2}.
\end{align*}
With \eqref{eq:dual inequality}, we   see for $\widetilde v\in H^1(\Gamma)$ that
\begin{align}\label{eqpr:one minus J}
\begin{split}
|J_\coarse \widetilde{v}|_{H^1(Q)}&\stackrel{\eqref{eq:dual inequality}}\lesssim |\supp \overline R_{\coarse,1-p,p}|^{-1/2} \norm{\widetilde{v}\circ\gamma}{L^2(\supp \widehat R_{\coarse,1-p,p}\cup \supp \widehat R_{\coarse,1-p,p})} |\supp\overline R_{\coarse,1-p,p}|_{H^1(Q)}\\
&\quad +\sum_{i=1-p}^{N_\coarse-p-1}|\supp \widehat R_{\coarse,i,p}|^{-1/2} \norm{\widetilde{v}\circ\gamma}{L^2(\supp \widehat R_{\coarse,i,p})} |R_{\coarse,i,p}|_{H^1(Q)}\\
&\,\,\lesssim\,\,h_{\coarse,Q}^{-1/2} \norm{\widetilde{v}}{L^2(\pi_\coarse^p(Q))}h_{\coarse,Q}^{-1/2}+\sum_{i=1-p\atop |Q\cap \supp R_{\coarse,i,p}|>0}^{N_\coarse-p-1}h_{\coarse,Q}^{-1/2}\norm{\widetilde{v}}{L^2(\pi_\coarse^p(Q))} h_{\coarse,Q}^{-1/2}\\
&\,\,\lesssim\,\,\norm{h_\coarse^{-1}\widetilde{v}}{L^2(\pi_\coarse^p(Q))}.
\end{split}
\end{align}
For $v\in H^1(\Gamma)$, let  $\overline{v}:=\int_{\pi_\coarse^p(Q)} v \,dx/|\pi_\coarse^p(Q)| $ be the integral mean of $v$ over $\pi_\coarse^p(Q)$. 
Choosing $\widetilde v:=v-\overline v$ in \eqref{eqpr:one minus J} and using the Poincar\'e inequality (see, e.g., \cite[Lemma 2.5]{faermann2d}), we conclude that 
\begin{align*}
|(1-J_\coarse)v|_{H^1(Q)}&\,\,\,\,=\,\,\,\,|(1-J_\coarse)(v-\overline v)|_{H^1(Q)}\\
&\stackrel{\eqref{eqpr:one minus J}}\lesssim|v-\overline v|_{H^1(Q)}+\norm{h_\coarse^{-1}(v-\overline v)}{L^2(\pi_\coarse^p(Q))}\lesssim |v|_{H^1(\pi_\coarse^p(Q))}.
\end{align*}


\noindent
\textbf{Proof of {\rm \textbf{(iv)}}:}
First, we prove \eqref{eq:local approx}.
With {\rm (ii)},  it holds  that 
\begin{align*}
\norm{(1-J_\coarse)\widetilde v}{L^2(Q)}\lesssim \norm{\widetilde v}{L^2(\pi_\coarse^p(Q))}\quad \text{for all }\widetilde v\in L^2(\Gamma).
\end{align*}
By taking the square and summing over all elements, this already proves the assertion for $\sigma=0$.
Now, we prove it for $\sigma=1$ by showing  that 
\begin{align}\label{eq:aux local approx2}
\norm{h_\coarse^{-1}(1-J_\coarse)v}{L^2(Q)}\lesssim |{v}|_{H^1(\pi_\coarse^p(Q))}\quad \text{for all } v\in L^2(\Gamma).
\end{align}
We choose $\widetilde v:=v-\overline v$ with $\overline v:=\int_{\pi_\coarse^p(Q)} v \,dx/|\pi_\coarse^p(Q)|$ and apply the Poincar\'e  inequality.
Note, that \eqref{eq:local approx} for arbitrary $\sigma$ is equivalent to the boundedness of 
\begin{align*}
1-J_\coarse:(H^\sigma(\Gamma),\norm{\cdot}{H^\sigma}(\Gamma))\to (L^2(\Gamma),\norm{h_\coarse^{-\sigma}(\cdot)}{L^2(\Gamma)}),
\end{align*}
which follows with \eqref{eq:intp1} and \eqref{eq:intp2} from the interpolation theorem
\cite[Theorem~B.2]{mclean}.

Next, we prove \eqref{eq:local approx}. 
The localization argument \cite[Lemma~2.3]{faermann2d} in combination with \cite[Lemma~4.5]{resigabem} proves  that 
$\norm{\widetilde v}{H^{1/2}(\Gamma)}
\lesssim 
\norm{h_\coarse^{-1/2}\widetilde v}{L^2(\Gamma)}
+\norm{h_\coarse^{1-1/2}\partial_\Gamma \widetilde v}{L^2(\Gamma)}$ for all $\widetilde v\in\widetilde H^{1/2}(\Gamma)$.
The proofs extend verbatim to $0<\sigma<1$ and the assertion is trivially satisfied for $\sigma\in\{0,1\}$. 
With $\widetilde v:=(1-J_\coarse)v$, the latter inequality (for $0\le \sigma\le1$), \eqref{eq:aux local approx2}, and \eqref{eq:local H1} show that 
\begin{eqnarray*}
\norm{(1-J_\coarse)v}{H^\sigma(\Gamma)}
&\stackrel{\text{\cite{faermann2d,resigabem}}}\lesssim &
\norm{h_\coarse^{-\sigma}(1-J_\coarse)v}{L^2(\Gamma)}
+\norm{h_\coarse^{1-\sigma}\partial_\Gamma (1-J_\coarse)v}{L^2(\Gamma)}
\\
&\stackrel{\eqref{eq:aux local approx2}+\eqref{eq:local H1}}\lesssim&
 \norm{h_\coarse^{1-\sigma}\partial_\Gamma v}{L^2(\Gamma)}.
\end{eqnarray*}
This concludes the proof.

\end{proof}

\subsection{Inverse inequalities}
\label{sec:inverse inequalities}
\label{sec:inverse inequalities}
The first result is taken from \cite[Theorem~3.1]{invest}.
\begin{proposition}
Let $\KK_\coarse\in\K$.
Then, there exists a constant $\Cinv>0$ such that 
\begin{align}\label{eq:inverse W}
\norm{h_\coarse^{1/2}\mathfrak{W} v}{L^2(\Gamma)}\le \Cinv \big(\norm{v}{H^{1/2}(\Gamma)} + \norm{h_\coarse^{1/2}\partial_\Gamma v}{L^2(\Gamma)}\big)\quad \text{for all }v\in H^1(\Gamma)
\end{align}
and 
\begin{align}\label{eq:inverse Kpr}
\norm{h_\coarse^{1/2} \mathfrak{K}'\psi}{L^2(\Gamma)}\le \Cinv \big(\norm{\psi}{H^{-1/2}(\Gamma)} + \norm{h_\coarse^{1/2}\psi}{L^2(\Gamma)}\big)\quad\text{for all }\psi\in L^2(\Gamma).
\end{align}
The constant $\Cinv>0$ depends only on $\gamma$ and $\widehat\QQ_0$.
\hfill$\square$
\end{proposition}

The next proposition provides inverse inequalities for rational splines, which are well-known for piecewise polynomials; see \cite{inverse2,hypsing3d}.
It also recalls a standard inverse inequality for piecewise polynomials.


\begin{proposition}\label{lem:discrete invest2}
Let $\KK_\coarse\in\K$ and $0\le\sigma\le 1$.
Then, there exists $\widetilde C_{\rm inv}>0$ such that 
\begin{align}\label{eq:discrete invest}
\norm{h_\coarse^{1-\sigma}\partial_\Gamma{V_\coarse}}{L^2(\Gamma)}\le \widetilde C_{\rm inv} \norm{V_\coarse}{H^\sigma(\Gamma)} \quad\text{for all } V_\coarse\in \XX_\coarse,
\end{align}
and
\begin{align}\label{eq:discrete invest old}
\norm{h_\coarse^{\sigma}\Psi_\coarse}{L^2(\Gamma)}\le\widetilde C_{\rm inv}\norm{\Psi_\coarse}{H^{-\sigma}(\Gamma)}\quad\text{for all }\Psi_\coarse\in\PP^p(\QQ_\coarse),
\end{align}
and
\begin{align}\label{eq:discrete invest2}
\norm{V_\coarse}{H^\sigma(\Gamma)}\le \widetilde C_{\rm inv} \norm{h_\coarse^{-\sigma}V_\coarse}{L^2(\Gamma)} \quad\text{for all } V_\coarse\in \XX_\coarse.
\end{align}
The constant $\widetilde C_{\rm inv}>0$ depends only on $\gamma, p, \widehat\QQ_0, w_{\min}, w_{\max}$, and $\sigma$.
\end{proposition}
\begin{proof}
\eqref{eq:discrete invest old} is proved, e.g., in \cite[Proposition~4.1]{optigabem} even for piecewise rational splines. 
We prove the other two assertions in three steps.

\noindent
\textbf{Step 1:} We prove that \eqref{eq:discrete invest} and \eqref{eq:discrete invest2}  hold (even elementwise) for $\sigma\in\{0,1\}$.
We start with \eqref{eq:discrete invest}.
For $\sigma=1$ the assertion is trivial. 
If $\sigma=0$, let $Q\in\QQ_\coarse$, define $\widehat{Q}:=\gamma^{-1}(Q)$, and let $\Phi_{\widehat{Q}}$ be the affine bijection  which maps $[0,1]$ onto $\widehat Q$.
Then, it holds that
\begin{align*}
|V_\coarse|_{H^1(Q)}^2&=\int_Q( \partial_\Gamma V_\coarse)^2 \,dx=\int_{\widehat{Q}} (V_\coarse\circ\gamma)'(t)^2/|\gamma'(t)| \,dt=|\widehat Q|\int_0^1 (V_\coarse\circ\gamma)'(\Phi_{\widehat{Q}}(t))^2/|\gamma'(\Phi_{\widehat Q}(t))|\,dt\\
&\simeq |\widehat{Q}|^{-1}\int_0^1 (V_\coarse\circ\gamma\circ\Phi_{\widehat Q})'(t)^2\,dt=|\widehat{Q}|^{-1}|V_\coarse\circ\gamma\circ\Phi_{\widehat Q}|_{H^1(0,1)}^2.
\end{align*}
Note that $V_\coarse\circ\gamma\circ\Phi_{\widehat{Q}}$ is just a rational function on the interval $[0,1]$. 
It can be written as $q/\widetilde w$ with some polynomials $q,\widetilde w\in \mathcal{P}^p(0,1)$ of degree $p$, where $0<w_{\min}\le\widetilde w\le w_{\max}$. 
Independently of the norm on the finite dimensional space $\mathcal{P}^p(0,1)$, differentiation $(\cdot)':\mathcal{P}^p(0,1)\to \mathcal{P}^p(0,1)$ is continuous.
This proves that $\norm{q'}{L^2(0,1)}\le C\norm{q}{L^2(0,1)}$ as well as $\norm{\widetilde w'}{L^2(0,1)}\le C$, where $C>0$  depends only on   $p$, $w_{\min}$, $w_{\max}$. 
With the quotient rule, we conclude that
\begin{align*}
|V_\coarse\circ\gamma\circ\Phi_{\widehat{Q}}|_{H^1(0,1)}=|q/\widetilde{w}|_{H^1(0,1)}\lesssim \norm{q}{L^2(0,1)}\simeq\norm{V_\coarse\circ\gamma\circ\Phi_{\widehat{Q}}}{L^2(0,1)}.
\end{align*}
This shows that 
\begin{align*}
|V_\coarse|^2_{H^1(Q)}\lesssim|\widehat{Q}|^{-1} \norm{V_\coarse\circ\gamma\circ\Phi_{\widehat Q}}{L^2(0,1)}^2\simeq \norm{h_\coarse^{-1} V_\coarse}{L^2(Q)}^2.
\end{align*}
Now, we consider \eqref{eq:discrete invest2}.
For $\sigma=0$, the assertion is trivial.
The case $\sigma=1$ follows from \eqref{eq:discrete invest} with $\sigma=0$.


\noindent
\textbf{Step 2:}
We prove \eqref{eq:intp1} and \eqref{eq:intp2} for the discrete space $\XX_\coarse$. Note that Proposition \ref{lem:Scott properties} proves that $J_\coarse$ is a stable projection onto $\XX_\coarse$ considered as a mapping from $(L^2(\Gamma),\norm{\cdot}{L^2(\Gamma)})$ to $(\XX_\coarse,\norm{\cdot}{L^2(\Gamma)})$, from $(L^2(\Gamma),\norm{h^{-1}(\cdot)}{L^2(\Gamma)})$ to $(\XX_\coarse,\norm{h^{-1}(\cdot)}{L^2(\Gamma)})$, or from $(H^1(\Gamma),\norm{\cdot}{H^1(\Gamma)})$ to $(\XX_\coarse,\norm{\cdot}{H^1(\Gamma)})$.
Due to \eqref{eq:intp1} and \eqref{eq:intp2}, Lemma~\ref{lem:discrete interpolation} is applicable and proves that
\begin{align}\label{eq:intp3}
(\XX_\coarse,\norm{h_\coarse^{-\sigma}{(\cdot})}{L^2(\Gamma)})\simeq [(\XX_\coarse,\norm{\cdot}{L^2(\Gamma)}),(\XX_\coarse,\norm{h_\coarse^{-1}(\cdot)}{L^2(\Gamma)})]_\sigma,
\end{align}
and 
\begin{align}\label{eq:intp4}
(\XX_\coarse,\norm{\cdot}{H^\sigma(\Gamma)})\simeq[(\XX_\coarse,\norm{\cdot}{L^2(\Gamma)}), (\XX_\coarse,\norm{\cdot}{H^1(\Gamma)})]_\sigma.
\end{align}

\noindent
\textbf{Step 3:}
Consider the differentiation operator
\begin{align*}
{ \partial_\Gamma}:(\XX_\coarse,\norm{\cdot}{H^\sigma(\Gamma)})\to (L^2(\Gamma),\norm{h_\coarse^{1-\sigma}(\cdot)}{L^2(\Gamma)}),
\end{align*}
and the formal identity
\begin{align*}
{\rm id}:(\XX_\coarse,\norm{h_\coarse^{-\sigma}(\cdot)}{L^2(\Gamma)})\to(H^\sigma(\Gamma),\norm{\cdot}{H^\sigma(\Gamma)}).
\end{align*}
Then, \eqref{eq:discrete invest} resp. \eqref{eq:discrete invest2}  is equivalent to boundedness of   $\partial_\Gamma$ resp. ${\rm id}$.
For $\sigma\in\{0,1\}$,  $\partial_\Gamma$ and ${\rm id}$ are bounded according to Step 1. 
Finally, Step 2 and the  well-known interpolation theorem   \cite[Theorem B.2]{mclean} prove boundedness of the mappings  $\partial_\Gamma$ and ${\rm id}$.
\end{proof}

\subsection{Stability on non-refined node patches (E1)}
\label{sec:hypsing stability}
Note that $\pi_\coarse(z)=\pi_\fine(z)$ if $z\in\NNnr_{\coarse,\fine}$ and that $\widetilde h_{\coarse,z}=\widetilde h_{\fine,z}$ for $z\in\NNnr_{\coarse,\fine}$, which follows from the definitions \eqref{eq:tilde h equivalent} and \eqref{eq:nnr}.
The reverse triangle inequality  proves that
\begin{align*}
&|\widetilde\eta_\fine(\NNnr_{\coarse,\fine})-\widetilde\eta_\coarse(\NNnr_{\coarse,\fine})|\\
&\quad\le \Big| \sum_{ z\in\NNnr_{\coarse,\fine}}\big(\norm{\widetilde h_{\fine, z}^{1/2} (g_\fine-\mathfrak{W}U_\fine)}{L^2(\pi_\fine(z))} -\norm{\widetilde h_{\coarse, z}^{1/2} (g_\coarse-\mathfrak{W}U_\coarse)}{L^2(\pi_\fine(z))}  \big)^2\Big|^{1/2}\\
&\qquad+ \Big| \sum_{ z\in\NNnr_{\coarse,\fine}}\big(\norm{\widetilde h_{\fine, z}^{1/2} (\phi-\phi_\fine)}{L^2(\pi_\fine(z))} -\norm{\widetilde h_{\coarse, z}^{1/2} (\phi-\phi_\coarse)}{L^2(\pi_\fine(z))}  \big)^2\Big|^{1/2}\\
&\quad\le \Big| \sum_{ z\in\NNnr_{\coarse,\fine}}\norm{\widetilde h_{\fine, z}^{1/2} \mathfrak{W}(U_\fine-U_\coarse)}{L^2(\pi_\fine(z))}  ^2\Big|^{1/2}
+\Big| \sum_{ z\in\NNnr_{\coarse,\fine}}\norm{\widetilde h_{\fine, z}^{1/2} (g_\fine-g_\coarse)}{L^2(\pi_\fine(z))}  ^2\Big|^{1/2}\\
&\qquad+\Big| \sum_{ z\in\NNnr_{\coarse,\fine}}\norm{\widetilde h_{\fine, z}^{1/2} (\phi_\fine-\phi_\coarse)}{L^2(\pi_\fine(z))}  ^2\Big|^{1/2}.
\end{align*}
The regularity of $\gamma$, local quasi-uniformity \eqref{eq:two kappa}, and the equivalence \eqref{eq:tilde h equivalent} yield that 
\begin{align*}
|\widetilde\eta_\fine(\NNnr_{\coarse,\fine})-\widetilde\eta_\coarse(\NNnr_{\coarse,\fine})|
\lesssim\norm{h_\fine^{1/2} \mathfrak{W}(U_\fine-U_\coarse)}{L^2(\Gamma)}+\norm{h_\fine^{1/2} (g_\fine-g_\coarse)}{L^2(\Gamma)}+\norm{h_\fine^{1/2} (\phi_\fine-\phi_\coarse)}{L^2(\Gamma)}.
\end{align*}
If $\KK_\fine\in\refine(\KK_\coarse)$, nestedness \eqref{eq:nested} shows that $U_\fine-U_\coarse\in\XX_\fine$. 
Otherwise, we define  $\KK_{\ell\cup(\ell+1)}\in\K$ via $\NN_{\ell\cup(\ell+1)}:=\NN_\ell\cup\NN_{\ell+1}=\NN_{\ell+1}$ and $\#_{\ell\cup(\ell+1)} z:=\max\{\#_\ell z,\#_{\ell+1} z\}$ for all $z\in\NN_{\ell+1}$, where $\#_\ell z:=0$ if $z\not\in\NN_\ell$.
Then, $U_\fine-U_\coarse\in\XX_{\ell\cup(\ell+1)}$ and $h_{\ell\cup(\ell+1)}=h_\fine$.
Therefore, in each case, the inverse inequalities \eqref{eq:inverse W}--\eqref{eq:discrete invest old} are applicable and conclude the proof. 
The overall constant $\const{stab}$ depends only on  the parametrization $\gamma$, the polynomial order $p$, and the initial mesh $\widehat\QQ_0$.

\subsection{Reduction on refined node patches (E2)}
Let $\delta>0$.
We apply the  triangle inequality and the Young inequality to see that
\begin{align*}
&\widetilde\eta_\fine(\NNr_{\fine,\coarse})^2=\sum_{ z\in\NNr_{\fine,\coarse}} \Big(\norm{\widetilde h_{\fine,z}^{1/2} (g_\fine-\mathfrak{W}U_\fine)}{L^2(\pi_\fine(z))}^2+\norm{\widetilde h_{\fine,z}^{1/2} (\phi-\phi_\fine)}{L^2(\pi_\fine(z))}^2\Big)\\
&\le \sum_{ z\in\NNr_{\fine,\coarse}}\Big( (1+\delta)^2\norm{\widetilde h_{\fine,z}^{1/2} (g_\coarse-\mathfrak{W}U_\coarse)}{L^2(\pi_\fine(z))}^2+(1+\delta^{-1})\norm{\widetilde h_{\fine,z}^{1/2}\mathfrak{W}(U_\fine-U_\coarse)}{L^2(\pi_\fine(z))}^2\\
&\hspace{9cm}+(1+\delta)(1+\delta^{-1})\norm{\widetilde h_{\fine,z}^{1/2}(g_\fine-g_\coarse)}{L^2(\pi_\fine(z))}^2\Big)\\
&\quad+ \sum_{ z\in\NNr_{\fine,\coarse}} \Big((1+\delta)\norm{\widetilde h_{\fine,z}^{1/2} (\phi-\phi_\coarse)}{L^2(\pi_\fine(z))}^2+(1+\delta^{-1})\norm{\widetilde h_{\fine,z}^{1/2}(\phi_\fine-\phi_\coarse)}{L^2(\pi_\fine(z))}^2\Big).
\end{align*}
We only have to estimate the first terms in each of the last two sums, the other terms can be  estimated as in Section~\ref{sec:hypsing stability}.
We split each patch $\pi_\fine(z)=Q_{\fine,{\rm left}}(z)\cup Q_{\fine,{\rm right}}(z)$  into a (with respect to the parametrization $\gamma$) left  and a right element in $\QQ_\fine$.
We obtain that
\begin{align*}
&\sum_{ z\in\NNr_{\fine,\coarse}} \Big(\norm{\widetilde h_{\fine,z}^{1/2} (g_\coarse-\mathfrak{W}U_\coarse)}{L^2(\pi_\fine(z))}^2+ \norm{\widetilde h_{\fine,z}^{1/2} (\phi-\phi_\coarse)}{L^2(\pi_\fine(z))}^2\Big)\\
&\quad=
\sum_{ z\in\NNr_{\fine,\coarse}}\Big( \norm{\widetilde h_{\fine,z}^{1/2} (g_\coarse-\mathfrak{W}U_\coarse)}{L^2(Q_{\fine,{\rm left}}(z))}^2+ \norm{\widetilde h_{\fine,z}^{1/2} (\phi-\phi_\coarse)}{L^2(Q_{\fine,{\rm left}}(z))}^2\Big)
\\&\qquad+\sum_{ z\in\NNr_{\fine,\coarse}} \Big(\norm{\widetilde h_{\fine,z}^{1/2} (g_\coarse-\mathfrak{W}U_\coarse)}{L^2(Q_{\fine,{\rm right}}(z))}^2+\norm{\widetilde h_{\fine,z}^{1/2} (\phi-\phi_\coarse)}{L^2(Q_{\fine,{\rm right}}(z))}^2\Big).
\end{align*}
Let  $z\in\NNr_{\fine,\coarse}$.
If $z\in\NN_\coarse$, we define $z':=z$ and note that $z'\in \NNr_{\coarse,\fine}$. 
 Otherwise,  there exists a unique $z'\in\NN_\coarse$ with $z\in Q_{\coarse,{\rm left}}(z')$, where $Q_{\coarse,{\rm left}}(z')$ is defined analogously as above.
Again, this implies  that $z'\in\NNr_{\coarse,\fine}$.
Altogether, the contraction property~\eqref{eq:h tilde ctr} yields that 
\begin{align*}
&\sum_{ z\in\NNr_{\fine,\coarse}} \Big(\norm{\widetilde h_{\fine,z}^{1/2} (g_\coarse-\mathfrak{W}U_\coarse)}{L^2(Q_{\fine,{\rm left}}(z))}^2+ \norm{\widetilde h_{\fine,z}^{1/2} (\phi-\phi_\coarse)}{L^2(Q_{\fine,{\rm left}}(z))}^2\Big)\\
&\quad\,\,\le\,\,\sum_{ z'\in\NNr_{\coarse,\fine}}
\sum_{\quad{z=z' \text{ or }  z\in Q_{\coarse,{\rm left}}(z')\setminus\NN_\coarse}} 
\Big(
\norm{\widetilde h_{\fine,z}^{1/2} (g_\coarse-\mathfrak{W}U_\coarse)}{L^2(Q_{\fine,{\rm left}}(z))}^2
+\norm{\widetilde h_{\fine,z}^{1/2} (\phi-\phi_\coarse)}{L^2(Q_{\fine,{\rm left}}(z))}^2\Big)\\
&\quad\stackrel{\eqref{eq:h tilde ctr}}\le  \sum_{ z'\in\NNr_{\coarse,\fine} }\ro{ctr} \Big(\,\norm{\widetilde h_{\coarse,z'}^{1/2} (g_\coarse-\mathfrak{W}U_\coarse)}{L^2(Q_{\coarse,{\rm left}}(z'))}^2+\norm{\widetilde h_{\coarse,z'}^{1/2} (\phi-\phi_\coarse)}{L^2(Q_{\coarse,{\rm left}}(z'))}^2\Big).
\end{align*}
The same holds for the right elements. 
Hence, we end up with 
\begin{align*}
&\sum_{ z\in\NNr_{\fine,\coarse}} \norm{\widetilde h_{\fine,z}^{1/2} (g_\coarse-\mathfrak{W}U_\coarse)}{L^2(\pi_{\fine}(z))}^2+ \norm{\widetilde h_{\fine,z}^{1/2} (\phi-\phi_\coarse)}{L^2(\pi_{\fine}(z))}^2\\
&\le \ro{ctr}  \Big(\sum_{z'\in\NNr_{\coarse,\fine} }\norm{\widetilde h_{\coarse,z'}^{1/2} (g_\coarse-\mathfrak{W}U_\coarse)}{L^2(\pi_{\coarse}(z'))}^2+\norm{\widetilde h_{\coarse,z'}^{1/2} (\phi-\phi_\coarse)}{L^2(\pi_{\coarse}(z'))}^2\Big)=\ro{ctr}\,\widetilde\eta_\coarse(\NNr_{\coarse,\fine})^2.
\end{align*}
Choosing $\delta$ sufficiently small such that $\ro{red}:=(1+\delta)^2 \ro{ctr}<1$, we conclude the proof. 
Moreover, our argument shows that $\const{red}\simeq(1+\delta^{-1})\const{stab}^2$ with a generic hidden constant.

\subsection{Discrete reliability (E3)}
\label{sec:hypsing discrete reliability}
We show that there exist constants $\const{drel},\const{ref}\ge 1$ such that for all $\KK_\coarse\in\K$ and all $\KK_\fine\in\refine(\KK_\coarse)$, the subset 
\begin{align}\label{eq:RR defined2}
\RR_{\coarse,\fine}:=\NN_\coarse\cap\pi_\coarse^{2p+1}(\NNr_{\coarse,\fine})
\end{align}
 satisfies  that
\begin{align}\label{eq:drel start}
\big(\norm{U_\fine-U_\coarse}{H^{1/2}(\Gamma)}^2+\norm{\phi_\fine-\phi_\coarse}{H^{-1/2}(\Gamma)}^2\big)^{1/2}\le \const{drel}\,\widetilde\eta_\coarse(\RR_{\coarse,\fine}),
\end{align}
with
\begin{align}\label{eq:aux1}
\NNr_{\coarse,\fine}\subseteq\RR_{\coarse,\fine}\quad
\text{and}\quad
\#_\coarse\RR_{\coarse,\fine}\le\const{ref} (\#_\fine \NN_\fine-\#_\coarse\NN_\coarse).
\end{align}
Obviously, $\NNr_{\coarse,\fine}\subseteq\RR_{\coarse,\fine}$ is satisfied. 
Hence, the first property of \eqref{eq:aux1}, i.e.,  is obvious.
Since  the maximal knot multiplicity is bounded by $p+1$, it holds that 
\begin{align*}
\#_\coarse\RR_{\coarse,\fine}\le(p+1) |\RR_{\coarse,\fine}|\le(p+1)(4p+3)|\NNr_{\coarse,\fine}|\simeq \#_\coarse\NNr_{\coarse,\fine}, 
\end{align*}
where the hidden constant depends only on $p$.
Note that $z\in\NNr_{\coarse,\fine}$ holds only if a  knot is inserted in the corresponding patch $\pi_\coarse(z)$, where a new knot can be inserted in at most three old patches.
Since $\#_\fine \NN_\fine-\#_\coarse\NN_\coarse$ is the number of all new knots, we see that
\begin{align*}
\#_\coarse\NNr_{\coarse,\fine}\le 3\,(\#_\fine \NN_\fine-\#_\coarse\NN_\coarse).
\end{align*}
In the following four steps, we 
prove \eqref{eq:drel start}.

\noindent
\textbf{Step 1:} 
Let $U_{\fine,\coarse}$ denote the unique Galerkin solution to
\begin{align}
\edual{U_{\fine,\coarse}}{ V_\fine} =\dual{g_\coarse}{V_\fine}_{L^2(\Gamma)}\quad\text{for all }V_\fine \in\XX_\fine.
\end{align}
Ellipticity and the definition of $U_\fine$ as well as $U_{\fine,\coarse}$ show that 
\begin{align*}
\norm{U_\fine-U_{\fine,\coarse}}{H^{1/2}(\Gamma)}^2 &\lesssim \edual{U_\fine-U_{\fine,\coarse}}{U_\fine-U_{\fine,\coarse}}= \dual{g_\fine-g_\coarse}{U_\fine-U_{\fine,\coarse}}_{L^2(\Gamma)}\\
&\le \norm{g_\fine-g_\coarse}{H^{-1/2}(\Gamma)}\norm{U_\fine-U_{\fine,\coarse}}{H^{1/2}(\Gamma)}.
\end{align*}
Together with continuity of $\mathfrak{K}'$, this yields that 
\begin{align}\label{eq:ufine minus uqfine}
\norm{U_\fine-U_{\fine,\coarse}}{H^{1/2}(\Gamma)} \lesssim \norm{g_\fine-g_\coarse}{H^{-1/2}(\Gamma)}\lesssim \norm{\phi_\fine-\phi_\coarse}{H^{-1/2}(\Gamma)}.
\end{align}
Moreover, the triangle inequality and the Young inequality prove that 
\begin{align}
&\big(\norm{U_\fine-U_\coarse}{H^{-1/2}(\Gamma)}+\norm{\phi_\fine-\phi_\coarse}{H^{-1/2}(\Gamma)}\big)^2\notag\\
&\quad\,\,\,\lesssim\,\,\, \norm{U_\fine-U_{\fine,\coarse}}{H^{1/2}(\Gamma)}^2 + \norm{U_{\fine,\coarse}-U_\coarse}{H^{1/2}(\Gamma)}^2 + \norm{\phi_\fine-\phi_\coarse}{H^{-1/2}(\Gamma)}^2\notag\\
&\quad\stackrel{\eqref{eq:ufine minus uqfine}}\lesssim \norm{U_{\fine,\coarse}-U_\coarse}{H^{1/2}(\Gamma)} +\norm{\phi_\fine-\phi_\coarse}{H^{-1/2}(\Gamma)}. \label{eq:drel 3}
\end{align}


\noindent
\textbf{Step 2:}
We estimate the last term in \eqref{eq:drel 3}. 
Since the orthogonal projections $P_\coarse,P_\fine$ onto the space of (transformed) piecewise polynomials satisfy that $P_\fine(1-P_\coarse)=P_\fine-P_\coarse=(1-P_\coarse) P_\fine$, the approximation property \cite[Theorem~4.1]{cp06} shows that 
\begin{align}\label{eq:proj proj}
\begin{split}
\norm{\phi_\fine-\phi_\coarse}{H^{-1/2}(\Gamma)} 
=\norm{(P_{\fine}-P_\coarse)\phi}{H^{-1/2}(\Gamma)}
&\stackrel{\text{\cite{cp06}}}\lesssim \norm{h_\coarse^{1/2}(P_{\fine}-P_\coarse)\phi}{L^2(\Gamma)}\\
&\,\,\,\,\lesssim \,\,\,\,\norm{h_\coarse^{1/2}(\phi- \phi_\coarse)}{L^2(\bigcup (\QQ_\coarse\setminus\QQ_\fine))}.
\end{split}
\end{align}
Note that $\bigcup (\QQ_\coarse\setminus\QQ_\fine)\subseteq\pi_\coarse(\RR_{\coarse,\fine})$. 
Together with the equivalence \eqref{eq:tilde h equivalent}, we obtain that 
\begin{align}
\norm{\phi_\fine-\phi_\coarse}{H^{-1/2}(\Gamma)}\lesssim \widetilde{\rm osc}_\coarse(\RR_{\coarse,\fine}).
\end{align}

\noindent
\textbf{Step 3:} 
To proceed, we apply the projection property \eqref{eq:local projection} for  $V_\fine:=U_{\fine,\coarse}-U_\coarse$.
Let  $Q\in\QQ_\coarse\setminus \Pi_\coarse^{2p+1}(\NNr_{\coarse,\fine})$.
We show that 
\begin{align}\label{eq:S local bem2}
V_\fine|_{\pi_\coarse^p(Q)}\in \XX_\coarse|_{\pi_\coarse^p(Q)}=\set{V_\coarse|_{\pi_\coarse^p(Q)}}{V_\coarse\in\XX_\coarse}, 
\end{align}
wherefore \eqref{eq:local projection} will imply that 
\begin{align}\label{eq:S local bem2 imply}
(1-J_\coarse)(U_{\fine,\coarse}-U_\coarse)=0\quad\text{on}\quad\Gamma\setminus \pi_\coarse^{2p+1}(\NNr_{\coarse,\fine}).
\end{align}
First, we argue by contradiction to see that
\begin{align}\label{eq:inverse patch bem2}
\Pi_\coarse^{2p}(Q)\subseteq \Pi_\coarse(\NNnr_{\coarse,\fine}).
\end{align}
Suppose there exists $Q'\in\Pi_\coarse^{2p}(Q)$ with $Q'\not\in\Pi_\coarse(\NNnr_{\coarse,\fine})$.
This is equivalent to  $Q\in\Pi_\coarse^{2p}(Q')$ and $Q'\in \QQ_\coarse\setminus\Pi_\coarse(\NNnr_{\coarse,\fine})$, which yields that $Q\in\Pi_\coarse^{2p}\big(\QQ_\coarse\setminus
\Pi_\coarse(\NNnr_{\coarse,\fine}))$.
Note that 
\begin{align*}
\QQ_\coarse\setminus\Pi_\coarse(\NNnr_{\coarse,\fine})\subseteq\Pi_\coarse(\NNr_{\coarse,\fine}),
\end{align*}
since $Q''$ in the left-hand side implies that $Q''\cap\NN_\coarse\neq\emptyset$ and $z\not\in Q''$ for all $z\in\NNnr_{\coarse,\fine}$, and hence implies  the existence of $z\in\NNr_{\coarse,\fine}$ with $z\in Q''$.
Altogether, we see that
\begin{align*}
Q\in\Pi_\coarse^{2p+1}(\NNr_{\coarse,\fine}),
\end{align*}
which contradicts that  $Q\in\QQ_\coarse\setminus \Pi_\coarse^{2p+1}(\NNr_{\coarse,\fine})$ and thus proves \eqref{eq:inverse patch bem2}.

Next, we prove \eqref{eq:S local bem2}.
Note that $V_\fine|_{\pi_\coarse(Q)}$  can be written as  linear combination of (transformed) B-splines $B_{\fine,i,p}:=\widehat B_{\fine,i,p}\circ\gamma^{-1}$ that have support on $\pi_\coarse(Q)$.
By  Lemma~\ref{lem:properties for B-splines} \eqref{item:B-splines local},  $\supp(B_{\fine,i,p})$ is connected and consists of at most $p+1$ elements, which  implies that $\supp(B_{\fine,i,p})\subseteq\pi_\coarse^{2p}(T)$.
We show that no knots are inserted in $\pi_\coarse^{2p}(Q)$ 
and thus in $\supp(B_{\fine,i,p})$ 
during the refinement from $\KK_\coarse$ to $\KK_\fine$.
To see this, let $z'\in\NN_\fine$ be a corresponding node.
Since $\NNnr_{\coarse,\fine}$ is just the set of all nodes $z$ such that no new knot is inserted in the patch of $z$, $z'$ cannot belong to $\pi_\coarse(\NNnr_{\coarse,\fine})$.
Hence, \eqref{eq:inverse patch bem2} implies that $z'\not\in\pi_\coarse^{2p}(Q)$.
With this,  Lemma~\ref{lem:properties for B-splines} \eqref{item:B-splines determined}  proves that $B_{\fine,i,p}=B_{\coarse,i',p}$ for some B-spline $B_{\coarse,i',p}:=\widehat B_{\fine,i',p}\circ\gamma^{-1}$.
In particular,  $V_\fine|_{\pi_\coarse(Q)}$  can be written as  linear combination of (transformed) B-splines corresponding to  $\KK_\coarse$, which implies \eqref{eq:S local bem2} and  \eqref{eq:S local bem2 imply}.

\noindent
\textbf{Step 4:} 
It remains to estimate the second term in \eqref{eq:drel 3}.
Due to ellipticity as well as Galerkin orthogonality, we see that 
\begin{align*}
\norm{U_{\fine,\coarse}-U_\coarse}{H^{1/2}(\Gamma)}^2&\lesssim \edual{ U_{\fine,\coarse}-U_\coarse}{U_{\fine,\coarse}-U_\coarse} =\edual{ U_{\fine,\coarse}-U_\coarse}{(1-J_\coarse)(U_{\fine,\coarse}-U_\coarse)}.
\end{align*}
It holds that
\begin{align*}
0=\edual{ U_{\fine,\coarse}-U_\coarse}{1}= \dual{\mathfrak{W} (U_{\fine,\coarse}-U_\coarse)}{1}_{L^2(\Gamma)} +\dual{U_{\fine,\coarse}-U_\coarse}{1}_{L^2(\Gamma)} \dual{1}{1}_{L^2(\Gamma)}.
\end{align*}
Since $ \dual{\mathfrak{W} (U_{\fine,\coarse}-U_\coarse)}{1}_{L^2(\Gamma)}=0$, $U_{\fine,\coarse}-U_\coarse$ has integral mean zero.
Altogether, we see  that 
\begin{align}\label{eq:star plus minus star}
\norm{U_{\coarse,\fine}-U_\coarse}{H^{1/2}(\Gamma)}^2\lesssim
\dual{g_\coarse-\mathfrak{W} U_\coarse}{(1-J_\coarse)(U_{\fine,\coarse}-U_\coarse)}_{L^2(\Gamma)}.
\end{align}
With \eqref{eq:S local bem2 imply} of Step~3 and the Cauchy-Schwarz inequality, we thus obtain that
\begin{align*}
&\norm{U_{\fine,\coarse}-U_\coarse}{H^{1/2}(\Gamma)}^2\stackrel{\eqref{eq:S local bem2 imply}}{\lesssim} \dual{ g_{\coarse}-\mathfrak{W}U_\coarse}{(1-J_\coarse)(U_{\fine,\coarse}-U_\coarse)}_{L^2(\pi_\coarse^{2p+1}(\NNr_{\coarse,\fine}))}\\
&\le \norm{ h_\coarse^{1/2}(g_\coarse-\mathfrak{W}U_\coarse)}{L^2(\pi_\coarse^{2p+1}(\NNr_{\coarse,\fine}))} 
\norm{ h_\coarse^{-1/2}((1-J_\coarse)(U_{\fine,\coarse}-U_\coarse)}{L^2(\pi_\coarse^{2p+1}(\NNr_{\coarse,\fine}))}.
\end{align*}
The equivalence \eqref{eq:tilde h equivalent} and the approximation property \eqref{eq:local approx} yield that
\begin{align*}
\norm{U_{\fine,\coarse}-U_\coarse}{H^{1/2}(\Gamma)}^2\lesssim \widetilde\eta_\coarse(\RR_{\coarse,\fine})\norm{U_{\fine,\coarse}-U_\coarse}{H^{1/2}(\Gamma)}.
\end{align*}
This concludes the proof.
The constants $\const{drel},\const{ref}$ depend only on the parametrization $\gamma$, the polynomial order $p$, and the initial mesh $\widehat\QQ_0$.

\subsection{Reliability in (\ref{eq:reliable})}
\label{sec:reliable hyper}
We only consider the case $\phi_\coarse:=\phi$ for all $\KK_\coarse\in\K$.
The other case, i.e., $\phi_\coarse:=P_\coarse\phi$ for all $\KK_\coarse\in\K$, follows analogously. 

\noindent
\textbf{Step 1:} First, we show that for arbitrary $\varepsilon>0$, there exists a refinement $\KK_\fine\in\refine(\KK_\coarse)$ such that $\norm{u-U_\fine}{H^{1/2}(\Gamma)}\le\varepsilon$.
Indeed, the C\'ea lemma proves that $\norm{u-U_\fine}{H^{1/2}(\Gamma)}\lesssim\norm{(1-J_\fine)u}{H^{1/2}(\Gamma)}$.
Note  that $u\in H^1(\Gamma)$ due of the mapping properties of $\mathfrak{W}$ and $\mathfrak{K}'$  and the assumption that  $\phi\in L^2(\Gamma)$.
Therefore, the localization argument \cite[Lemma~2.3]{faermann2d} in combination with the Sobolev-seminorm estimate \cite[Lemma~4.5]{resigabem} gives that 
\begin{align*}
\norm{(1-J_\fine)u}{H^{1/2}(\Gamma)}\lesssim \norm{h_\fine^{1/2}\partial_\Gamma (1-J_\fine)u}{L^2(\Gamma)}+\norm{h_\fine^{-1/2} (1-J_\fine) u}{L^2(\Gamma)}.
\end{align*}
Proposition~\ref{lem:Scott properties} implies that 
\begin{align*}
\norm{(1-J_\fine)u}{H^{1/2}(\Gamma)}\lesssim \norm{h_\fine^{1/2} }{L^\infty(\Gamma)} \norm{u}{H^1(\Gamma)} \to 0\quad \text{as}\quad\norm{h_\fine^{1/2} }{L^\infty(\Gamma)}\to 0.
\end{align*}

\noindent
\textbf{Step 2:} 
For $\varepsilon>0$, let $\KK_\fine$ be as in Step~1. 
The triangle inequality and discrete reliability~(E3) yield that 
\begin{align*}
\norm{u-U_\coarse}{H^{1/2}(\Gamma)}\le\norm{u-U_\fine}{H^{1/2}(\Gamma)} + \norm{U_\fine-U_\coarse}{H^{1/2}(\Gamma)}\le \varepsilon + \widetilde\eta_\coarse\lesssim\varepsilon+\eta_\coarse.
\end{align*}
For $\varepsilon\to0$, we conclude reliability \eqref{eq:reliable}.

\subsection{Efficiency in (\ref{eq:reliable})}
\label{sec:efficient hyper}
Clearly, it suffices to bound the residual part ${\rm res}_\coarse$ of the estimator $\eta_\coarse$ by $(\norm{h_\coarse^{1/2} \partial_\Gamma( u - U_\coarse )}{L^2(\Gamma)}^2
 + \norm{h_\coarse^{1/2} ( \phi - \phi_\coarse )}{L^2(\Gamma)}^2)^{1/2}$.
 To do so, we use the triangle inequality 
\begin{align}\label{eq:aux efficiency}
{\rm res}_\coarse
\le
\norm{h_\coarse^{1/2}(1/2-\mathfrak{K}')(\phi-\phi_\coarse)}{L^2(\Gamma)}
+
\norm{h_\coarse^{1/2}(1/2-\mathfrak{K}')\phi-\mathfrak{W}U_\coarse}{L^2(\Gamma)}
\end{align}
and bound each of the two terms separately. 
To control the first one, we  apply the inverse inequality \eqref{eq:inverse Kpr} and the approximation property \cite[Theorem~4.1]{cp06}
\begin{align*}
\norm{h_\coarse^{1/2}(1/2-\mathfrak{K}')(\phi-\phi_\coarse)}{L^2(\Gamma)}
&\stackrel{\eqref{eq:inverse Kpr}}\lesssim 
\norm{\phi-\phi_\coarse}{H^{-1/2}(\Gamma)} + \norm{h_\coarse^{1/2}(\phi-\phi_\coarse)}{L^2(\Gamma)}
\\
&\stackrel{\text{\cite{cp06}}}\lesssim
 \norm{h_\coarse^{1/2}(\phi-\phi_\coarse)}{L^2(\Gamma)}.
\end{align*}
For the second term in \eqref{eq:aux efficiency}, we use the inverse inequality~\eqref{eq:inverse W}
\begin{align*}
\norm{h_\coarse^{1/2}(1/2-\mathfrak{K}')\phi-\mathfrak{W}U_\coarse}{L^2(\Gamma)}
&\,\,\,=\,\,\,
\norm{h_\coarse^{1/2}\mathfrak{W}(u-U_\coarse)}{L^2(\Gamma)}
\\
&\stackrel{\eqref{eq:inverse W}}\lesssim
\norm{u-U_\coarse}{H^{1/2}(\Gamma)}
+
\norm{h_\coarse^{1/2} \partial_\Gamma(u-U_\coarse)}{L^2(\Gamma)}. 
\end{align*}
Altogether, it only remains to estimate the term $\norm{u-U_\coarse}{H^{1/2}(\Gamma)}$. 
To this end, we denote the Galerkin projection onto $\XX_\coarse$ by $G_h:H^{1/2}(\Gamma)\to\XX_\coarse$ and note that $(1-G_\coarse)=(1-G_\coarse)(1-J_\coarse)(1-G_\coarse)$.
Then, stability of $G_\coarse$ and the  approximation property~\eqref{eq:local approx2} prove that
\begin{align*}
\norm{u-U_\coarse}{H^{1/2}(\Gamma)}&\,\,\,= \,\,\,
\norm{(1-G_\coarse)(1-J_\coarse)(1-G_\coarse)u}{H^{1/2}(\Gamma)}
\lesssim \norm{(1-J_\coarse)(1-G_\coarse)u}{H^{1/2}(\Gamma)}
\\
&\stackrel{\eqref{eq:local approx2}}\lesssim \norm{h_\coarse^{1/2} (1-G_\coarse) u}{L^2(\Gamma)}=
\norm{h_\coarse^{1/2} \partial_\Gamma(u-U_\coarse)}{L^2(\Gamma)}. 
\end{align*}

\subsection{General quasi-orthogonality (E4)}
\label{subsec:hypsing orthogonality}
For sufficiently small $\vartheta>0$, we prove general quasi-orthogonality in three steps.
For $\KK_\coarse,\KK_\star\in\K$, we define $\KK_{\coarse\cap\star}\in\K$ via $\NN_{\coarse\cap\star}:=\NN_\coarse\cap\NN_\star$ and $\#_{\coarse\cap\star} z:=\min\{\#_{\coarse}z,\#_\star z\}$ for all $z\in\NN_\coarse\cap\NN_\star$.

\noindent
\textbf{Step 1:}
First, we prove some kind of  discrete reliability of $\mu$:
There exists a constant $C_{\rm drel}^-\ge1$ 
such that %
\begin{align}\label{eq:drel}
\norm{U_{k}-U_{k\cap (k+1)}}{H^{1/2}(\Gamma)}^{2}\le C_{\rm drel}^-\, \mu_k(\MM_k^-)^{2} \quad \text{for all }k\in\N_0. 
\end{align}
To see this, we note that $U_{k\cap (k+1)}\in\XX_k\cap\XX_{k+1}$ is also the Galerkin projection of   $U_k$.
Hence, the C\'ea lemma and  the inverse estimate \eqref{eq:discrete invest2} yield that
\begin{align*}
\norm{U_{k}-U_{k\cap (k+1)}}{H^{1/2}(\Gamma)}^2&\stackrel{\text{C\'ea}}\lesssim \norm{(1-J_{k\cap (k+1)})U_{k}}{H^{1/2}(\Gamma)}^2\stackrel{\eqref{eq:discrete invest2}}\lesssim \norm{h_{k\cap (k+1)}^{-1/2}(1-J_{k\cap (k+1)})U_{k}}{L^2(\Gamma)}^2.
\end{align*}
Note that $h_{k}=h_{k\cap (k+1)}$ and $\pi_{k}(\cdot)=\pi_{k\cap (k+1)}(\cdot)$.
Further, Lemma~\ref{lem:properties for B-splines} \eqref{item:spline basis} shows for all $Q\in\QQ_k$ that $U_k|_{\pi_k^p(Q)}\in\XX_{k\cap (k+1)}|_{\pi_k^p(Q)}$ if $\pi_k^p(Q)\cap\MM_k^-=\emptyset$.
Thus, the local projection property \eqref{eq:local projection} yields that
\begin{align*}
\norm{h_{k\cap (k+1)}^{-1/2}(1-J_{k\cap (k+1)})U_{k}}{L^2(\Gamma)}^2=\norm{h_{k}^{-1/2}(1-J_{k\cap (k+1)})U_{k}}{L^2(\pi_{k}^{p+1}(\MM_{k}^-))}^2.
\end{align*}
Note that $\XX_{k\ominus1}\subseteq\XX_{k\cap (k+1)}$.
Together with the  projection property \eqref{eq:local projection} and the local $L^2$-stability \eqref{eq:local L2}, the triangle inequality implies that
\begin{align*}
&\norm{h_{k}^{-1/2}(1-J_{k\cap (k+1)})U_{k}}{L^2(\pi_{k}^{p+1}(\MM_{k}^-))}
\le  \norm{h_{k}^{-1/2}(1-J_{k\cap (k+1)}J_{k-\ominus 1})U_{k}}{L^2(\pi_{k}^{p+1}(\MM_{k}^-))} 
\\&\hspace{7cm}+\norm{h_{k}^{-1/2}J_{k\cap(k+1)}(1-J_{\ominus1})U_{k}}{L^2(\pi_{k}^{p+1}(\MM_{k}^-))}^2
\\
&\stackrel{\eqref{eq:local projection}+\eqref{eq:local L2}}\lesssim \norm{h_{k}^{-1/2}(1-J_{k\ominus 1})U_{k}}{L^2(\pi_{k}^{2p+1}(\MM_{k}^-))}\le \mu_k(\MM_{k}^-).
\end{align*}
The constant $\const{drel}^-$ in \eqref{eq:drel} depends only on the parametrization $\gamma$, the polynomial order $p$, and the initial mesh $\widehat\QQ_0$.


\noindent\textbf{Step 2:}
Next, we prove 
the existence of   some constant $C_{\rm mon}'\ge 1$ such that
\begin{align}\label{eq:qquasi monotonicity}
\widetilde\eta_{k+1}^{\,2}\le C_{\rm mon}'\, \widetilde\eta_k^{\,2}\quad\text{for all }k\in\N_0.
\end{align}
By stability (E1) and reduction (E2), we have that
\begin{align}\label{eq:monoton help}
\widetilde\eta_{k+1}^{\,2}\lesssim \widetilde\eta_k^{\,2}+\norm{U_{k+1}-U_k}{H^{1/2}(\Gamma)}^2+\norm{\phi_{k+1}-\phi_k}{H^{-1/2}(\Gamma)}^2.
\end{align}
To estimate $\norm{U_{k+1}-U_k}{H^{1/2}(\Gamma)}^2$ of \eqref{eq:monoton help}, we use  ellipticity,  Galerkin orthogonality, and Young's inequality
\begin{align}
&\norm{U_{k+1}-U_k}{H^{1/2}(\Gamma)}^2\simeq\norm{U_{k+1}-U_k}{\mathfrak{W}}^2\lesssim 
\norm{U_{k+1}-U_{k\cap (k+1)}}{\mathfrak{W}}^2+
\norm{U_k-U_{k\cap (k+1)}}{\mathfrak{W}}^2\notag\\
\label{eq:k+1-k}
&\quad=
\big(\norm{u-U_{k\cap (k+1)}}{\mathfrak{W}}^2
-\norm{u-U_{k+1}}{\mathfrak{W}}^2\big)
+\norm{U_k-U_{k\cap (k+1)}}{\mathfrak{W}}^2\\
&\quad\le 2\norm{u-U_k}{\mathfrak{W}}^2+3 \norm{U_k-U_{k\cap (k+1)}}{\mathfrak{W}}^2.
\stackrel{\eqref{eq:drel}}
\lesssim \norm{u-U_k}{\mathfrak{W}}^2+\const{drel}^-\,\mu_k(\MM_k^-)^2
\stackrel{\eqref{eq:reliable}+\rm(v)}\lesssim \eta_k^2
\stackrel{\eqref{eq:abstract equivalence}}\simeq\widetilde\eta_k^{\,2}.\notag
\end{align}
To estimate $\norm{\phi_{k+1}-\phi_k}{H^{-1/2}(\Gamma)}^2$ of \eqref{eq:monoton help}, we note that (although $\XX_k$ and $\XX_{k+1}$ are not necessarily nested) the set of (transformed) $\QQ_k$-piecewise polynomials of degree $p$ is a subset of the set of (transformed) $\QQ_{k+1}$-piecewise polynomials of degree $p$.
Hence, \eqref{eq:proj proj} gives that
\begin{align}\label{eq:E4 E3 connect}
\begin{split}
\norm{\phi_{k+1}-\phi_k}{H^{-1/2}(\Gamma)} 
\stackrel{\eqref{eq:proj proj}}\lesssim
 \norm{h_k^{1/2}(\phi- \phi_k)}{L^2(\bigcup (\QQ_k\setminus\QQ_{k+1}))}
\le {\rm osc}_k \stackrel{\eqref{eq:tilde h equivalent}}\simeq\widetilde{\rm osc}_k.
\end{split}
\end{align}
The constant $C_{\rm mon}'$ depends only on the parametrization $\gamma$, the polynomial order $p$,  the initial mesh $\widehat\QQ_0$, and an arbitrary but fixed upper bound for the parameter $\vartheta$.
\\
\textbf{Step 3:} We finally come to (E4) itself.
With \eqref{eq:k+1-k},  Galerkin orthogonality gives that
\begin{align}
\label{eq:second Galerkin}
&\norm{U_{k+1}-U_k}{H^{1/2}(\Gamma)}^2
\stackrel{\eqref{eq:k+1-k} }\lesssim
\big(\norm{u-U_{k\cap (k+1)}}{\mathfrak{W}}^2
-\norm{u-U_{k+1}}{\mathfrak{W}}^2\big)
+\norm{U_k-U_{k\cap (k+1)}}{\mathfrak{W}}^2\\
&= \norm{u-U_{k\cap (k+1)}}{\mathfrak{W}}^2-\norm{u-U_{(k+1) \cap (k+2)}}{\mathfrak{W}}^2
+\norm{U_{k+1}-U_{(k+1) \cap (k+2)}}{\mathfrak{W}}^2+\norm{U_k-U_{k\cap (k+1)}}{\mathfrak{W}}^2.
\notag
\end{align}
We abbreviate the hidden (generic) constant by $C>0$.
With Step 1 and 2 in combination with
Algorithm~\ref{the algorithm}~(v),
the third plus the fourth term can be estimated by
\begin{align}
\begin{split}
&
\norm{U_{k+1}-U_{(k+1) \cap (k+2)}}{\mathfrak{W}}^2+\norm{U_k-U_{k\cap (k+1)}}{\mathfrak{W}}^2
\stackrel{\eqref{eq:drel}}
\le 
C_{\rm drel}^-\big(\mu_{k+1}(\MM_{k+1}^-)+\mu_k(\MM_k^-)\big)
\\
&\qquad\stackrel{{\rm (v)}}\le 
C_{\rm drel}^-\vartheta(\eta_{k+1}^2+\eta_k^2)
\stackrel{\eqref{eq:qquasi monotonicity}}\le 
 C_{\rm drel}^-C_{\rm eq}\vartheta(\widetilde\eta_{k+1}^{\,2}+\widetilde\eta_k^{\,2})
 \le
 C_{\rm drel}^- C_{\rm eq}(C_{\rm mon}'+1)\vartheta\widetilde\eta_k^{\,2}.
 \end{split}
\end{align}
Suppose that
$\vartheta>0$ is sufficiently small such that 
\begin{align}
\label{eq:E4 vartheta}
\varepsilon_{\rm qo}:=
C C_{\rm drel}^- C_{\rm eq}(C_{\rm mon}'+1)\vartheta
 <
 \sup_{\widetilde\delta>0}\frac{1-(1+\widetilde\delta)(1-(1-\rhored)\theta)}{\Cred+(2+\widetilde\delta^{\,-1})\Cstab^2}.
\end{align} 
Combining \eqref{eq:second Galerkin}--\eqref{eq:E4 vartheta}, we obtain that
\begin{align*}
\norm{U_{k+1}-U_k}{H^{1/2}(\Gamma)}^2-\varepsilon_{\mathrm{qo}}\widetilde\eta_k^{\,2}
\lesssim 
\norm{u-U_{k\cap (k+1)}}{\mathfrak{W}}^2-\norm{u-U_{(k+1) \cap (k+2)}}{\mathfrak{W}}^2.
\end{align*}
 
Together with Step 1, 
Algorithm~\ref{the algorithm}~(v) and reliability \eqref{eq:reliable}, we derive that 
\begin{align*}
&\sum_{k=\ell}^{\ell+N}\Big(\norm{U_{k+1}-U_k}{H^{1/2}(\Gamma)}^2-\varepsilon_{\mathrm{qo}}\widetilde\eta_k^{\,2}\Big)
\lesssim 
\sum_{k=\ell}^{\ell+N} \Big(\norm{u-U_{k\cap (k+1)}}{\mathfrak{W}}^2-\norm{u-U_{(k+1) \cap (k+2)}}{\mathfrak{W}}^2\Big)
\\
&\,\leq\, \norm{u-U_{\ell\cap (\ell+1)}}{\mathfrak{W}}^2\lesssim  \norm{u-U_\ell}{\mathfrak{W}}^2+\norm{U_\ell-U_{\ell\cap (\ell+1)}}{\mathfrak{W}}^2
\stackrel{\eqref{eq:drel}}\lesssim \norm{u-U_\ell}{\mathfrak{W}}^2+\mu_k(\MM_k^-)^{2} 
\\
&\stackrel{{\rm (v)}}\lesssim 
\norm{u-U_\ell}{\mathfrak{W}}^2+\vartheta\widetilde\eta_\ell^{\,2}
\stackrel{\eqref{eq:abstract equivalence}}\lesssim
\widetilde\eta_\ell^{\,2}
\stackrel{\eqref{eq:reliable}}\simeq
\eta_\ell^2.
\end{align*} 

It remains to estimate the sum  $\sum_{k=\ell}^{\ell+N}\norm{\phi_{k+1}-\phi_k}{H^{-1/2}(\Gamma)}^2$. 
To this end, we note that $h_{k+1}\le q h_k$ on $\bigcup(\QQ_k\setminus\QQ_{k+1})$ for some constant $0<q<1$ that depends only on $\gamma$ and $\widehat\QQ_0$.
 This yields that  $0\le(1-q)h_k \,\chi_{\bigcup\QQ_k\setminus\QQ_{k+1}}\le h_k-h_{k+1}$.
With \eqref{eq:E4 E3 connect} and the best approximation property of $P_{k+1}$, we thus derive that
\begin{align*}
&\norm{\phi_{k+1}-\phi_k}{H^{-1/2}(\Gamma)}^2
\stackrel{\eqref{eq:E4 E3 connect}}\lesssim 
\norm{h_k^{1/2}(\phi-\phi_k)}{L^{2}(\bigcup\QQ_k\setminus\QQ_{k+1})}^2
\lesssim 
\norm{(h_k-h_{k+1})^{1/2}(\phi-\phi_k)}{L^2(\Gamma)}^2\\
&\qquad\le\norm{h_k^{1/2}(\phi-\phi_k)}{L^2(\Gamma)}^2-\norm{h_{k+1}^{1/2}(\phi-\phi_{k+1})}{L^2(\Gamma)}^2={\rm osc}_k^2-{\rm osc}_{k+1}^2.
\end{align*}
In particular, we see that
\begin{align*}
\sum_{k=\ell}^{\ell+N}\norm{\phi_{k+1}-\phi_k}{H^{-1/2}(\Gamma)}^2\lesssim \sum_{k=\ell}^{\ell+N} ({\rm osc}_k^2-{\rm osc}_{k+1}^2)\le {\rm osc}_\ell^2\stackrel{\eqref{eq:tilde h equivalent}}\simeq \widetilde{\rm osc}_\ell^{\,2}\le\widetilde\eta_\ell^{\,2}.
\end{align*}

\subsection{Son estimate (R1)}
\label{sec:hypsing son}
According to Algorithm~\ref{alg:refinement}, any marked node of $\MM_\ell\subseteq\NN_\ell$ leads to at most two additional knots (if the marked node already has full multiplicity).
Since the generation of $\KK_{\ell+1}$ in Algorithm~\ref{the algorithm} is based on Algorithm~\ref{alg:refinement} with the additional possibility of multiplicity decrease, 
this yields (R1) even with the explicit constant $\const{son}=3$.

\subsection{Closure estimate (R2) and overlay property (R3)}
\label{sec:hypsing closure}
The proofs are already found in \cite[Proposition~2.2]{optigabem} for the weakly-singular case without the possibility of knot multiplicity decrease and can  immediately be extended to the current situation.

\subsection{Linear convergence (\ref{eq:R-linear})}
\label{sec:linear convergence}
In this section, we first prove that (E1)--(E2) imply estimator reduction of $\widetilde\eta$ in the sense that there exist $0<\ro{est}<1$ and $\const{est}>0$ such that
\begin{align}\label{eq:estimator reduction}
\widetilde\eta_{\ell+1}^{\,2}\le \ro{est}\,\widetilde\eta_\ell^{\,2}+\const{est} \varrho_{\ell,\ell+1}^2\quad\text{for all }\ell\in\N_0,
\end{align}
where $\ro{est}=(1+\delta)(1-(1-\ro{red})\widetilde\theta)$ with $\widetilde\theta:=\const{eq}^{-2}\theta$ and $\const{est}=\const{red}+(1+\delta^{-1})\const{stab}^2$ for all sufficiently small $\delta>0$ with $\ro{est}<1$. 
The critical observation is that Algorithm~\ref{the algorithm} implies that 
$\MM_\ell\subseteq \NNr_{\ell,\ell+1}$, where $\MM_\ell$ satisfies the D\"orfler marking $\theta\eta_\ell^2\le\eta_\ell(\MM_\ell)^2$ and thus $\widetilde\theta\widetilde\eta_\ell^{\,2}\le\widetilde\eta_\ell(\MM_\ell)^{\,2}$ due to the equivalence \eqref{eq:abstract equivalence}.
With this, the proof follows along the lines of \cite[Section~4.3]{axioms}.
We split the estimator, apply the Young inequality in combination with stability (E1) and reduction (E2) to see  that, for all $\delta>0$, 
\begin{align*}
\widetilde\eta_{\ell+1}^{\,2}&=\widetilde\eta_{\ell+1}(\NNnr_{\ell,\ell+1})^{\,2}+\widetilde\eta_{\ell+1}(\NNr_{\ell+1,\ell})^{\,2}\\
&\le (1+\delta)\widetilde\eta_{\ell}(\NNnr_{\ell,\ell+1})^{\,2}
+\ro{red}\widetilde\eta_{\ell}(\NNr_{\ell,\ell+1})^{\,2}
+\const{est}\,\varrho_{\ell,\ell+1}^2\\
&\le (1+\delta)\big(\widetilde\eta_\ell^{\,2}-(1-\ro{red})\big) \,\widetilde\eta_\ell(\NNnr_{\ell,\ell+1})^{\,2}
+\const{est}\varrho_{\ell,\ell+1}^2\\
&\le (1+\delta)(1-(1-\ro{red})\widetilde\theta)\,\widetilde\eta_\ell^{\,2}+\const{est}\,\varrho_{\ell,\ell+1}^2, 
\end{align*}
which concludes estimator reduction \eqref{eq:estimator reduction}. 
According to \cite[Proposition~4.10]{axioms}, this together with general quasi-orthogonality~(E4) yields linear convergence of $\widetilde\eta$ and thus also of $\eta$ due to the equivalence~\eqref{eq:abstract equivalence}.

\subsection{Optimal convergence (\ref{eq:optimal})}
\label{sec:optimal convergence}
We start with the following proposition, which states that D\"orfler marking is not only sufficient for linear convergence, but in some sense even necessary.
For standard element-based adaptive algorithms, it is proved, e.g., in \cite[Proposition~4.12]{axioms}.
We note that the proof follows essentially along the same lines and is only given for the sake of completeness.

 \begin{proposition}\label{prop:Doerfler optimal}
Suppose stability {\rm (E1)} and discrete reliability {\rm (E3)}.
Let $\KK_\coarse\in\K$ and $\KK_\fine\in\refine(\KK_\coarse)$.
Then, for all $0<\widetilde\theta<\widetilde\theta_{\rm opt}:=(1+\const{stab}^2\const{drel}^2)^{-1}$, there exists some $0<\ro{\widetilde\theta}<1$ such that 
\begin{align}\label{eq:Doerfler optimal}
\widetilde\eta_\fine^{\,2}\le \ro{\widetilde\theta}\,\widetilde\eta_\coarse^{\,2}
\quad \Longrightarrow \quad \widetilde\theta\,\widetilde\eta_\coarse^{\,2}\le \widetilde\eta_\coarse(\RR_{\coarse,\fine})^{\,2}.
\end{align}
The constant $\ro{\widetilde\theta}$ depends only on $\const{stab}, \const{drel}$, $\widetilde\theta$.
 \end{proposition}
 \begin{proof}
 Throughout the proof, we work with a free variable $\ro{\widetilde\theta}>0$, which will be fixed at the end.
For all $\delta>0$, the Young inequality together with stability (E1) shows  that
 \begin{align*}
 \widetilde\eta_\coarse^{\,2}=\widetilde\eta_\coarse(\NNr_{\coarse,\fine})^2 +\widetilde\eta_\coarse(\NNnr_{\coarse,\fine})^2\le \widetilde\eta_\coarse(\NNr_{\coarse,\fine})^2+(1+\delta^{-1}) \widetilde\eta_\fine(\NNnr_{\coarse,\fine})^2+(1+\delta)\Cstab^2\,\varrho_{\coarse,\fine}^2.
 \end{align*}
With $\RR_{\coarse,\fine}\supseteq\NNr_{\coarse,\fine}$, we get for the first term on the right-hand side that 
$\widetilde\eta_\coarse(\NNr_{\coarse,\fine})^2\le \widetilde\eta_\coarse(\RR_{\coarse,\fine})^2$.
The assumption  \eqref{eq:Doerfler optimal} proves that $\widetilde\eta_\fine(\NNnr_{\coarse,\fine})^2\le \widetilde\eta_\fine^{\,2}\le \ro{\widetilde\theta}\,\widetilde\eta_\coarse^{\,2}$.
Together with discrete reliability (E3), we obtain that 
\begin{align*}
\widetilde\eta_\coarse^{\,2}\le \widetilde\eta_\coarse(\RR_{\coarse,\fine})^2 +(1+\delta^{-1}) \ro{\widetilde\theta}\,\widetilde\eta_\coarse^{\,2}+(1+\delta)\Cstab^2\const{drel}^2\widetilde\eta_\coarse(\RR_{\coarse,\fine})^2.
\end{align*}
Put differently, we end up with
\begin{align*}
\frac{1-(1+\delta^{-1})\ro{\widetilde\theta}}{1+(1+\delta)\const{stab}^2\const{drel}^2}\widetilde\eta_\coarse^{\,2}\le \widetilde\eta_\coarse(\RR_{\coarse,\fine})^2.
\end{align*}
Finally, we choose $\delta>0$ and then $0<\ro{\widetilde\theta}<1$ such that 
\begin{align*}
\widetilde\theta\le \frac{1-(1+\delta^{-1})\ro{\widetilde\theta}}{1+(1+\delta)\const{stab}^2\const{stab}^2\const{drel}^2}<\frac{1}{1+\const{drel}^2}=\widetilde\theta_{\rm opt}.
\end{align*}
This concludes the proof.
\end{proof}

In the following lemma, we show that the estimator is monotone up to some multiplicative constant.
Again, the proof follows along the lines of the version from \cite[Lemma~3.5]{axioms}. 
 \begin{lemma}\label{lem:quasi-monotonicity}
Suppose {\rm (E1)--(E3)}, where the restriction $\ro{red}<1$ is not necessary.
Then, there exists a constant $\const{mon}\ge1$ such that there holds quasi-monotonicity in the sense that 
\begin{align}\label{eq:quasi-monotonicity}
\widetilde\eta_\fine^{\,2}\le \const{mon} \widetilde\eta_\coarse^{\,2} \quad\text{for all } \KK_\coarse\in\K, \KK_\fine\in\refine(\KK_\coarse).
\end{align}
The constant $\const{mon}$ depends only on $\const{stab},\const{red}$, $\ro{red}$, and $\const{drel}$.
 \end{lemma}
 \begin{proof}
 We split the estimator and apply Young's inequality in combination with (E1)--(E2).
For all $\delta>0$, we see that
 \begin{align*}
 \widetilde\eta_\fine^{\,2}=\widetilde\eta_\fine(\NNnr_{\coarse,\fine})^2+\widetilde\eta_\fine(\NNr_{\fine,\coarse})^2&\le (1+\delta)\widetilde\eta_\coarse(\NNnr_{\coarse,\fine})^2+\ro{red}\widetilde\eta_\coarse(\NNr_{\coarse,\fine})^2+(\const{stab}^2+\const{red}(1+\delta^{-1}))\varrho_{\coarse,\fine}^2\\
 &\le \max\{1+\delta,\ro{red}\} \widetilde\eta_\coarse^{\,2}+(\const{stab}^2+\const{red}(1+\delta^{-1}))\varrho_{\coarse,\fine}^2.
 \end{align*}
 The application of (E3) yields that 
 \begin{align*}
 \widetilde\eta_\fine^{\,2}&\le  \max\{1+\delta,\ro{red}\}\widetilde\eta_\coarse^{\,2}+(\const{stab}^2+\const{red}(1+\delta^{-1}))\const{drel}^2\widetilde\eta_\coarse(\RR_{\coarse,\fine})^2\\
 &\le\big(\max\{1+\delta,\ro{red}\}+(\const{stab}^2+\const{red}(1+\delta^{-1}))\const{drel}^2\big)\widetilde\eta_\coarse^{\,2}.
 \end{align*}
This concludes the proof.
 \end{proof}

 The next lemma provides the key ingredient for the proof of optimal convergence rates.
 Again, the proof follows along the lines of \cite[Lemma~4.14]{axioms}. 

 \begin{lemma}\label{lem:optimality}
Suppose the overlay property {\rm (R3)} and quasi-monotonicity \eqref{eq:quasi-monotonicity}.  
Let $\ell\in\N_0$ such that $\widetilde\eta_\ell>0$ and let $0<\ro{}<1$. 
Let $s>0$ with $\norm{u}{\widetilde{\mathbb{A}}_s}:=\sup_{N\in\N_0} \big( (N+1)^s\inf_{\KK_\coarse\in\K(N)} \widetilde\eta_\coarse\big)<\infty$.
	Then, there exists a refinement $\KK_{\fine}\in\refine(\KK_\ell)$ with
	\begin{subequations}\label{eq:lemoptresult}
		\begin{align}
		\label{eq:lemoptresult1}
		\widetilde\eta_{\fine}^{\,2} &\leq  \ro{} \,\widetilde\eta_\ell^{\,2},
		\\
				\label{eq:lemoptresult2}
		\#_\fine\NN_\fine-\#_\ell\NN_\ell &<\const{mon}^{1/(2s)}\norm{u}{\widetilde{\mathbb{A}}_s}^{1/s}\,\ro{}^{-1/(2s)}\widetilde\eta_\ell^{\,-1/s}.
		\end{align}
		\end{subequations} 
\end{lemma}

\begin{proof}
We prove the assertion in two steps.

\noindent\textbf{Step 1:} 
We show  a modified \eqref{eq:lemoptresult} for some $\KK_\coarse\in\K$ instead of a refinement $\KK_\fine\in\refine(\KK_\coarse)$, i.e., we prove 
that
\begin{subequations}
		\begin{align}		
		\label{eq:lemoptresult3}
		\widetilde\eta_{\coarse}^{\,2} &\leq  {(\ro{}}/\const{mon}) \,\widetilde\eta_\ell^{\,2},
		\\
		\label{eq:lemoptresult4}
		\#_\coarse\NN_{\coarse}-\#_0\NN_0 &<\norm{u}{\widetilde{\mathbb{A}}_s}^{1/s}\,(\ro{}/\const{mon})^{-1/(2s)}\widetilde\eta_\ell^{\,-1/s}.
		\end{align}
				\end{subequations}
Let $N\in\N_0$ be minimal such that $\norm{u}{\widetilde{\mathbb{A}}_s}(N+1)^{-s} \le {(\ro{}/\const{mon})}^{\,1/2}\widetilde\eta_\ell$.
	Note that $N>0$ by the fact that $\widetilde\eta_\ell\leq\const{mon}^{1/2}\widetilde\eta_0\le\const{mon}^{1/2}\norm{u}{\widetilde{\mathbb{A}}_s}$ and $0<q<1$.
	Hence, minimality of $N$ yields that
${(\ro{}/\const{mon})}^{\,1/2}\widetilde\eta_\ell<\norm{u}{\widetilde{\mathbb{A}}_s}N^{-s}$ and hence
	\begin{align}\label{eq:comparison help}
	N<\norm{u}{\widetilde{\mathbb{A}}_s}{(\ro{}/\const{mon})}^{-1/(2s)}\widetilde\eta_\ell^{\,-1/s}.
	\end{align}
	Next, we choose $\KK_\coarse\in\K(N)$ with $\widetilde\eta_\coarse=\min_{\KK_\coarse\in\K(N)}\widetilde\eta_\coarse$. 
	By definition of $\norm{u}{\widetilde{\mathbb{A}}_s}$ and the choice of $N$, this gives \eqref{eq:lemoptresult3}.
	Moreover,  \eqref{eq:lemoptresult4} follows at from \eqref{eq:comparison help}.

\noindent \textbf{Step 2:} We consider a common refinement $\KK_\fine$ of $\KK_\ell$ and $\KK_\coarse$ as in (R3).
Estimate \eqref{eq:lemoptresult3} and quasi-monotonicity~\eqref{eq:quasi-monotonicity} show \eqref{eq:lemoptresult1}.
Moreover, (R3) and 		\eqref{eq:lemoptresult4} prove that 
\begin{align*}
\#_\fine\NN_\fine-\#_\ell\NN_\ell\le \#_\coarse\NN_\coarse-\#_0\NN_0
\le 
\norm{u}{\widetilde{\mathbb{A}}_s}^{1/s}\,(\ro{}/\const{mon})^{-1/(2s)}\widetilde\eta_\ell^{\,-1/s},
\end{align*}
which is just \eqref{eq:lemoptresult2}.
	\end{proof}

We finally have the means to prove optimal convergence \eqref{eq:optimal}. 

\begin{proof}[Proof of \eqref{eq:optimal}]
We prove the assertion in two steps.

\noindent \textbf{Step 1:}
We show that $0<\theta<\theta_{\rm opt}:=\const{eq}^{-2}\widetilde\theta_{\rm opt}=\const{eq}^{-2}(1+\const{stab}^2\const{drel}^2)^{-1}$ implies that
\begin{align}
\sup_{\ell\in\N_0}{(\#_\ell\NN_\ell-\#_0\NN_0+1)^{s}}{\widetilde\eta_\ell}\le\widetilde C_{\rm opt} \norm{u}{\widetilde{\mathbb{A}}_s}
\end{align}
for some constant $\widetilde C_{\rm opt}>0$.
Clearly, with the equivalence \eqref{eq:abstract equivalence}, this immediately gives \eqref{eq:optimal}.
Without loss of generality, we assume that $\norm{u}{\widetilde{\mathbb{A}}_s}<\infty$.
If $\widetilde\eta_{\ell_0}=0$ for some $\ell_0\in\N_0$, then, 
Algorithm~\ref{the algorithm} implies that $\widetilde\eta_\ell=0$ for all $\ell\geq \ell_0$.
Moreover, $(\#_0\NN_0-\#_0\NN_0+1)^s\widetilde\eta_0\le \norm{u}{\widetilde{\mathbb{A}}_s}$ is trivially satisfied.
Thus, it is sufficient to consider $0<\ell< \ell_0$ resp.\ $0<\ell$ if no such $\ell_0$ exists.
Now, let $k< \ell$ and define $\widetilde\theta:=\const{eq}^{2}\theta$.
According to Lemma~\ref{lem:quasi-monotonicity}, we may apply Lemma~\ref{lem:optimality} for  $\KK_k$, where we choose $\ro{\widetilde\theta}$ as in Proposition~\ref{prop:Doerfler optimal}.
In particular, \eqref{eq:Doerfler optimal} in combination with \eqref{eq:lemoptresult1} shows that $\RR_{k,\fine}$ satisfies the D\"orfler marking $\widetilde\theta\widetilde\eta_k^{\,2}\le\widetilde\eta_k(\RR_{k,\fine})^2$ and hence $\theta\eta_k^2\le\eta_k(\RR_{k,\fine})^2$.
Since, $\MM_k$ is an essentially minimal set satisfying D\"orfler marking (see Remark~\ref{rem:feature} {\rm (a)}) we get that
$|\MM_k|\lesssim|\RR_{k,\fine}|$.
Since the maximal multiplicity is bounded, we see that  
\begin{align*}
\#_k\MM_k\lesssim\#_k\RR_{k,\fine}\stackrel{\rm (E3)}\lesssim
	 \#_\fine\NN_\fine-\#_k\NN_k
\stackrel{\eqref{eq:lemoptresult2}}\lesssim
\norm{u}{\widetilde{\mathbb{A}}_s}^{1/s}\,
	  \widetilde\eta_k^{\,-1/s}.
	\end{align*}
For $\ell>0$, the closure estimate (R2) proves that
	\begin{align*}
	\begin{split}
	\#_\ell\NN_\ell-\#_0\NN_0+1\le 2(\#_\ell\NN_\ell-\#_0\NN_0)
	\leq 2\const{clos} \sum_{k=0}^{\ell-1}\#_k\MM_k
	\leq 
	\norm{u}{\widetilde{\mathbb{A}}_s}^{1/s}
	\sum_{k=0}^{\ell-1}\widetilde\eta_k^{\,-1/s}.
	\end{split}
	\end{align*}
	Finally, linear convergence of $\eta_k$ \eqref{eq:R-linear} and thus of $\widetilde \eta_k$  and elementary analysis show that the term $\sum_{k=0}^{\ell-1}\widetilde\eta_k^{\,-1/s}$ can be bounded from above by $C\widetilde\eta_\ell^{\,-1/s}$ where $C>0$ depends only on $\ro{lin}, \const{lin}$, and $s$.
Therefore,	we end up with 
	\begin{align*}
	(\#_\ell\NN_\ell-\#_0\NN_0+1)^{s}\widetilde\eta_\ell\lesssim 
\norm{u}{\widetilde{\mathbb{A}}_s}\quad\text{for all }\ell>0.
	\end{align*}
For $\ell=0$, the latter estimate is trivially satisfied. This concludes the proof.

\noindent \textbf{Step 2:}
To see the lower bound in \eqref{eq:optimal}, let $N\in\N_0$ and choose the maximal $\ell\in\N_0$ such that $\#_\ell\NN_\ell-\#_0\NN_0\le N$.
Due to the maximality of $\ell$ and the son estimate (R1), we have that 
$N+1\le\#_{\ell+1}\NN_{\ell+1}-\#_0\NN_0
\le \const{son}\#_\ell\NN_\ell-\#_0\NN_0
\simeq \#_\ell\NN_\ell-\#_0\NN_0+1$,
where the hidden constants depend only on $\#_0\NN_0$.
This leads to
\begin{align*}
\inf_{\KK_\coarse\in\K(N)} (N+1)^s\eta_\coarse\lesssim (\#_\ell\NN_\ell-\#_0\NN_0+1)^s\eta_\ell
\end{align*}
and concludes the proof.

\end{proof}

\subsection{Approximability constants satisfy (\ref{eq:classes})}
\label{sec:classes}
The second inequality in \eqref{eq:classes} is trivially satisfied by definition of the approximability constants and the fact that $\K^1\cup\K^p\subseteq\K$. 

For the first inequality, we call Algorithm~\ref{the algorithm} with parameters as in Remark~\ref{rem:feature}~(c) such that only $h$-refinement takes place.
\eqref{eq:optimal} gives that
$\norm{u}{\mathbb{A}_s}$ $\simeq$ $\sup_{\ell\in\N_0}{(\#_\ell\NN_\ell-\#_0\NN_0+1)^{s}}{\eta_\ell}$. 
Since $\KK_\ell\in\K^1$ for all $\ell\in\N_0$, we can argue along the lines of Step~2 of the proof of~\eqref{eq:optimal} to see that $\norm{u}{\mathbb{A}_s^1}\lesssim\sup_{\ell\in\N_0}{(\#_\ell\NN_\ell-\#_0\NN_0+1)^{s}}{\eta_\ell}$.

For the third inequality in \eqref{eq:classes}, we note the elementary equivalence for arbitrary fixed constants $C>0$
\begin{align*}
\norm{u}{\mathbb{A}^p_s}=\sup_{N\in\N_0}\inf_{\KK_\coarse\in\K^p(N)} (N+1)^s \eta_{\coarse}\simeq\sup_{N\in\N_0}\inf_{\KK_\coarse\in\K^p(CN)}(N+1)^s \eta_{\coarse}.
\end{align*}
To conclude the proof of \eqref{eq:classes}, it thus remains to show that 
\begin{align*}
\sup_{N\in\N_0}\inf_{\KK_\coarse\in\K^p(CN)} \eta_{\coarse}\lesssim\sup_{N\in\N_0} \inf_{\KK_\coarse\in\K^1(N)} \eta_{\coarse}
\end{align*}
for some generic constant $C>0$.
Let $N\in\N_0$. To verify the latter inequality, let   $\KK_{\coarse,1}\in\K^1(N)$.
Moreover, let $\KK_{\coarse,p}$ be the corresponding knots in $\K^p$ with $\NN_{\coarse,p}=\NN_{\coarse,1}$.
Recall the initial knots $\KK_{0,p}$ with maximal multiplicity $p$ from \eqref{eq:Ks}. 
If $\KK_{\coarse,1}\neq\KK_0$ and thus  $\#_{\coarse,1}\NN_{\coarse,1}>\#_0\NN_0$,  there exists a constant $C>0$ only depending on $p$ and $|\NN_0|$, such that
\begin{align*}
\#_{\coarse,p}\NN_{\coarse,p}-\#_{0,p}\NN_{0,p}
= p(|\NN_{\coarse,1}|-|\NN_0|)
\le p( \#_{\coarse,1}\NN_{\coarse,1}-|\NN_0|)
\le C( \#_{\coarse,1}\NN_{\coarse,1}-\#_0\NN_0) \le C N,
\end{align*}
which yields that $\KK_{\coarse,1}\in\K^1(N)$. 
To conclude the proof, it is thus remains to show that $\eta_{\coarse,p}\lesssim \eta_{\coarse,1}$. 
Since $\phi_{\coarse,p}=\phi_{\coarse,1}$, we have that ${\rm osc}_{\coarse,p}={\rm osc}_{\coarse,1}$. 
For the residual term, we note that $h_{\coarse,p}=h_{\coarse,1}$ and $g_{\coarse,p}=g_{\coarse,1}$.
The triangle inequality gives that
\begin{align*}
{\rm res}_{\coarse,p}=\norm{h_{\coarse,p}^{1/2}(g_{\coarse,p}-\mathfrak{W}U_{\coarse,p})}{L^2(\Gamma)}
&\le
\norm{h_{\coarse,1}^{1/2}(g_{\coarse,1}-\mathfrak{W}U_{\coarse,1})}{L^2(\Gamma)}
+\norm{h_{\coarse,p}^{1/2}\mathfrak{W}(U_{\coarse,p}-U_{\coarse,1})}{L^2(\Gamma)}
\\
&={\rm res}_{\coarse,1}
+\norm{h_{\coarse,p}^{1/2}\mathfrak{W}(U_{\coarse,p}-U_{\coarse,1})}{L^2(\Gamma)}.
\end{align*}
To estimate the second summand, we note that $\KK_{\coarse,p}\in\refine(\KK_{\coarse,1})$.
Therefore, we can use the inverse inequalities \eqref{eq:inverse W} and \eqref{eq:discrete invest} and discrete reliability (E3) to see that
\begin{align*}
\norm{h_{\coarse,p}^{1/2}\mathfrak{W}(U_{\coarse,p}-U_{\coarse,1})}{L^2(\Gamma)}
&\stackrel{\eqref{eq:inverse W}}\lesssim
\norm{U_{\coarse,p}-U_{\coarse,1}}{H^{1/2}(\Gamma)} + \norm{h_\coarse^{1/2}\partial_\Gamma (U_{\coarse,p}-U_{\coarse,1})}{L^2(\Gamma)}\\
&\stackrel{\eqref{eq:discrete invest}}\lesssim 
\norm{U_{\coarse,p}-U_{\coarse,1}}{H^{1/2}(\Gamma)}
\stackrel{\rm (E3)}\lesssim
\eta_{\coarse,1}. 
\end{align*}
which concludes the proof.

\section{Numerical experiments}
\label{section:numerics}

In this section, we empirically investigate the performance of Algorithm~\ref{the algorithm} on the geometries $\Omega$ from Figure~\ref{fig:geometries}. 
Their boundaries $\Gamma$ can be parametrized via rational splines of degree~$2$, i.e., there exists a $2$-open knot vector $\widehat\KK_\gamma$ on $[0,1]$, and positive weights $\mathcal{W}_\gamma$ such that the components of $(\gamma_1,\gamma_2)=\gamma:[0,1]\to \Gamma$ satisfy that 
\begin{align}
\gamma_1,\gamma_2\in\widehat{\mathcal{S}}^{\,2}(\widehat\KK_\gamma,\mathcal{W}_\gamma);
\end{align}
see \cite[Section~5.9]{gantner17} for details.
On the pacman geometry, we prescribe an exact solution $P$ of the Laplace problem as 
\begin{equation}
P(x_1,x_2):=r^{\tau}\cos\left(\tau\beta\right)
\end{equation}
in polar coordinates $(x_1,x_2)=r(\cos \beta,\sin \beta)$ with $\beta\in(-\pi,\pi)$.
Similarly, we prescribe 
\begin{equation}
P(x_1,x_2):=r^{2/3}\cos\left(\frac{2}{3}(\beta+\pi/2)\right)
\end{equation}
in polar coordinates $(x_1,x_2)=r(\cos \beta,\sin \beta)$ with $\beta\in (-3\pi/2,\pi/2)$ on the heart-shaped domain.
Figure~\ref{fig:solutions} shows the corresponding Dirichlet data $u=P|_\Gamma$ as well as  Neumann data $\phi=\partial P/\partial\nu$.
In each case, the latter have a generic singularity at the origin. 
It is well-known that the boundary data satisfy the hyper-singular integral equation \eqref{eq:hyper strong} as well as the weakly-singular integral equation \eqref{eq:weak strong}; see Appendix~\ref{sec:weaksing} for details on the latter. 
In the following Sections~\ref{sec:hyper pacman}--\ref{sec:weak heart}, we aim to numerically solve these boundary integral equations. 

\begin{figure}
\label{fig:geometries}
\psfrag{Pacman}[c][c]{}
\psfrag{Heart}[c][c]{}
\centering
	\includegraphics[width=.475\textwidth,clip=true]{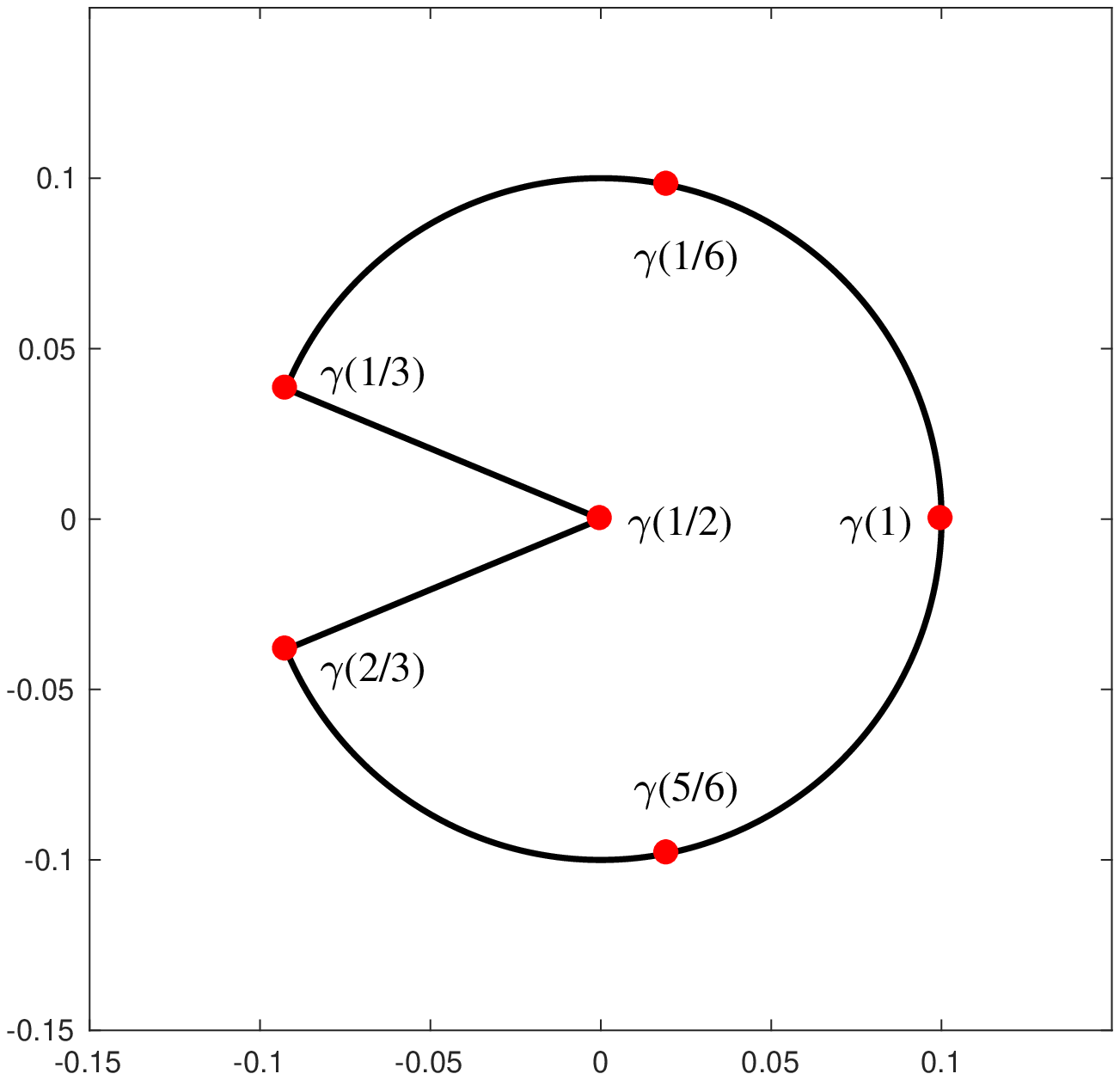}\quad
	\includegraphics[width=.475\textwidth,clip=true]{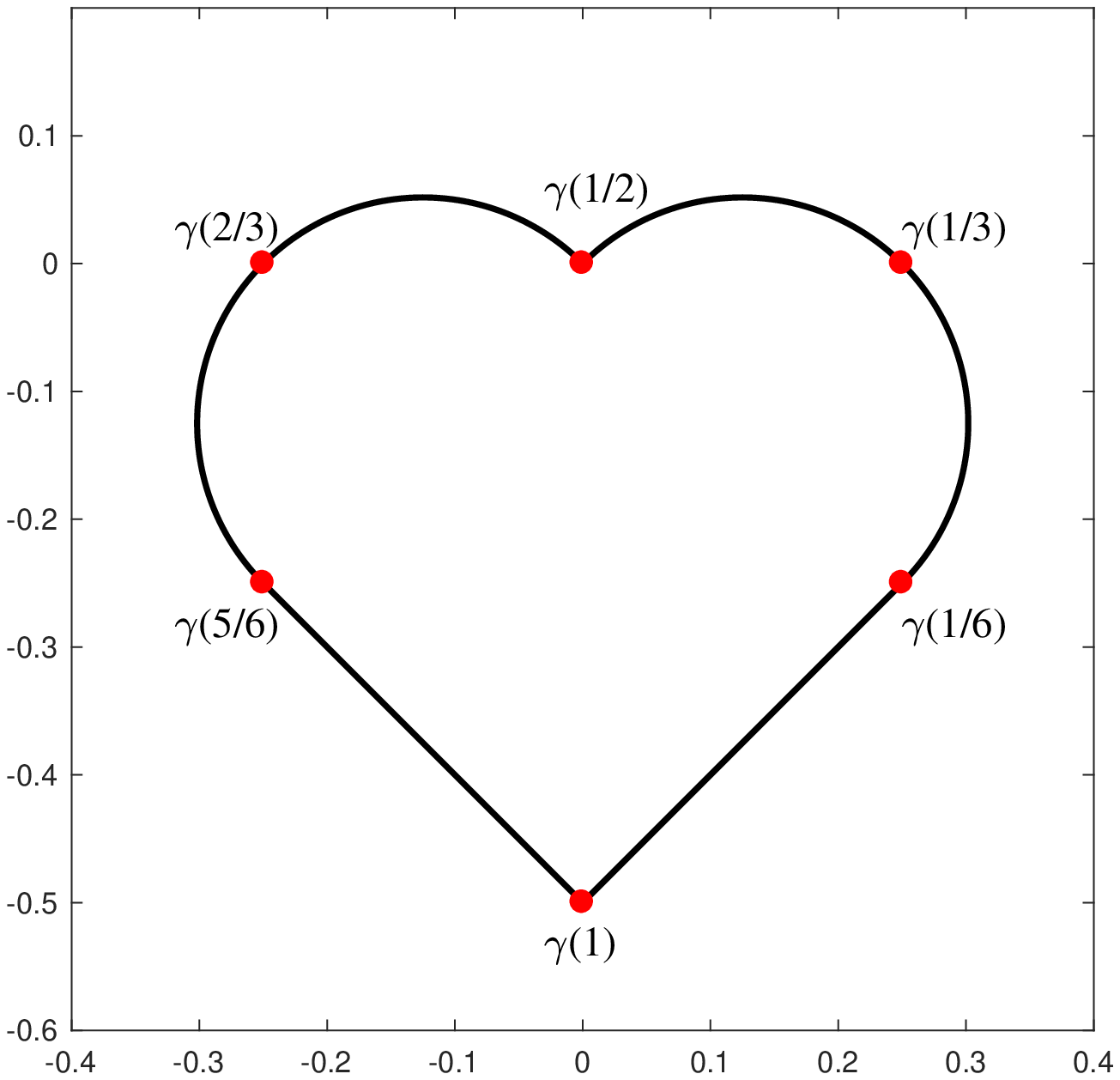}
	\caption{Geometries and initial nodes.}
\end{figure}

\begin{figure}
\label{fig:pacman solutions}
\psfrag{parameter domain}[c][c]{\small parameter domain}
\psfrag{Dirichlet data pacman}[c][c]{\small Dirichlet data $u$ on pacman}
\psfrag{Neumann data pacman}[c][c]{\small Neumann data $\phi$ on pacman}
\psfrag{Dirichlet data heart}[c][c]{\small Dirichlet data $u$ on heart}
\psfrag{Neumann data heart}[c][c]{\small Neumann data $\phi$ on heart}

\centering
	\includegraphics[width=.475\textwidth,clip=true]{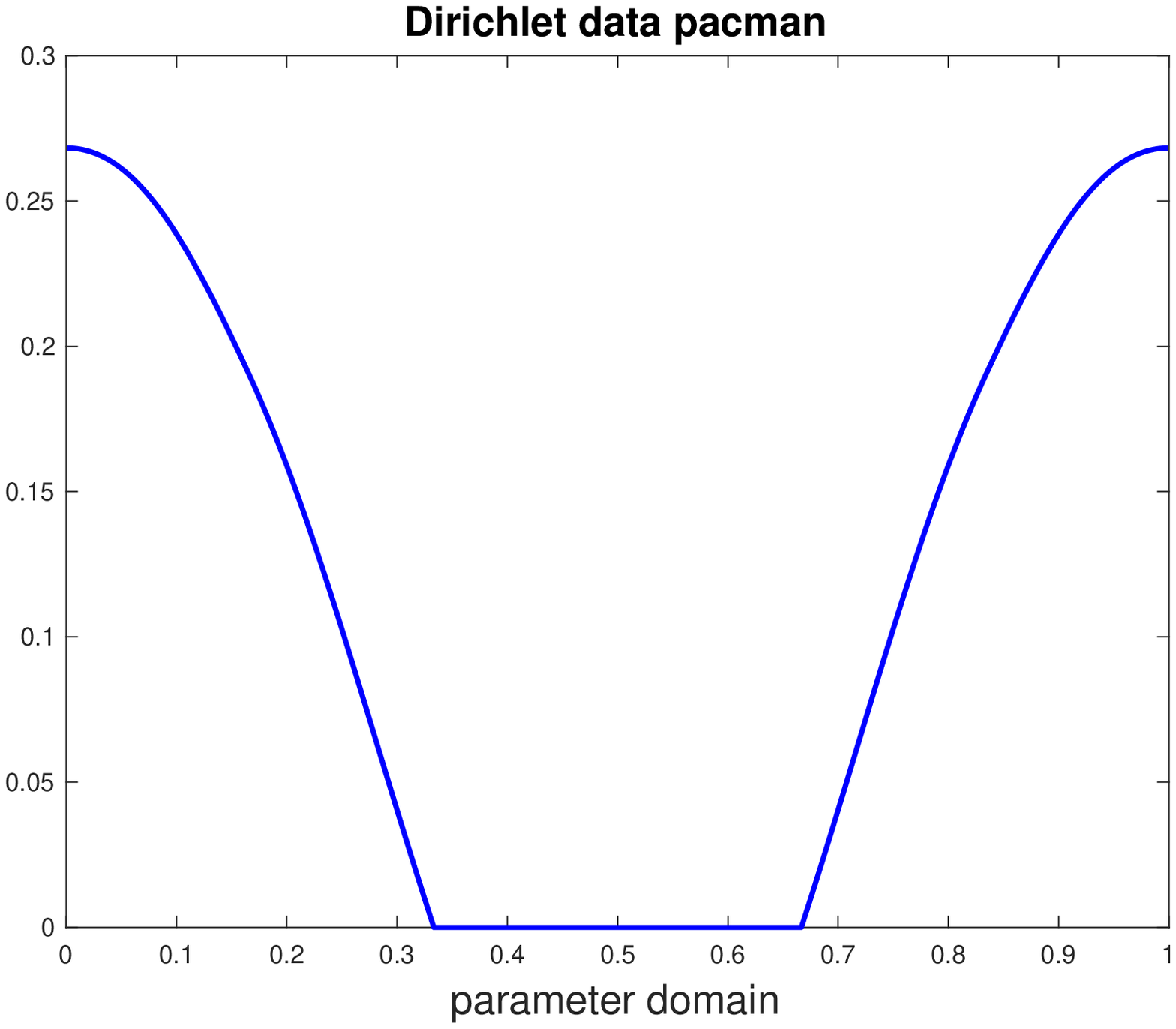}\quad
	\includegraphics[width=.475\textwidth,clip=true]{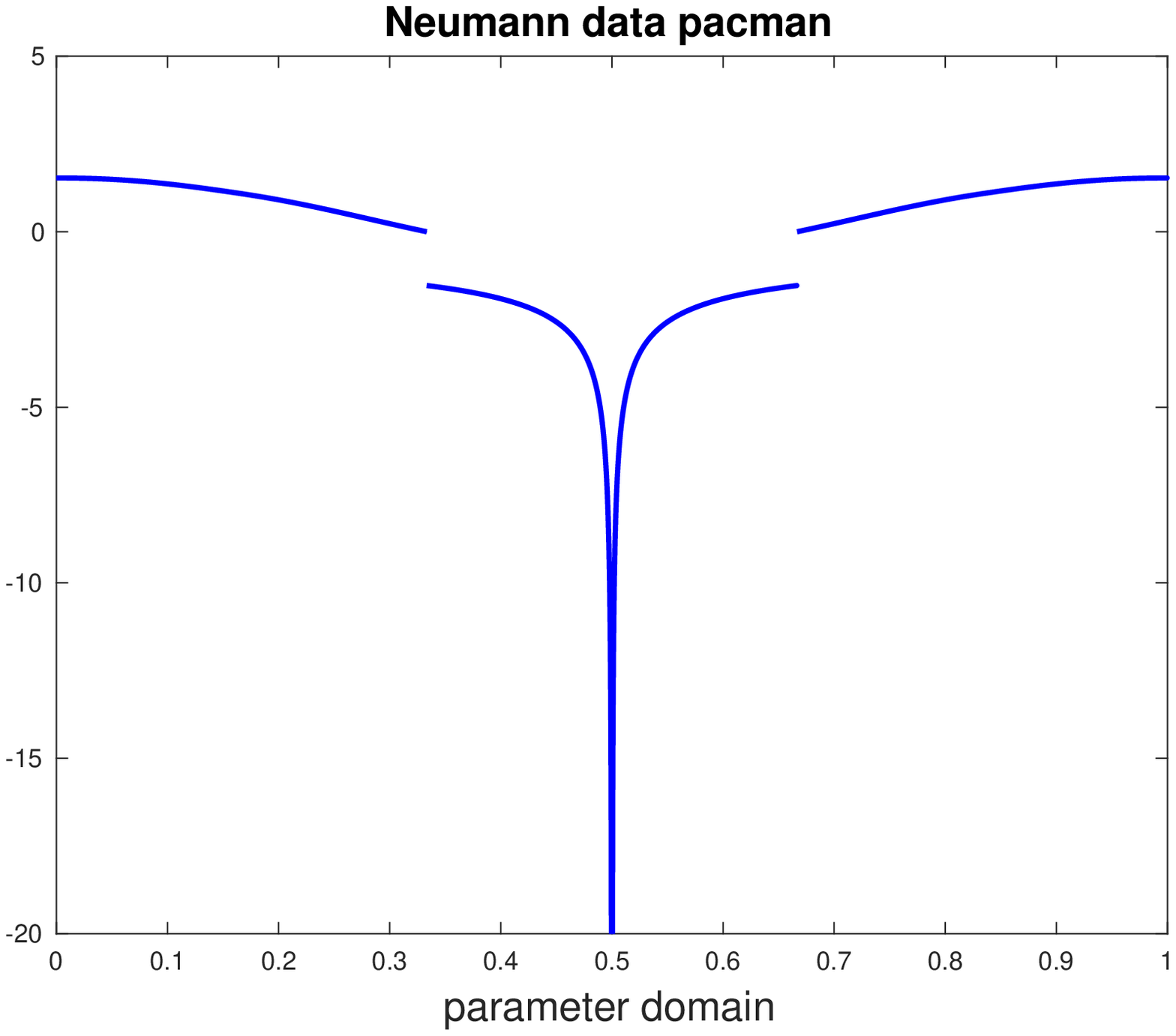}\\
	\hspace{2mm}
	
	\includegraphics[width=.475\textwidth,clip=true]{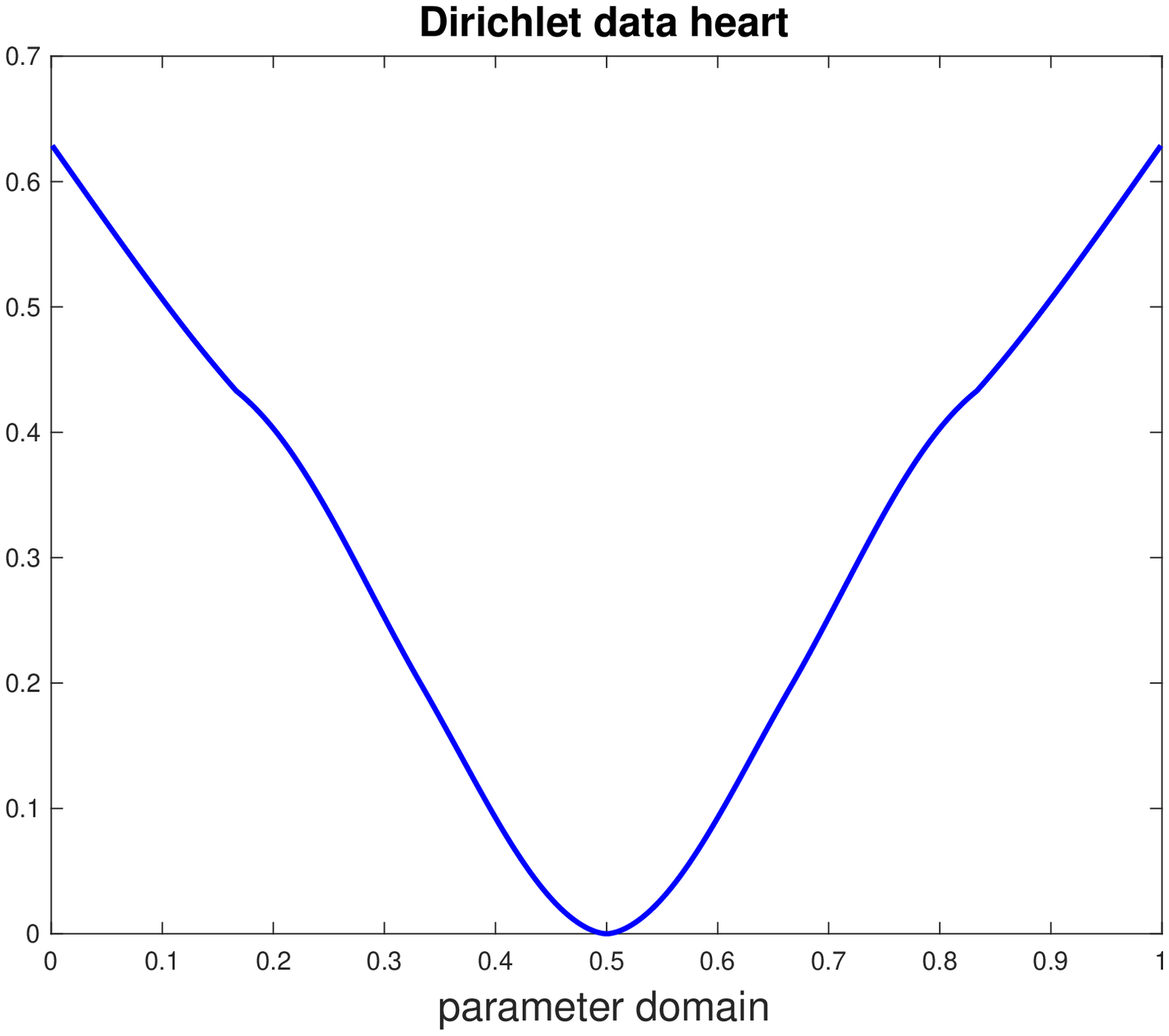}\quad
	\includegraphics[width=.475\textwidth,clip=true]{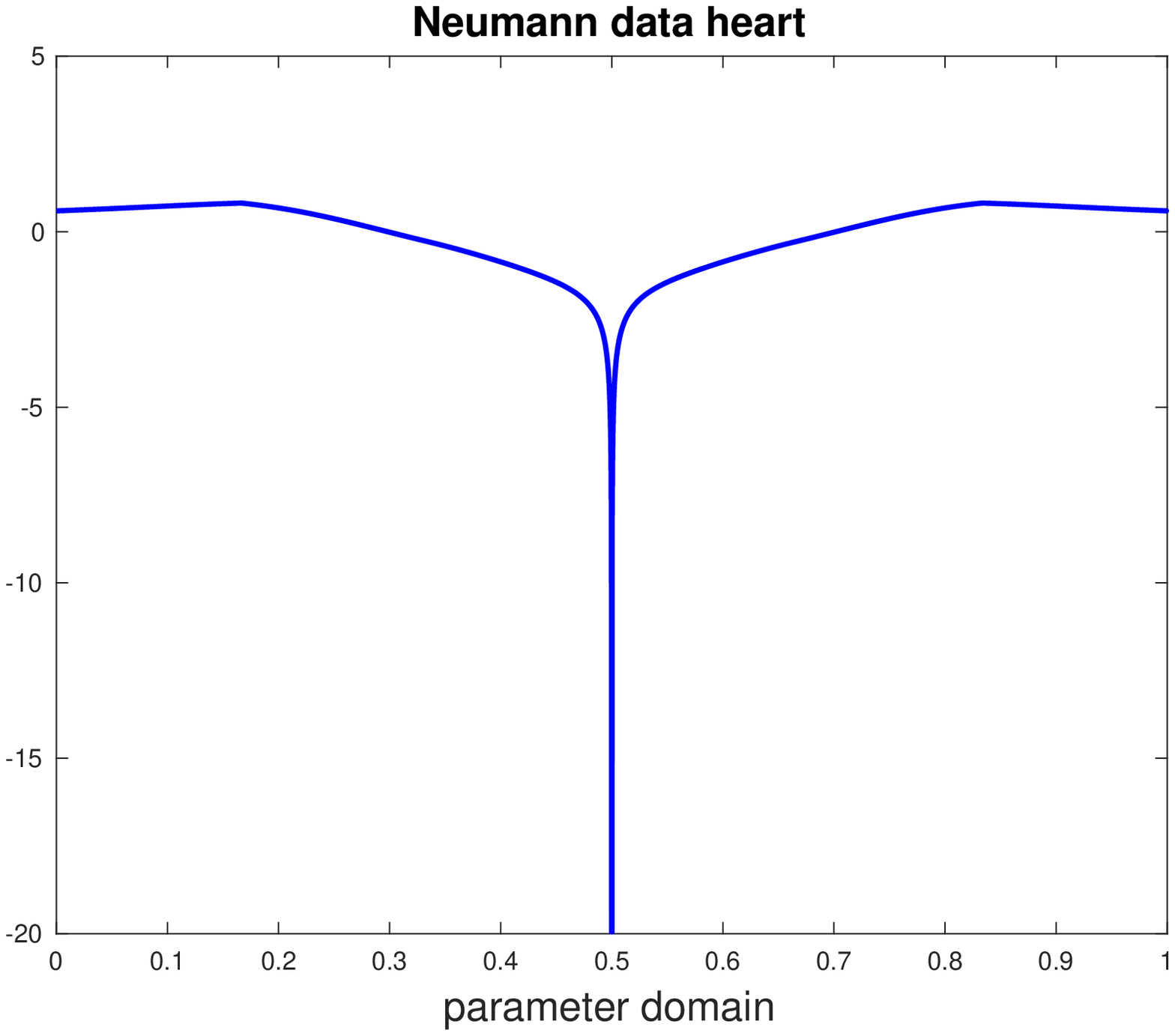}
	\caption{Dirichlet and Neumann data of the Laplace solutions $P$ from Section~\ref{section:numerics} plotted over the parameter domain of the respective parametrization $\gamma:[0,1]\to\Gamma$.}
        \label{fig:solutions}
\end{figure}

For the discretization of the boundary integral equations, we employ (transformed) splines of degree $p=2$. 
Based on the knots $\widehat\KK_\gamma$ for the geometry, we choose the initial knots $\widehat\KK_0$ for the discretization  such that the corresponding nodes coincide, i.e., $\widehat\NN_0=\widehat \NN_\gamma$.
Moreover, we assume that all interior knots of $\widehat\KK_0$ have multiplicity $1$ so that Algorithm~\ref{the algorithm} can decide, where higher knot multiplicities are required. 
In each case, this gives that
\begin{align*}
\widehat\KK_0:=\Big(0,0,0,\frac{1}{6},\frac{1}{3},\frac{1}{2},\frac{2}{3},\frac{5}{6},1,1,1\Big),\notag\\
\end{align*}
As basis for the considered  ansatz spaces, we use \eqref{eq:hypsing basis} for the hyper-singular  equation and~\eqref{eq:weak space} for the weakly-singular equation.
To (approximately) calculate the Galerkin matrix, the right-hand side vector, and the weighted-residual error  estimators, we transform the singular integrands into a sum of a smooth part and a logarithmically singular part.
Then, we use adapted Gauss quadrature to compute the resulting integrals with appropriate accuracy;
 see \cite[Section 5]{diplomarbeit} and \cite[Section~6]{schimanko} for details.
Finally, we note that for the hyper-singular case, we approximate $\phi$ by its $L^2$-orthogonal projection onto piecewise polynomials as in Section~\ref{section:igabem}.
We empirically found that such an approximation is necessary for the hyper-singular equation due to stability issues of the implementation. 
We do not apply any data approximation for the weakly-singular case. 

For each example,  we choose the parameters of Algorithm~\ref{the algorithm} resp.\ its version for the weakly-singular case of Appendix~\ref{sec:weaksing} as $\theta=0.5$, $\const{min}=1$,  $\vartheta\in\{0,0.1,1\}$, and $\const{mark}=1$.
Recall that $\vartheta=0$ prevents any multiplicity decrease.
For comparison, we also consider uniform refinement with $\theta=1$ and $\vartheta=0$, where we mark all nodes in each step, i.e., $\MM_\ell=\NN_\ell$ for all $\ell\in\N_0$. 
Note that this leads to uniform bisection  of all elements (without knot multiplicity increase).

\subsection{Hyper-singular integral equation on pacman}
\label{sec:hyper pacman}
In Figure~\ref{fig:hyper pacman}, the corresponding error estimators $\eta_\ell$ are plotted.
All values are plotted in a double logarithmic scale such that the experimental convergence rates are visible as the slope of the corresponding curves.
Since the Neumann data, which have to be resolved, lack regularity, uniform refinement regains the suboptimal rate $\mathcal{O}(N^{-4/7})$, whereas adaptive refinement leads to the optimal rate $\mathcal{O}(N^{-1/2-p})=\mathcal{O}(N^{-5/2})$.
In this example, the estimator curves look very similar for all considered $\vartheta$.
For adaptive refinement, Figure~\ref{fig:hyper pacman} additionally provides  histograms of the knots $\widehat\KK_\ell$ from the last refinement step.
Moreover, all knots with higher multiplicity than one are marked with crosses. 
Note that the exact solution $u\circ\gamma$ on the parameter domain (depicted in Figure~\ref{fig:solutions}) is only $C^0$ at $1/3$ and $2/3$.
We see that $\vartheta=0$ leads to a great amount of unnecessary multiplicity increases. 
In contrast to this, $\vartheta\in\{0.1,1\}$ can reduce them immensely.
In particular, the latter choices give a much more accurate information on the regularity of the solution.


\begin{figure}
\label{fig:hyper pacman}
\psfrag{number of knots N}[c][c]{\small number of knots $N$}
\psfrag{number of knots}[c][c]{\small number of knots}
\psfrag{parameter domain}[c][c]{\small parameter domain}
\psfrag{error estimator}[c][c]{\small error estimator}
\psfrag{O(47)}[c][c]{\tiny $\mathcal{O}(N^{-4/7})$}
\psfrag{O(52)}[c][c]{\tiny $\mathcal{O}(N^{-5/2})$}
\psfrag{th=0.5, vth=0}[c][c]{\small  $\theta=0.5, \vartheta=0$}
\psfrag{th=0.5, vth=0.1}[c][c]{\small  $\theta=0.5, \vartheta=0.1$}
\psfrag{th=0.5, vth=1}[c][c]{\small  $\theta=0.5, \vartheta=1$}
\psfrag{ th=1, vth=0}[c][c]{\tiny $\theta=1, \vartheta=0$}
\psfrag{ th=0.5, vth=0}[c][c]{\tiny  $\theta=0.5, \vartheta=0$}
\psfrag{ th=0.5, vth=0.1}[c][c]{\tiny $\theta=0.5, \vartheta=0.1$}
\psfrag{ th=0.5, vth=1}[c][c]{\tiny $\theta=0.5, \vartheta=1$}
	\centering
	\includegraphics[width=.475\textwidth,clip=true]{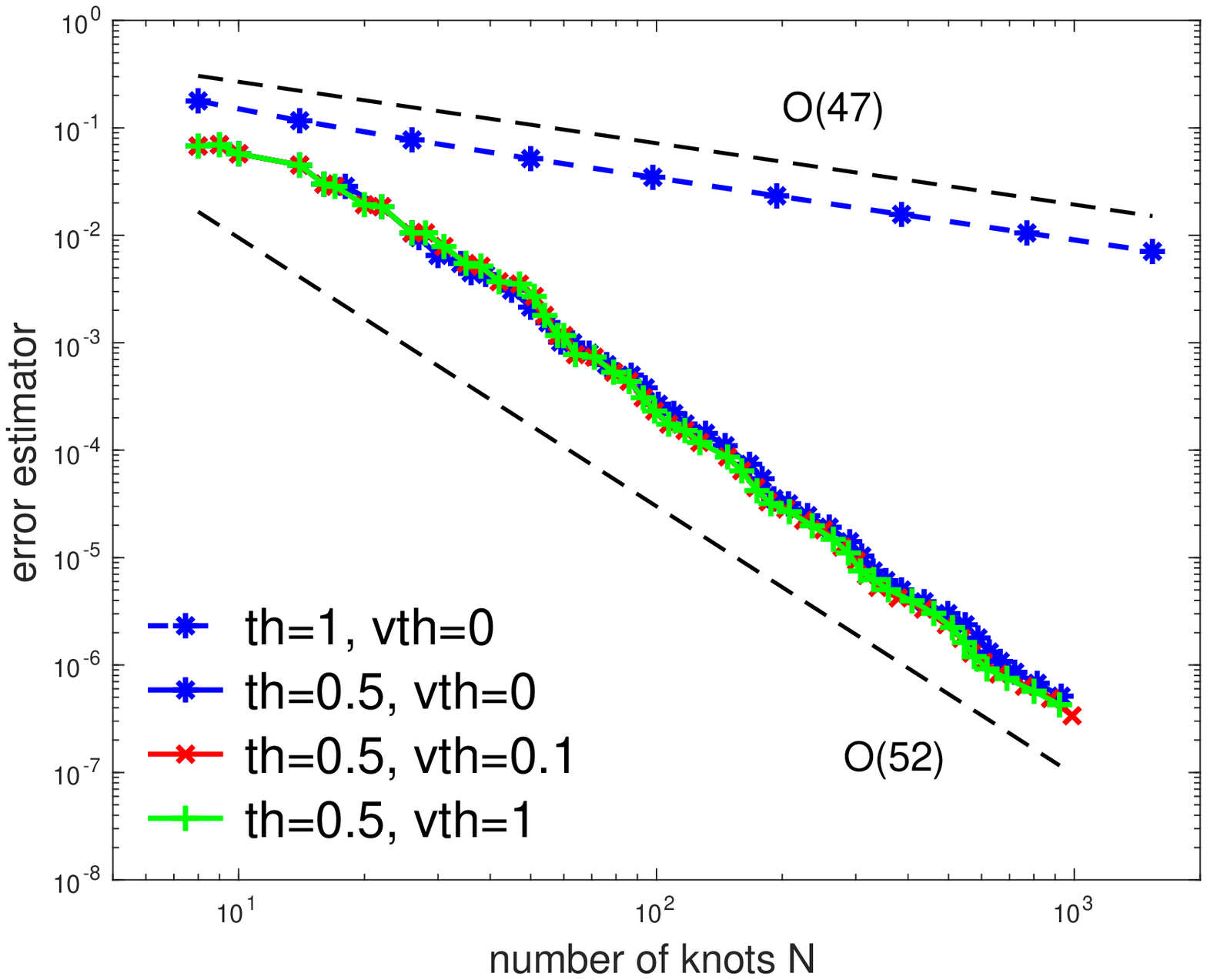}\quad
	\includegraphics[width=.475\textwidth,clip=true]{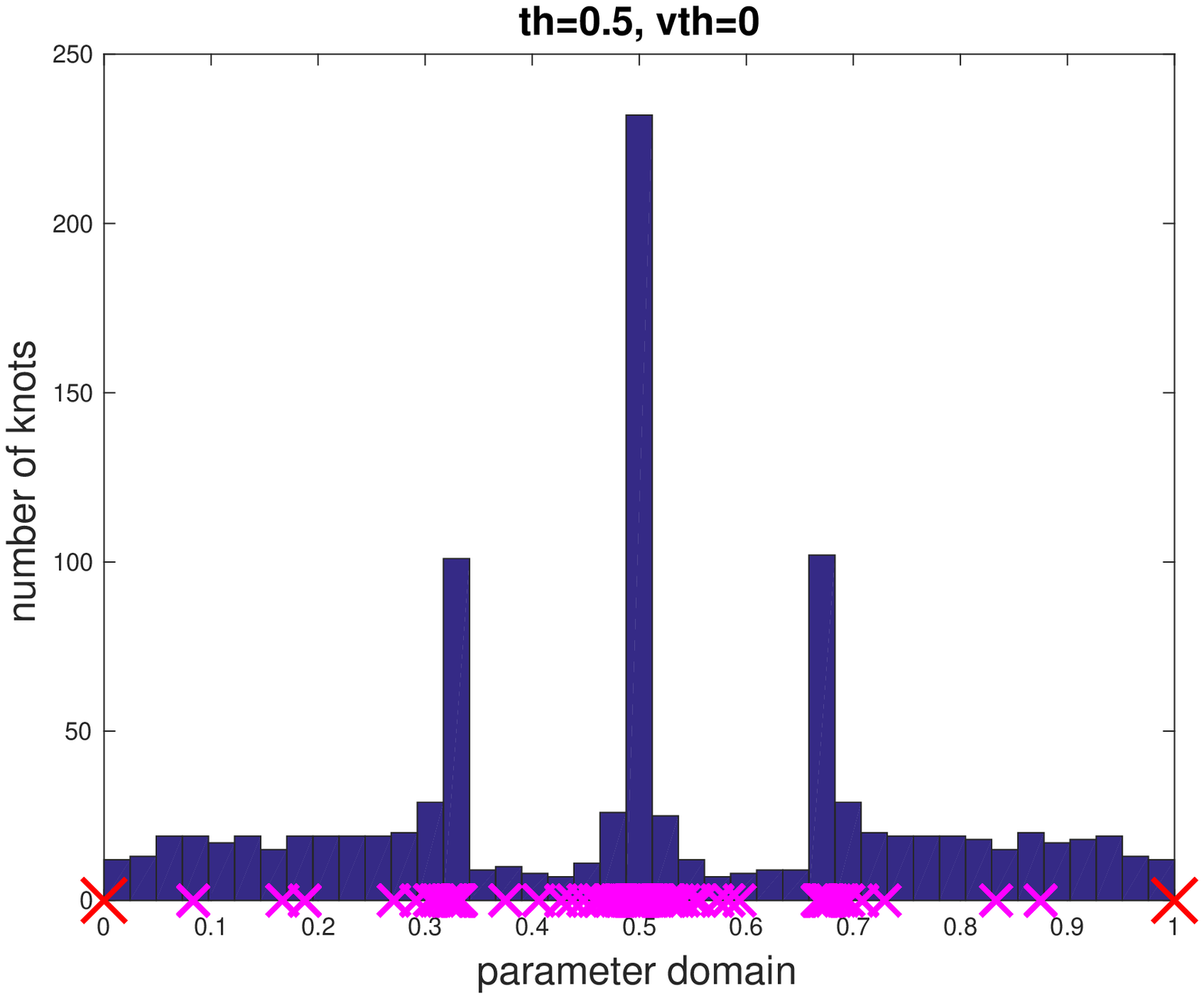}%
	\\
	\hspace{2mm}
	
	\includegraphics[width=.475\textwidth,clip=true]{figures/hist_hyper_1geo_pacmanp_2vartheta_0.1}\quad
	\includegraphics[width=.475\textwidth,clip=true]{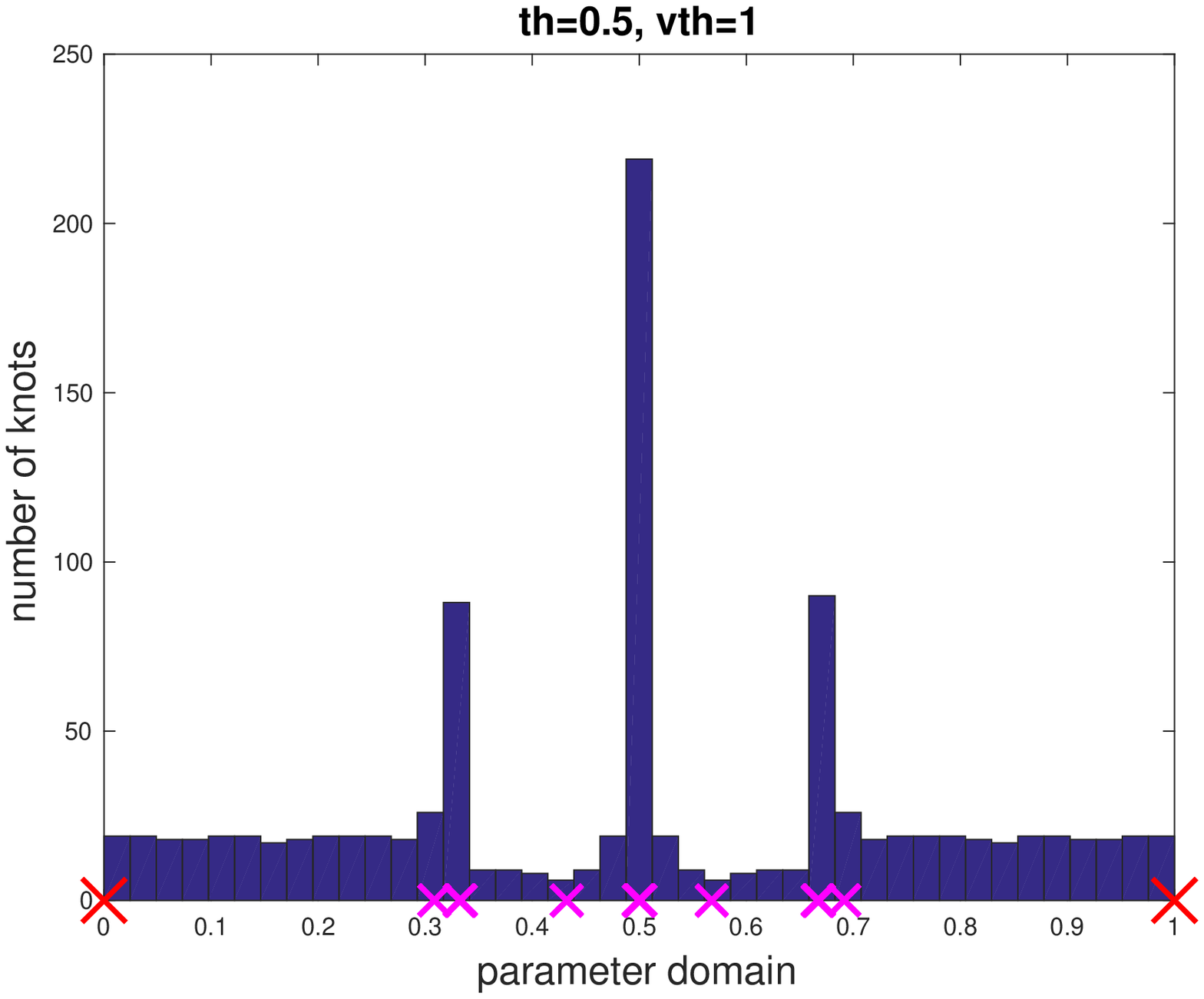}%
	\caption{Hyper-singular integral equation on pacman (Section~\ref{sec:hyper pacman}). Convergence plot of the error estimators $\eta_\ell$ and histograms of the knots $\widehat\KK_\ell$ of the last refinement step. 
	Knots with multiplicity 3 are highlighted by a red cross, knots with multiplicity 2 by a smaller magenta cross.}
\end{figure}

\subsection{Weakly-singular integral equation on pacman}
\label{sec:weak pacman}
In Figure \ref{fig:weak pacman}, the corresponding error estimators $\eta_\ell$ are plotted.
Since the solution lacks regularity, uniform refinement leads to the suboptimal rate $\mathcal{O}(N^{-4/7})$, whereas adaptive refinement leads to the optimal rate $\mathcal{O}(N^{-3/2-p})=\mathcal{O}(N^{-7/2})$.
For $\vartheta=1$, the corresponding multiplicative constant is clearly larger than for $\vartheta\in\{0,0.1\}$. 
A possible explanation is that $\vartheta=1$ results in too few multiplicity increases. 
Indeed, the histograms in  Figure~\ref{fig:hyper pacman} of the knots $\widehat\KK_\ell$ from the last refinement step indicate that $\vartheta\in\{0,0.1\}$ leads to full multiplicity of the knots $1/3$ and $2/3$, which is exactly where the solution $\phi\circ\gamma$ on the parameter domain (depicted in Figure~\ref{fig:solutions}) has jumps. 
In contrast, the choice $\vartheta=1$ compensates the lacking regularity at these points by $h$-refinement; see also \cite[Section~3]{resigabem}. 
Again, $\vartheta\in\{0.1,1\}$ give a more accurate information on the regularity of the solution.

\begin{figure}
\label{fig:weak pacman}
\psfrag{number of knots N}[c][c]{\small number of knots $N$}
\psfrag{number of knots}[c][c]{\small number of knots}
\psfrag{parameter domain}[c][c]{\small parameter domain}
\psfrag{error estimator}[c][c]{\small error estimator}
\psfrag{O(47)}[c][c]{\tiny $\mathcal{O}(N^{-4/7})$}
\psfrag{O(72)}[c][c]{\tiny $\mathcal{O}(N^{-7/2})$}
\psfrag{th=0.5, vth=0}[c][c]{\small  $\theta=0.5, \vartheta=0$}
\psfrag{th=0.5, vth=0.1}[c][c]{\small  $\theta=0.5, \vartheta=0.1$}
\psfrag{th=0.5, vth=1}[c][c]{\small  $\theta=0.5, \vartheta=1$}
\psfrag{ th=1, vth=0}[c][c]{\tiny $\theta=1, \vartheta=0$}
\psfrag{ th=0.5, vth=0}[c][c]{\tiny  $\theta=0.5, \vartheta=0$}
\psfrag{ th=0.5, vth=0.1}[c][c]{\tiny $\theta=0.5, \vartheta=0.1$}
\psfrag{ th=0.5, vth=1}[c][c]{\tiny $\theta=0.5, \vartheta=1$}
	\centering
	\includegraphics[width=.475\textwidth,clip=true]{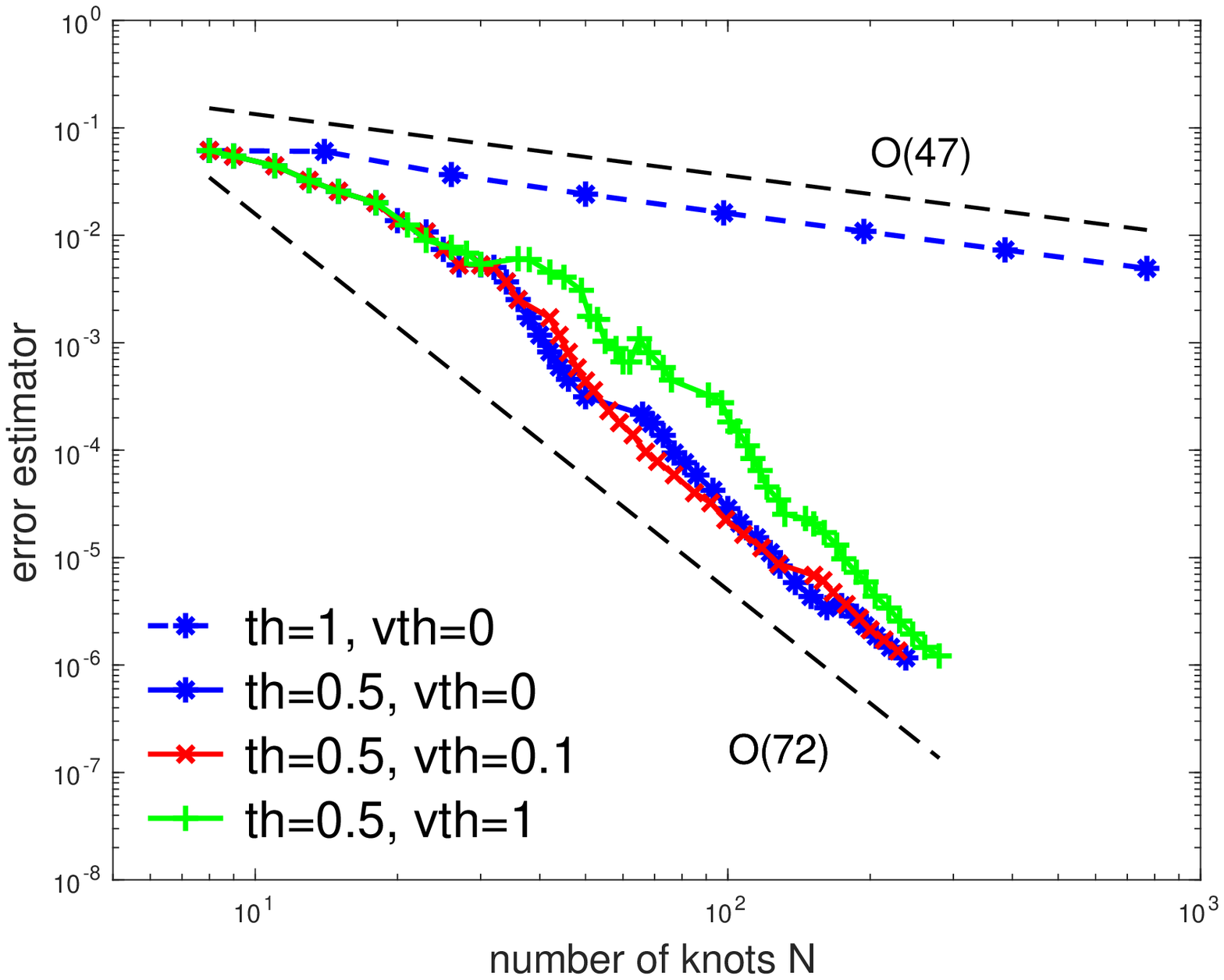}\quad
	\includegraphics[width=.475\textwidth,clip=true]{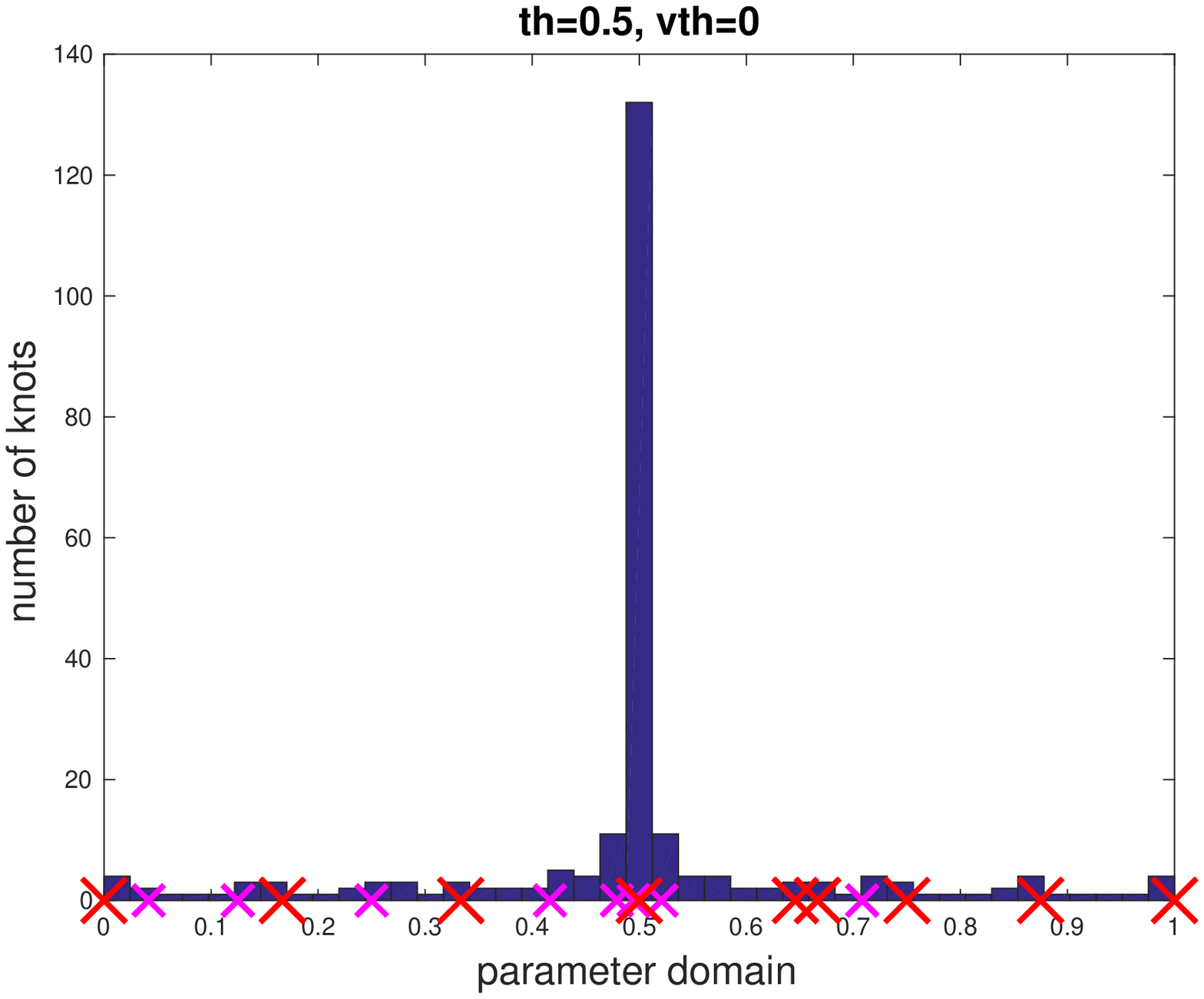}%
	\\
	\hspace{2mm}
	
	\includegraphics[width=.475\textwidth,clip=true]{figures/hist_hyper_0geo_pacmanp_2vartheta_0.1}\quad
	\includegraphics[width=.475\textwidth,clip=true]{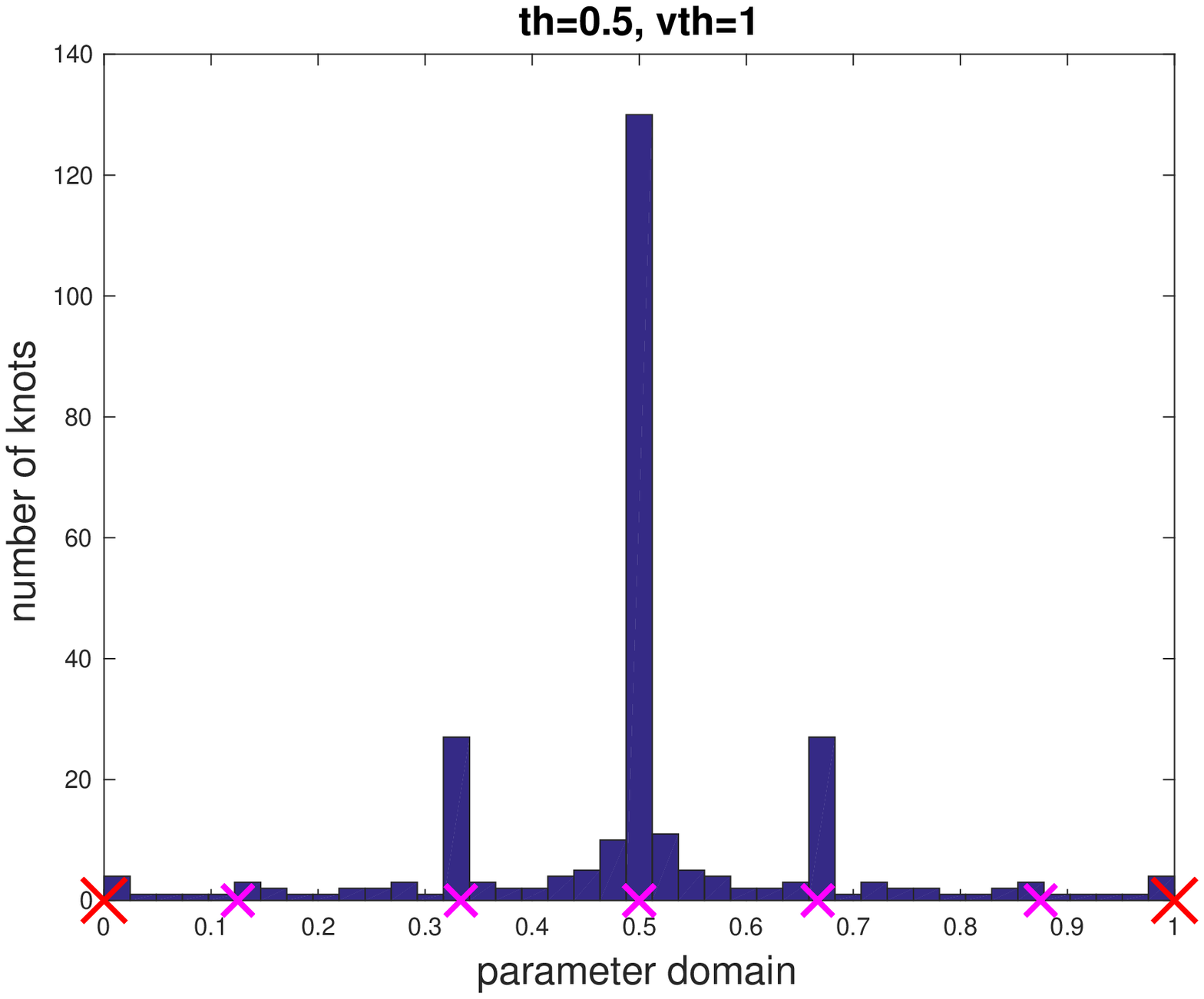}%
	\caption{Weakly-singular integral equation on pacman (Section~\ref{sec:weak pacman}). Convergence plot of the error estimators $\eta_\ell$ and histograms of the knots $\widehat\KK_\ell$ of the last refinement step. 
	Knots with multiplicity 3 are highlighted by a red cross, knots with multiplicity 2 by a smaller magenta cross.}
\end{figure}

\subsection{Hyper-singular integral equation on heart}
\label{sec:hyper heart}
In Figure~\ref{fig:hyper heart}, the corresponding error estimators $\eta_\ell$ are plotted.
Since the Neumann data, which have to be resolved, lack regularity, uniform refinement leads to the suboptimal rate $\mathcal{O}(N^{-2/3})$, whereas adaptive refinement leads to the optimal rate $\mathcal{O}(N^{-1/2-p})=\mathcal{O}(N^{-5/2})$.
While the estimator curves look very similar, the choices $\vartheta\in\{0.1,1\}$ (allowing for knot multiplicity decrease) additionally give accurate information on the regularity of the solution; see the histograms in Figure~\ref{fig:hyper heart}. 
Note that the (periodic extension of the) exact solution $u\circ\gamma$ on the parameter domain (depicted in Figure~\ref{fig:solutions}) is only $C^0$ at $0$ resp. $1$,  $1/6$, and $5/6$.

\begin{figure}
\label{fig:hyper heart}
\psfrag{number of knots N}[c][c]{\small number of knots $N$}
\psfrag{number of knots}[c][c]{\small number of knots}
\psfrag{parameter domain}[c][c]{\small parameter domain}
\psfrag{error estimator}[c][c]{\small error estimator}
\psfrag{O(23)}[c][c]{\tiny $\mathcal{O}(N^{-2/3})$}
\psfrag{O(52)}[c][c]{\tiny $\mathcal{O}(N^{-5/2})$}
\psfrag{th=0.5, vth=0}[c][c]{\small  $\theta=0.5, \vartheta=0$}
\psfrag{th=0.5, vth=0.1}[c][c]{\small  $\theta=0.5, \vartheta=0.1$}
\psfrag{th=0.5, vth=1}[c][c]{\small  $\theta=0.5, \vartheta=1$}
\psfrag{ th=1, vth=0}[c][c]{\tiny $\theta=1, \vartheta=0$}
\psfrag{ th=0.5, vth=0}[c][c]{\tiny  $\theta=0.5, \vartheta=0$}
\psfrag{ th=0.5, vth=0.1}[c][c]{\tiny $\theta=0.5, \vartheta=0.1$}
\psfrag{ th=0.5, vth=1}[c][c]{\tiny $\theta=0.5, \vartheta=1$}
	\centering
	\includegraphics[width=.475\textwidth,clip=true]{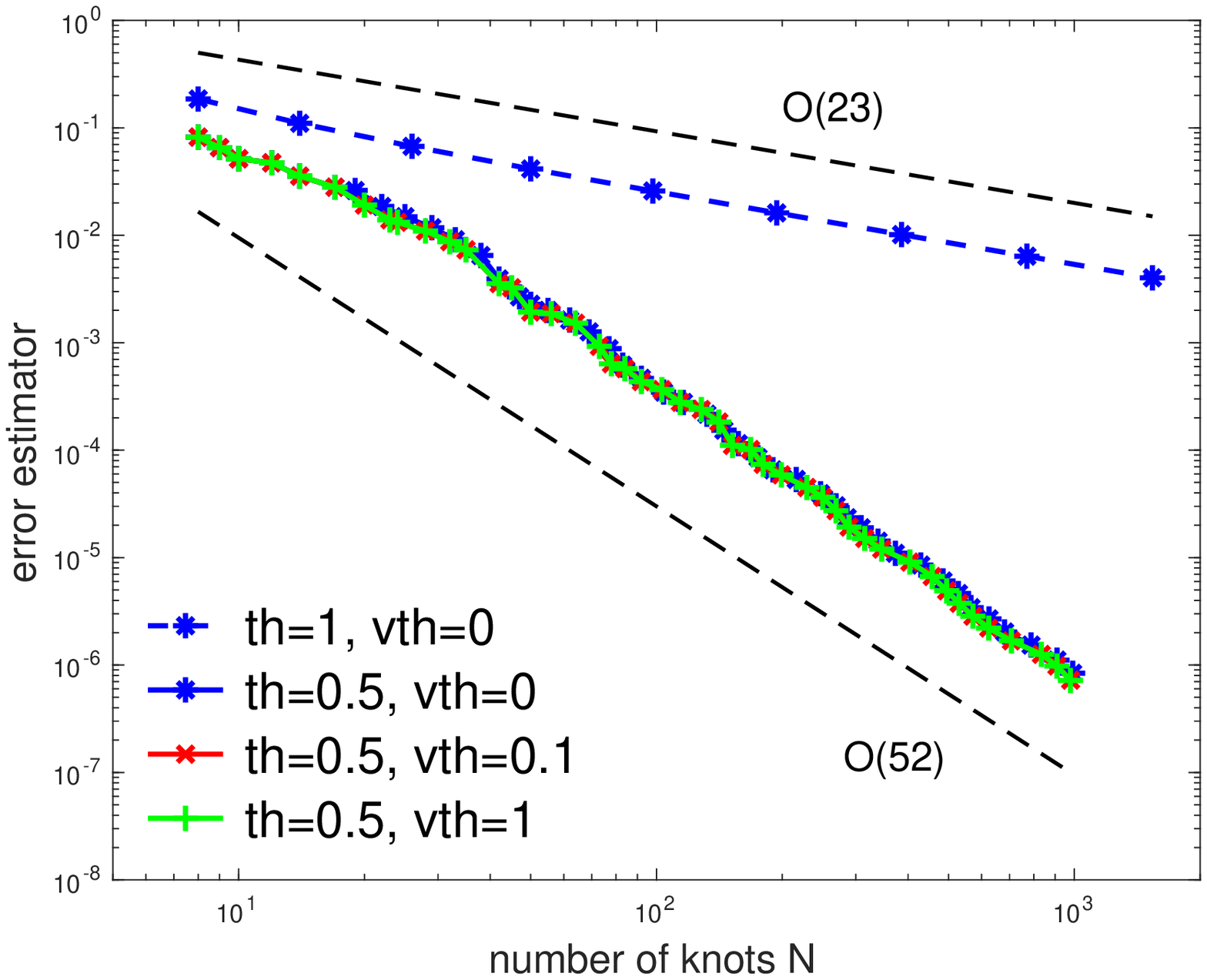}\quad
	\includegraphics[width=.475\textwidth,clip=true]{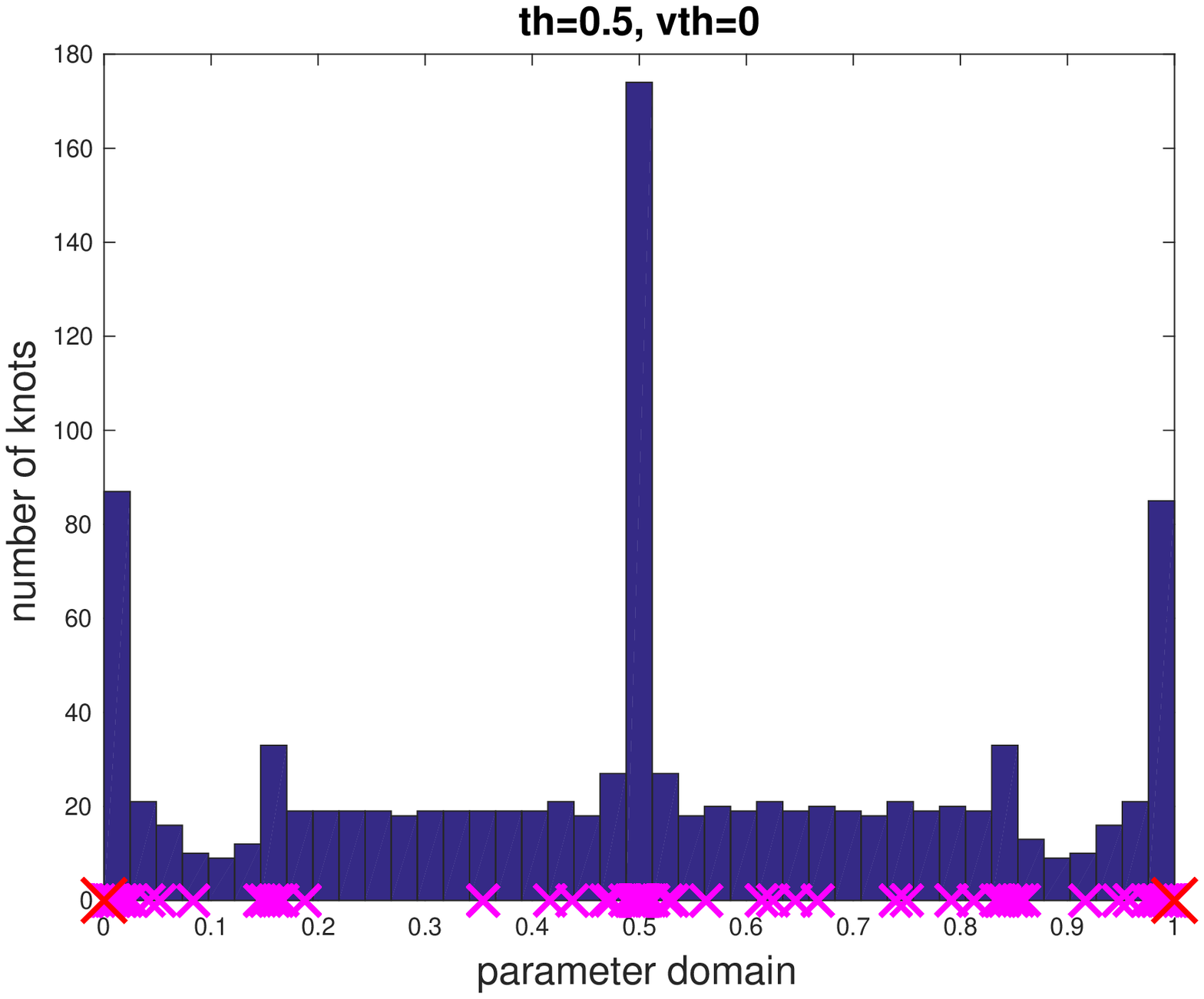}%
	\\
	\hspace{2mm}
	
	\includegraphics[width=.475\textwidth,clip=true]{figures/hist_hyper_1geo_heartp_2vartheta_0.1}\quad
	\includegraphics[width=.475\textwidth,clip=true]{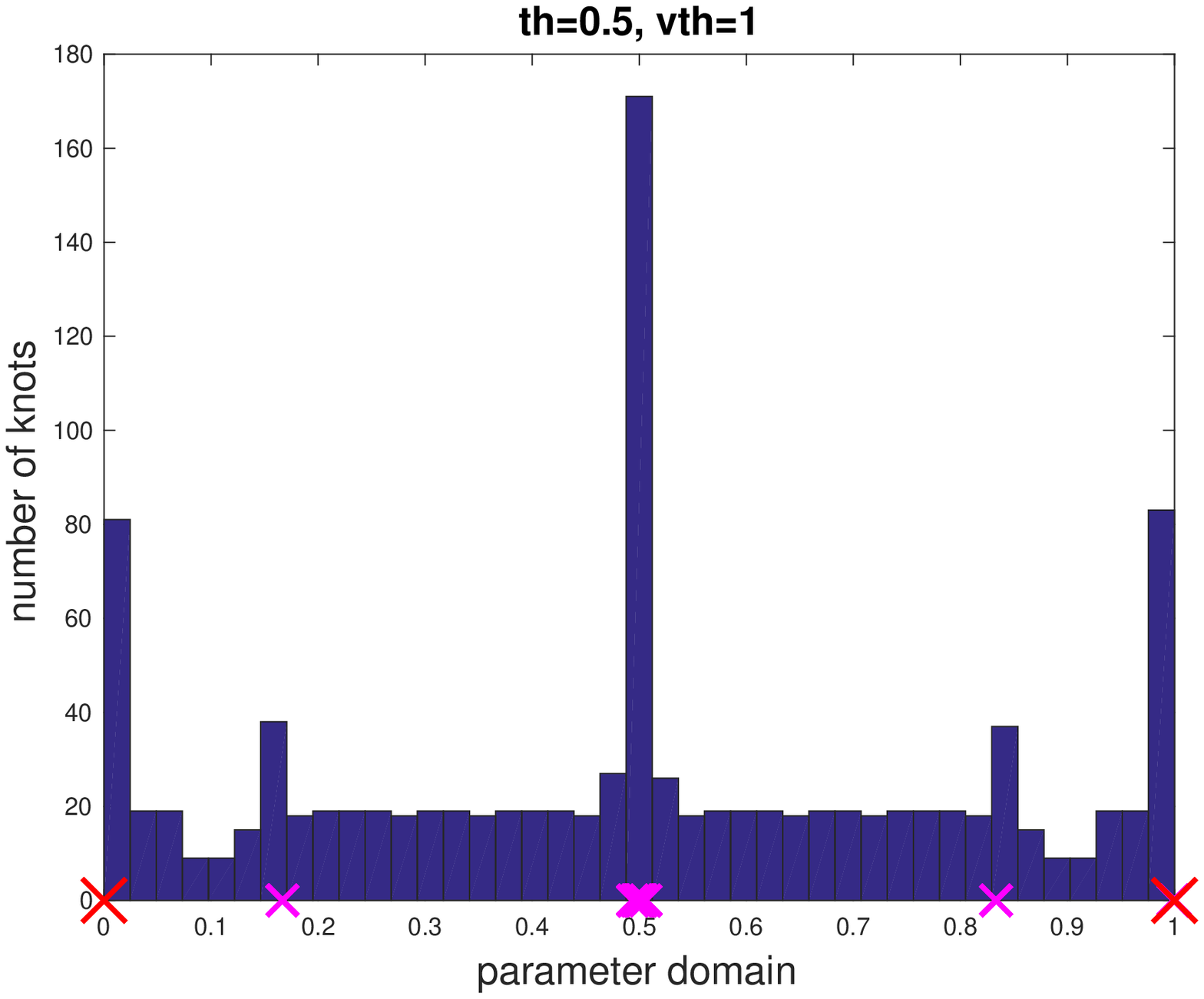}%
	\caption{Hyper-singular integral equation on heart (Section~\ref{sec:hyper heart}). Convergence plot of the error estimators $\eta_\ell$ and histograms of the knots $\widehat\KK_\ell$ of the last refinement step. 
	Knots with multiplicity 3 are highlighted by a red cross, knots with multiplicity 2 by a smaller magenta cross.}
\end{figure}

\subsection{Weakly-singular integral equation on heart}
\label{sec:weak heart}
In Figure \ref{fig:weak pacman}, the corresponding error estimators $\eta_\ell$ are plotted.
Since the solution lacks regularity, uniform refinement leads to the suboptimal rate $\mathcal{O}(N^{-2/3})$, whereas adaptive refinement leads to the optimal rate $\mathcal{O}(N^{-3/2-p})=\mathcal{O}(N^{-7/2})$.
Figure~\ref{fig:hyper pacman} further provides histograms of the knots $\widehat\KK_\ell$ from the last refinement step.
Overall, we observe a similar behavior as in Section~\ref{sec:weak pacman}.
Note that the (periodic extension of the) exact solution $\phi\circ\gamma$ on the parameter domain (depicted in Figure~\ref{fig:solutions}) is only $C^0$ at $0$ resp. $1$,  $1/6$, and $5/6$.

\begin{figure}
\label{fig:weak heart}
\psfrag{number of knots N}[c][c]{\small number of knots $N$}
\psfrag{number of knots}[c][c]{\small number of knots}
\psfrag{parameter domain}[c][c]{\small parameter domain}
\psfrag{error estimator}[c][c]{\small error estimator}
\psfrag{O(23)}[c][c]{\tiny $\mathcal{O}(N^{-2/3})$}
\psfrag{O(72)}[c][c]{\tiny $\mathcal{O}(N^{-7/2})$}
\psfrag{th=0.5, vth=0}[c][c]{\small  $\theta=0.5, \vartheta=0$}
\psfrag{th=0.5, vth=0.1}[c][c]{\small  $\theta=0.5, \vartheta=0.1$}
\psfrag{th=0.5, vth=1}[c][c]{\small  $\theta=0.5, \vartheta=1$}
\psfrag{ th=1, vth=0}[c][c]{\tiny $\theta=1, \vartheta=0$}
\psfrag{ th=0.5, vth=0}[c][c]{\tiny  $\theta=0.5, \vartheta=0$}
\psfrag{ th=0.5, vth=0.1}[c][c]{\tiny $\theta=0.5, \vartheta=0.1$}
\psfrag{ th=0.5, vth=1}[c][c]{\tiny $\theta=0.5, \vartheta=1$}
	\centering
	\includegraphics[width=.475\textwidth,clip=true]{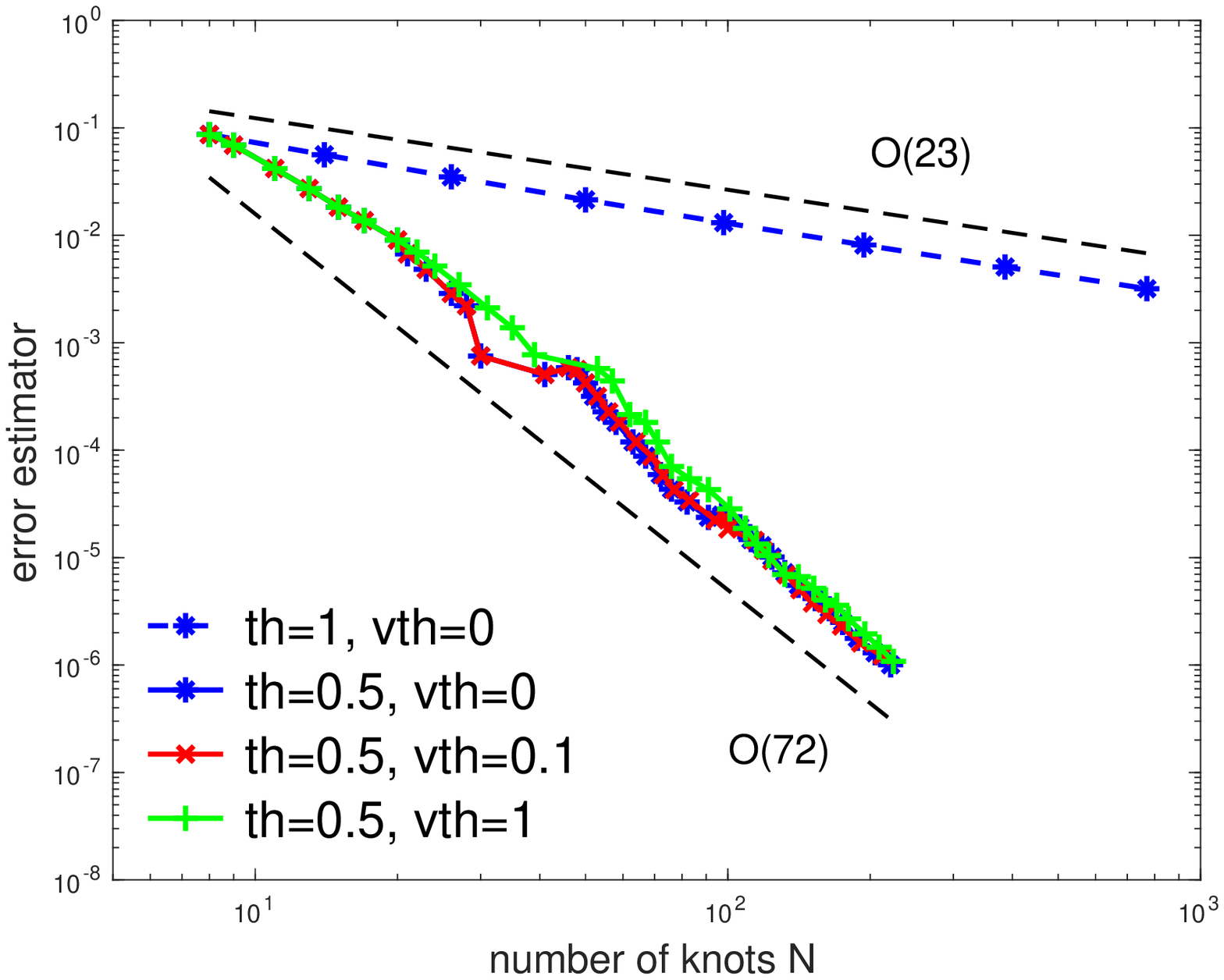}\quad
	\includegraphics[width=.475\textwidth,clip=true]{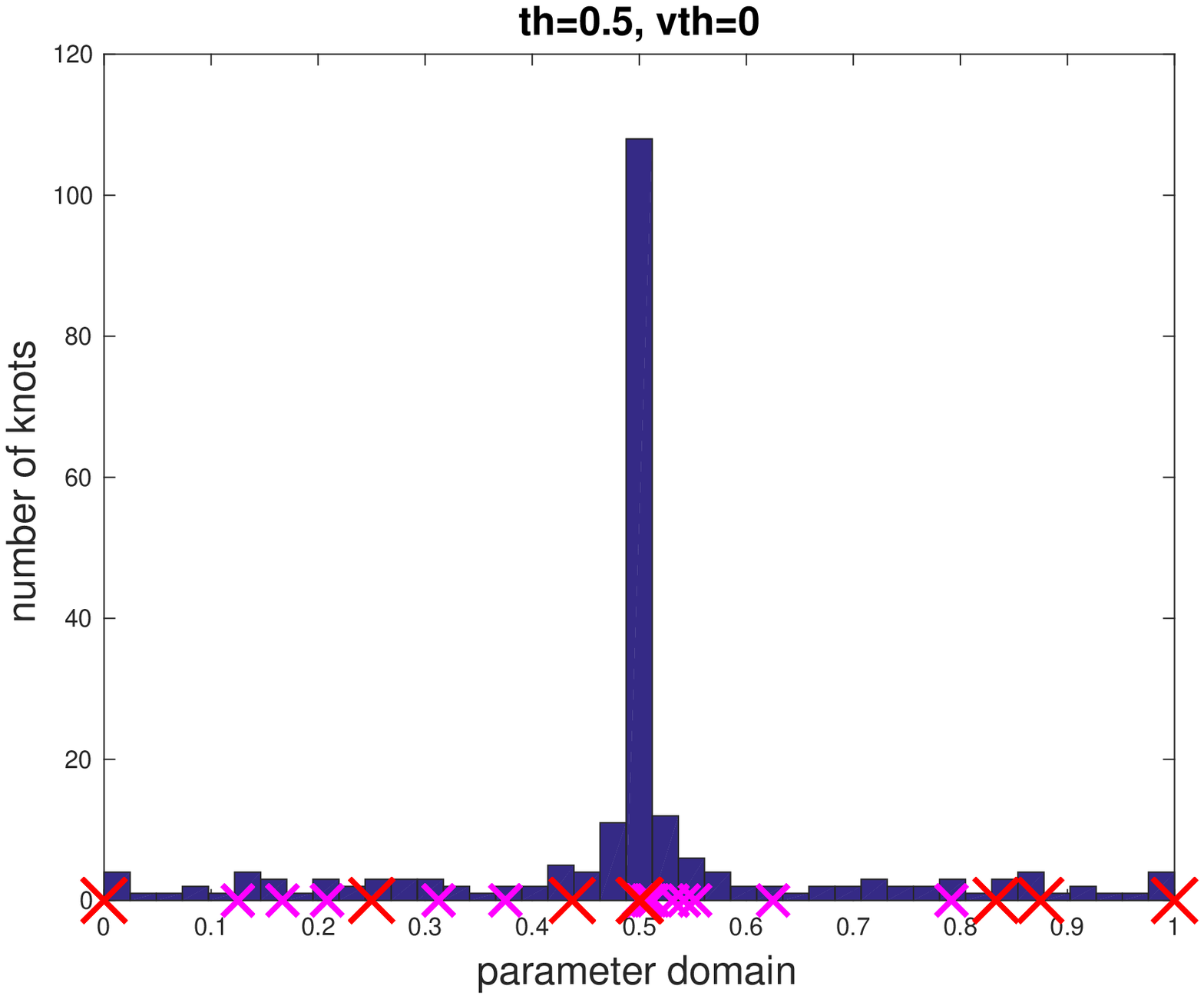}%
	\\
	\hspace{2mm}
	
	\includegraphics[width=.475\textwidth,clip=true]{figures/hist_hyper_0geo_heartp_2vartheta_0.1}\quad
	\includegraphics[width=.475\textwidth,clip=true]{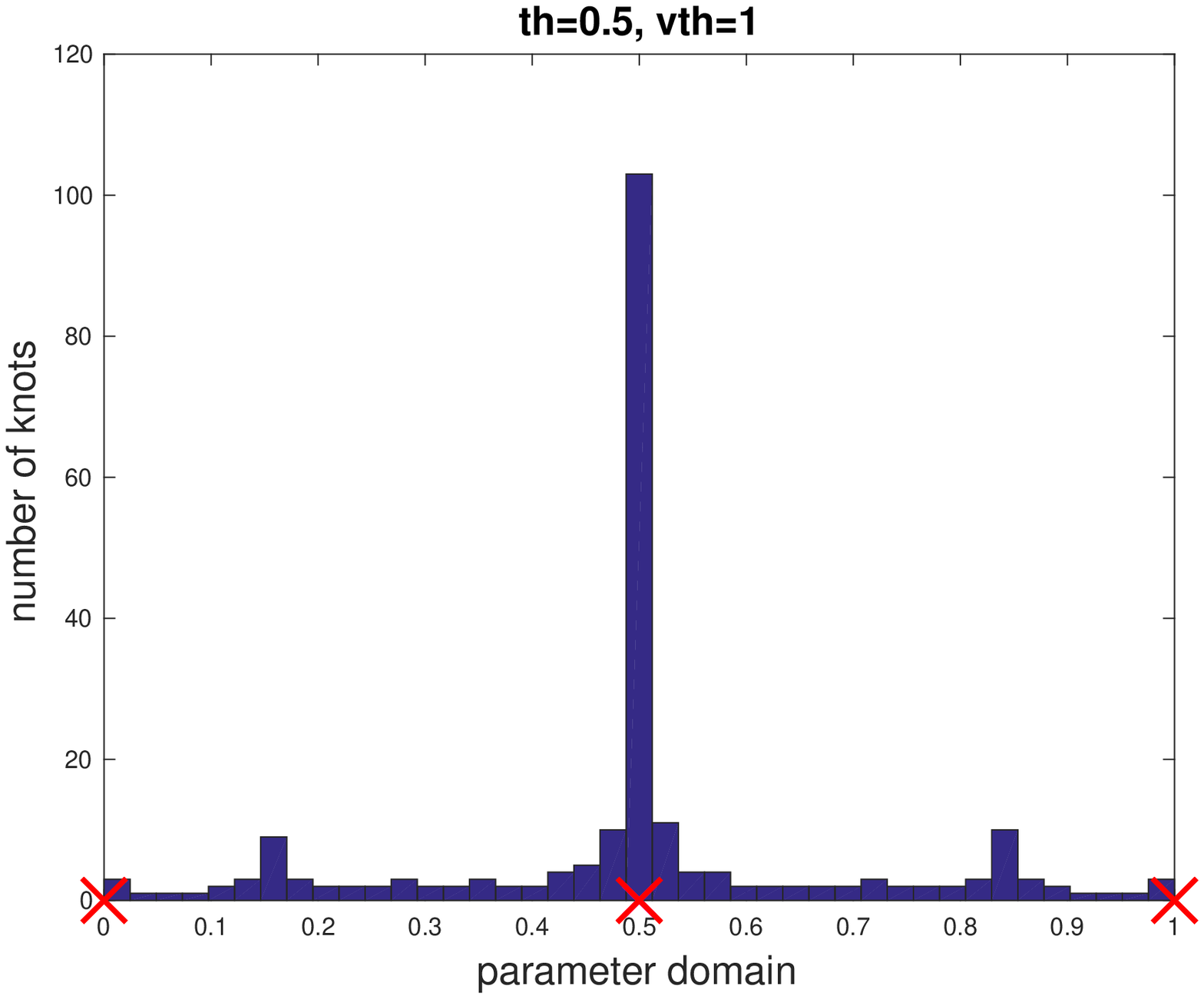}%
	\caption{Weakly-singular integral equation on heart (Section~\ref{sec:weak heart}). Convergence plot of the error estimators $\eta_\ell$ and histograms of the knots $\widehat\KK_\ell$ of the last refinement step. 
	Knots with multiplicity 3 are highlighted by a red cross, knots with multiplicity 2 by a smaller magenta cross.}
\end{figure}


\appendix
\section{Weakly-singular integral equation}
\label{sec:weaksing}%

In this appendix, we consider the weakly-singular integral equation
\begin{align}\label{eq:weak strong}
\mathfrak{V}\phi=(1/2+\mathfrak{K})u\quad\text{with given } u\in H^{1/2}(\Gamma),
\end{align}
which is equivalent to the Laplace--Dirichlet problem
\begin{align}
 -\Delta P = 0 \text{ in } \Omega
 \quad \text{subject to Dirichlet boundary conditions} \quad
 P|_\Gamma= u \text{ on } \Gamma=\partial\Omega,
\end{align}
where $\phi =  \partial P/\partial \nu$ is the normal derivative of the sought potential $P$.

\subsection{Functional analytic setting}
The weakly-singular integral equation~\eqref{eq:weak strong} employs the  single-layer operator $\mathfrak{V}$ as well as the double-layer operator $\mathfrak{K}'$.
These have  the  boundary integral representations
\begin{align}
\mathfrak{V}\phi(x)=\int_\Gamma  \phi(y) G(x,y)\,dy \quad \text{and}\quad
\mathfrak{K}u(x)=\int_\Gamma u(y)\frac{\partial_y}{\partial\nu(y)}G(x,y)\,dy 
\end{align}
for smooth densities $\phi,u:\Gamma\to\R$.

For $0\le \sigma\le1$, the single-layer operator 
$\VV:H^{\sigma-1}(\Gamma)\to H^{\sigma}(\Gamma)$ and the double-layer operator $\mathfrak{K}:H^{\sigma}(\Gamma)\to H^{\sigma}(\Gamma)$ are well-defined, linear, and continuous. 

For $\sigma=1/2$,  
$\VV:H^{-1/2}(\Gamma)\to H^{1/2}(\Gamma)$ is symmetric and 
elliptic under the assumption that $\diam(\Omega)<1$, which can always be achieved by scaling $\Omega$.
In particular,
\begin{align}
\edualV{\phi}{\psi}:=\dual{\VV \phi}{\psi}_\Gamma
\end{align}
 defines an equivalent scalar product on 
$H^{-1/2}(\Gamma)$ with corresponding norm $\enormV{\cdot}$.
With this notation, the strong form~\eqref{eq:weak strong} with data $u\in H^{1/2}(\Gamma)$ is equivalently stated by
\begin{align}
\label{eq:weakform2}
 \edualV{\phi}{\psi} = \dual{(1/2+\mathfrak{K})u}{\psi}_\Gamma
 \quad\text{for all }\psi\in H^{-1/2}(\Gamma).
\end{align}
Therefore, the Lax-Milgram lemma applies and proves that \eqref{eq:weakform2} resp.\  \eqref{eq:weak strong} admits a unique solution $\phi\in H^{-1/2}(\Gamma)$.
Details are found, e.g.,  in  \cite{hw,mclean,ss,s}.

\subsection{IGABEM discretization}
As refinement algorithm for the boundary meshes, we use Algorithm~\ref{alg:refinement}, where the multiplicity of the knots in step (iii) is now increased up to $p+1$ (instead of $p$), allowing for discontinuities at the nodes.
The set $\K$ now denotes the set of all possible knot vectors that can be generated with this modified refinement algorithm starting from an initial knot vector $\KK_0$ as in Section~\ref{section:boundary:discrete}, where each knot in $\KK_0$ might have multiplicity up to $p+1$. 
Further, we do no longer require the restriction $w_{0,1-p}=w_{0,N_0-p}$ for the weights.
For $\KK_\coarse\in\K$, we define the corresponding ansatz space as 
\begin{align}\label{eq:weak space}
\XX_\coarse:={\mathcal{S}}^p(\KK_\coarse,\mathcal{W}_\coarse)={\rm span}\set{R_{\coarse,i,p}}{i=1-p,\dots,N_\coarse-p}
\end{align}
and the Galerkin approximation $\Phi_\coarse\in\XX_\coarse$ of $\phi$ via 
\begin{align}
 \edualV{\Phi_\coarse}{\Psi_\coarse} = \dual{g_\coarse}{\Psi_\coarse}_\Gamma 
\,\text{ for all }\Psi_\coarse\in \XX_\coarse,\quad\text{where }g_\coarse:=(1/2+\mathfrak{K})u_\coarse. 
\end{align}
Here, we  define $u_\coarse:=P_\coarse\phi$, where $P_\coarse$ is either the  identity or the Scott--Zhang operator of Section~\ref{sec:scott zhang} onto the space of (transformed) continuous piecewise polynomials $\PP^p(\QQ_\coarse)\cap C^0(\Gamma)$.

In order to employ the weighted-residual error estimator (plus oscillations)
\begin{align}
\eta_\coarse(z)^2:=\norm{h_\coarse^{1/2}\partial_\Gamma(g_\coarse-\mathfrak{V}\Phi_\coarse)}{L^2(\pi_\coarse(z))}^2+\norm{ h_\coarse^{1/2}\partial_\Gamma(u-u_\coarse)}{L^2(\pi_\coarse(z))}^2\quad\text{for all }z\in\NN_\coarse, 
\end{align}
 we require the additional regularity $u\in H^1(\Gamma)$. 
Moreover, we define 
\begin{align}
\mu_\coarse(z):=\norm{ h_\coarse^{1/2}(1-I_{\coarse\ominus 1})\Phi_\coarse}{L^2(\pi_\coarse^{2p+1}(z))}\text{ for all }z\in\NN_\coarse,
\end{align}
where $I_\coarse$ is  now the Scott--Zhang operator onto $\XX_\coarse$ defined in \cite[Section 3.1.2]{overview}. 

With these definitions, Algorithm~\ref{the algorithm} is also well-defined for the weakly-singular case.
As already mentioned in Remark~\ref{rem:feature} (b), the choice $\vartheta=0$ and $\MM_\ell^-=\emptyset$ leads to no multiplicity decreases and then the adaptive algorithm coincides with the one from \cite{resigabem} if $u_\coarse:=u$. 
For the latter, linear convergence at optimal rate has already been proved in our earlier work \cite{optigabem}.
Theorem~\ref{thm:main} holds accordingly in the weakly-singular case and thus generalizes \cite{optigabem}.    
We will briefly sketch the proof in the remainder of this appendix.

\subsection{Reliability and efficiency}
\label{subsection:reliable weak}
Reliability, i.e., $\norm{\phi-\Phi_\coarse}{H^{-1/2}(\Gamma)}\lesssim\eta_\coarse$ is already stated in \cite[Theorem~4.4]{resigabem}. 
Efficiency, i.e., $\eta_\coarse\lesssim \norm{h_\coarse^{1/2}(\phi-\Phi_\coarse)}{L^2(\Gamma)}+
\norm{h_\coarse^{-1/2}(g-g_\coarse)}{L^2(\Gamma)}$ follows as in \cite[Corollary~3.3]{invest}, which proves the assertion for standard BEM.

\subsection{Linear and optimal convergence}
Linear convergence \eqref{eq:R-linear} at optimal rate~\eqref{eq:optimal} (with similarly defined approximability constant $\norm{\phi}{\A_s}$) follows again from the axioms of adaptivity.
Stability (E1) follows exactly as the corresponding version \cite[Lemma~5.1]{optigabem}. 
The main argument is the inverse estimate $\norm{h_\coarse^{1/2}\partial_\Gamma (\VV\Psi_\coarse)}{L^2(\Gamma)}\lesssim\norm{\Psi_\coarse}{H^{-1/2}(\Gamma)}$ of \cite[Proposition~4.1]{optigabem} for (transformed) rational splines $\Psi_\coarse\in\XX_\coarse$.
Reduction (E2) follows as in \cite[Lemma~4.4]{optigabem}.
It is proved via the same inverse estimate  together with the contraction property \eqref{eq:h tilde ctr} of $\widetilde h_{\pi_\coarse(z)}$.
Details are also found in \cite[Section~5.8.4 resp.\ Section~5.8.5]{gantner17}.
We have already proved discrete reliability in \cite[Lemma 5.2]{optigabem}; see also \cite[Section~5.8.7]{gantner17} for details.
The proof of quasi-orthogonality (E4) is almost identical as for the hyper-singular case in Section~\ref{subsec:hypsing orthogonality}. 
The son estimate is verified as in Section~\ref{sec:hypsing son} and the closure estimate (R2) as well as the overlay property (R3) are already found in \cite[Proposition~2.2]{optigabem}. 

\subsection{Approximability constants}
Similarly as in \eqref{eq:classes}, we have that $\norm{\phi}{\A_s^1}\lesssim\norm{\phi}{\A_s}\le 
\min\{\norm{\phi}{\A_s^1},\norm{\phi}{\A_s^{p+1}}\}\lesssim\norm{\phi}{\A_s^1}$, where the approximability constants are defined analogously as in Section~\ref{sec:main}.
The proof follows along the lines of Section~\ref{sec:classes}.

\section{Indirect BEM}
\subsection{Hyper-singular case}
Theorem~\ref{thm:main} remains valid if, instead of \eqref{eq:hyper strong}, one considers the indirect integral formulation
\begin{align}
\mathfrak{W} u=\phi\quad\text{with given }\phi\in H_0^{-1/2}(\Gamma),
\end{align}
where the function $\phi$ is again approximated via $\phi_\coarse$ as in Section~\ref{section:igabem}. 
Indeed, the proof becomes even simpler due to the absence of the operator $\mathfrak{K}'$. 

\subsection{Weakly-singular case}
Similarly, one can consider 
\begin{align}
\mathfrak{V}\phi=u \quad\text{with given } u\in H^{1/2}(\Gamma)
\end{align}
instead of \eqref{eq:weak strong}, where $u$ is approximated as in Appendix~\ref{sec:weaksing}, and the results of Appendix~\ref{sec:weaksing} remain valid.
Again, the proof even simplifies due to the absence of the operator $\mathfrak{K}$.

\section{Slit problems}
In contrast to before, let now $\Gamma\subsetneqq \partial\Omega$ be a connected proper subset of the boundary $\partial\Omega$ with parametrization $\gamma:[a,b]\to\Gamma$. 
Let $E_0(\cdot)$ denote the extension of a function on $\Gamma$ by zero onto the whole boundary.
\subsection{Hyper-singular case}
We consider the slit problem
\begin{align}
(\mathfrak{W}E_0 u)|_\Gamma= \phi\quad\text{with given } \phi\in H^{-1/2}(\Gamma).
\end{align}
Here, $H^{-1/2}(\Gamma)$ denotes the dual space of $\H^{1/2}(\Gamma)=[L^2(\Gamma),\H^1(\Gamma)]_{1/2}$, where $\H^1(\Gamma)$ is the set of all $H^1(\Gamma)$ functions with vanishing trace on the relative boundary $\partial\Gamma$.
The operator $(\mathfrak{W}E_0(\cdot))|_\Gamma:\H^{1/2}(\Gamma)\to H^{-1/2}(\Gamma)$ is linear, continuous, and elliptic. 
Let $\K$ now denote the set of all knot vectors that can be generated via Algorithm~\ref{alg:refinement} from an initial knot vector $\KK_0$ on $\Gamma$. 
For $\KK_\coarse\in\K$ and corresponding weights $\mathcal{W}_\coarse$ as in Section~\ref{section:igabem} without the restriction $w_{\coarse,1-p}=w_{\coarse,N_\coarse-p}$, we define the ansatz space 
\begin{align}
\XX_\coarse:=\set{V}{V\circ\gamma \in\widehat{\mathcal{S}}(\widehat\KK_\coarse,\mathcal{W}_\coarse)\wedge 0=V(\gamma(a+))=V(\gamma(b-))}\subset \H^1(\Gamma).
\end{align}
Moreover, we define the corresponding Scott--Zhang operator $J_\coarse:L^2(\Gamma)\to\XX_\coarse$ similarly as in Section~\ref{sec:scott zhang} via 
\begin{align}
J_\coarse v:=\sum_{i=2-p}^{N_\coarse-p-1} \Big(\int_a^b  \widehat R_{\coarse,i,p}^*(v\circ\gamma) \,dt \Big)R_{\coarse,i,p}.
\end{align}
Replacing the spaces $H^\sigma(\Gamma)$ and $H^1(\Gamma)$ by $\H^\sigma(\Gamma)$ and $\H^1(\Gamma)$, Proposition~\ref{lem:Scott properties} is also satisfied in this case, where the proof additionally employs the  Friedrichs inequality.
If we approximate $\phi$ as in Section~\ref{section:igabem}, Theorem~\ref{thm:main} holds accordingly. 
The proof follows along the same lines. 

\subsection{Weakly-singular case}
We consider the slit problem
\begin{align}\label{eq:weak slit}
(\mathfrak{V}E_0 \phi)|_\Gamma= u\quad\text{with}\quad u\in H^{1/2}(\Gamma),
\end{align}
where $H^{1/2}(\Gamma)=[L^2(\Gamma),H^1(\Gamma)]_{1/2}$.
The dual space of the latter is denoted by $\H^{-1/2}(\Gamma)$.
The operator $(\mathfrak{V}E_0(\cdot))|_\Gamma:\H^{-1/2}(\Gamma)\to H^{1/2}(\Gamma)$ is linear, continuous, and elliptic provided that $\diam(\Omega)<1$. 
An adaptive IGABEM can be formulated as in Appendix~\ref{sec:weaksing}. 
Without multiplicity decrease and oscillation terms, it coincides with the algorithm from \cite{resigabem}, which converges linearly at optimal rate according to \cite{optigabem}. 
For the generalized version, one can prove the same results as in Appendix~\ref{sec:weaksing}. 



\subsection*{Acknowledgement.}
The authors are supported by the Austrian Science Fund (FWF) through the research projects \textit{Optimal isogeometric boundary element method} (grant P29096) and \textit{Optimal adaptivity for BEM and FEM-BEM coupling} (grant P27005), the doctoral school \textit{Dissipation and dispersion in nonlinear PDEs} (grant W1245), and the special research program \textit{Taming complexity in PDE systems} (grant SFB F65).
\bibliographystyle{alpha}
\bibliography{literature}



\end{document}